\documentclass[11pt]{article}
\usepackage[intlimits]{amsmath}
\usepackage{amsfonts, amssymb, amsthm}
\usepackage{esint}
\usepackage{epsfig}
\usepackage{color}
\usepackage{threeparttable} 

\renewcommand{\O}{\Omega}
\newcommand{\po}{\partial\O}
\renewcommand{\o}{\omega}
\renewcommand{\a}{\alpha}
\renewcommand{\b}{\beta}
\renewcommand{\div}{\,\mathrm{div}\,}

\newcommand{\G}{\Gamma}
\newcommand{\g}{\gamma}

\renewcommand{\d}{\delta}
\newcommand{\D}{\Delta}
\newcommand{\s}{\sigma}

\newtheorem{thm}{Theorem}[section]
\newtheorem{prop}[thm]{Proposition}
\newtheorem{cor}[thm]{Corollary}
\newtheorem{lemma}[thm]{Lemma}
\newtheorem{defn}[thm]{Definition}

\newtheorem{preremark}[thm]{Remark}

\numberwithin{equation}{section}

\newcommand\res{\hbox{ {\vrule height .22cm}{\leaders\hrule\hskip.2cm} } }

\newcommand{\R}{\mathbb R}

\newcommand{\dist} {\mathrm{dist}}

\def\Xint#1{\mathchoice
                 {\XXint\displaystyle\textstyle{#1}}%
                 {\XXint\textstyle\scriptstyle{#1}}%
                 {\XXint\scriptstyle\scriptscriptstyle{#1}}%
                 {\XXint\scriptscriptstyle\scriptscriptstyle{#1}}%
                 \!\int}
\def\XXint#1#2#3{{\setbox0=\hbox{$#1{#2#3}{\int}$}
       \vcenter{\hbox{$#2#3$}}\kern-.5\wd0}}

\def\dashint{\Xint-}
\newcommand{\meanbar}[1]{%
\setbox0 = \hbox{$#1 \int$}
\hbox to 0pt{%
\thinspace
\hskip 0.1\wd0
\raise 0.5\ht0
\hbox{%
\lower 0.5\dp0
\hbox{\rule{0.8\wd0}{2\linethickness}}
}%
\hss
}%
}

    \newcounter{myfootertablecounter}

\begin{document}
\title{Harmonic Analysis on chord arc domains.}

\author{Emmanouil Milakis \and Jill Pipher \and Tatiana Toro}
\date{}
\maketitle

\begin{abstract}
In the present paper we study the solvability of the Dirichlet problem for second order divergence form elliptic operators with bounded measurable coefficients which are small perturbations of given operators in rough domains beyond the Lipschitz category. In our approach, the development of the theory of tent spaces on these domains is essential. 

\end{abstract}

\tableofcontents

\section{Introduction}
\renewcommand{\thesection}{\arabic{section}}

In this paper, we establish fine properties of the elliptic measure associated to the solvability of the Dirichlet problem for certain small perturbations of elliptic operators in chord arc domains.
The elliptic measure is that which arises naturally as the representing measure associated
to the solution of the Dirichlet problem for a second order elliptic operator with continuous boundary data.
The ``fine properties" of such measures are sharply described by the conditions defining
the Muckenhoupt weight classes, in which these measures are compared to other natural
measures, such as surface measure,  which live on the boundary of the domain. 

We will consider second order elliptic operators in divergence form, $L = \textrm{div}A\nabla$ ,
which are perturbations, in a sense to be made precise, of some given elliptic operators. The 
perturbation theory developed here for chord arc domains is the extension of that same theory for
Lipschitz domains: see \cite{d1}, \cite{esc1} \cite{feff}, \cite{fkp} for some prior literature.

A chord arc domain in $\R^n$ is a non-tangentially accessible (NTA)
domain whose boundary is rectifiable and whose surface measure is Ahlfors regular (i.e. the surface measure 
on boundary balls of radius $r$ grows like $r^{n-1}$) . We refer the
reader to the available literature, and specifically to \cite{JK} for the 
precise definition of NTA domains. NTA domains possess all
of the following properties: (i) a quantified standard relationship
between elliptic measure on the boundary of a domain and the Green's function for that
domain, (ii) the doubling property of elliptic measure, and (iii)  comparison principles
for non-negative solutions to elliptic divergence form equations. These properties are
consequences of the geometric definition of NTA domains and are stated
precisely in the next section. 

We briefly recall the Muckenhoupt weight classes. If $\mu$ and $\nu$ are mutually
absolutely continuous positive measures defined on
the boundary of a domain, $\partial\O$, then there exists a weight function $g$
such that $d\mu = g d\nu$. The measure $d\mu$ belongs to the weight class
$B_q(d\nu)$ if there exists a constant $C>0$ such that for all balls $B \subset \partial\O$, 
$ (\nu(B)^{-1}\int_B g^q d\nu)^{1/q} \leq C\nu(B)^{-1}\int_B g d\nu$. 
The union of the $B_q$ classes is the $A_{\infty}$ class. By real variable methods,
it is known that if elliptic measure and surface measure on a domain are
related via a weight in the $A_{\infty}$ class then the Dirichlet problem with data
in $L^p(d\sigma)$ is solvable {\it for some} $p < \infty$.
There is a well known relationship between the Muckenhoupt $B_q$ weight classes, the
existence of estimates for maximal functions and nontangential maximal functions, and
the solvability of Dirichlet problems for second order elliptic divergence form
operators. We assume that the reader is familiar with these results in harmonic
analysis/elliptic theory. 

One specific and nontrivial result in this theory is Dahlberg's result of 1977: the
harmonic measure $\omega$ on a Lipschitz domain is mutually absolutely continuous
with respect to surface measure, $\sigma$, and the weight $k$ relating the two measures
$(d\omega = k d\sigma$) belongs to the $B_2(d\sigma)$ class.
There is a further relationship between Muckenhoupt weights and the function
space $BMO$ of functions of bounded mean oscillation which then implies 
that $\log k \in BMO(d\sigma)$. On $C^1$ domains, Jerison and Kenig
proved that $\log k \in VMO(d\sigma)$. $VMO$ is the Sarason space of vanishing
mean oscillation, a proper subspace of $BMO$, and arises as the predual
of the Hardy space $H^1$. In \cite{kt1}, Kenig and Toro showed that $\log k$ belongs
to $VMO(d\sigma)$ when the domain is merely chord arc (with a vanishing condition).

The theory of perturbation of elliptic operators on Lipschitz domains begins
with a result of Dahlberg, \cite{d}, which measures the difference between
coefficients of the matrices of two divergence form elliptic operators in
a Carleson norm. Here is the set up for the general perturbation theory:
If $L_0 = \textrm{div}A_0\nabla$ is elliptic in 
a domain $\O$, then an elliptic operator 
$L_1 = \textrm{div}A_1\nabla$ is a perturbation of $L_0$ when 
the difference $\epsilon(X) = |A_1(X) - A_0(X)|$ is equal to zero
when $X \in \partial\O$. How closely should these operators, $L_0$ and $L_1$,
agree in the interior of the domain so that good properties of the elliptic
measure associated to $L_0$ be preserved?  The correct answer to this
question is stated in terms of Carleson measures. The Carleson condition on $\epsilon(X)$ is a delicate
measure of the rate at which $\epsilon(X)$ tends to zero as $X$ approaches
the boundary of $\O$. In terms of such Carleson conditions, sharp results on perturbations
were obtained in \cite{fkp}. And in $\cite{esc1}$, Escauriaza showed
that a (vanishing) Carleson condition on a perturbation of the Laplacian
in $C^1$ domains preserved the Jerison-Kenig result, namely that
$\log k \in VMO$.  We will provide precise statements of some 
of these results in the next
section. 

Our aim is to extend the perturbation results of \cite{fkp} to
the setting of chord arc domains (CADs).  Much of the technology of function
spaces on the boundary which is available when the domain has
Lipschitz boundary is not available in this setting. Therefore, a good portion of
this paper is devoted to developing the theory of these function spaces for chord
arc domains, especially the theory of tent spaces due to Coifman, Meyer and Stein.
These function spaces and their duals figure prominently in the theory of Hardy spaces
and $BMO$ spaces - the connection between them is established via Carleson measures. The
development of the theory of tent spaces on chord arc domains is a purely geometric
and independent aspect of this paper. 

In \cite{mt}, it was shown that a (vanishing) Carleson measure
condition on perturbations of the Laplacian on CADs with vanishing constant preserves
$A_{\infty}$. In the last section of the paper we show that 
this result holds for perturbations from arbitrary
elliptic divergence form operators on general CADs.

\section{Preliminaries}
\renewcommand{\thesection}{\arabic{section}}

In this section we recall some definitions and give the necessary background on properties of solutions to elliptic equations in divergence form. We will also introduce some notations which will be used throughout the paper.

\begin{defn}
Let $\Omega \subset \R^{n}$. We say that
$\Omega$ is a chord arc domain (CAD) if $\Omega$ is an NTA
set of locally finite perimeter whose boundary is Ahlfors regular, i.e. the surface measure to the boundary satisfies the following condition: there exists $C\ge 1$ so that for $r\in (0,{\rm{diam}}\,\Omega)$ and $Q\in\partial \Omega$
\begin{equation}\label{1.7A}
C^{-1}r^{n-1}\le
\sigma(B(Q,r))\le
 Cr^{n-1}.
\end{equation}
Here $B(Q,r)$ denotes the $n$-dimensional ball of radius
$r$ and center $Q$ and $\sigma=\mathcal{H}^{n-1}\res \partial\Omega$ and $\mathcal{H}^{n-1}$ denotes the $n-1$-dimensional Hausdorff measure. The best constant $C$ above is referred to as the Ahlfors regularity constant.
\end{defn}

\begin{defn}
Let $\Omega \subset \R^{n}$, $\delta> 0$ and $R>0$. 
If D denotes Hausdorff measure and $\mathcal L(Q)$ denotes an $ n-1$-plane containing
a point $Q \in \Omega$, set 
\begin{equation}
\theta(r) = \sup_{Q \in \partial\Omega}\inf_{\mathcal L(Q)}
r^{-1} D[\partial\Omega \cap B(Q,r), \mathcal L \cap B(Q,r)]
\end{equation}\
We say that
$\Omega$ is a $(\delta,R)$-chord arc domain (CAD) if $\Omega$ is a
set of locally finite perimeter such that
\begin{equation}\label{1.6}
\sup_{0<r\le
 R}\theta(r)\le
 \delta
\end{equation}
and
\begin{equation}\label{1.7}
\sigma( B(Q,r))\le
 (1+\delta)\omega_{n-1}r^{n-1}\ \ \forall
Q\in\partial\Omega\ \ {\rm{and}} \ \forall r\in (0,R].
\end{equation}
Here $\omega_{n-1}$ is the volume of the $(n-1)$-dimensional unit
ball in $\R^{n-1}$.
\end{defn}

\begin{defn}
Let $\Omega \subset \R^{n}$, we say that $\Omega$ is a chord arc
domain with vanishing constant if it is a $(\delta, R)$-CAD for
some $\delta>0$ and $R>0$,
\begin{equation}\label{1.8}
\limsup_{r\rightarrow 0}\theta(r)=0
\end{equation}
and
\begin{equation}\label{1.9}
\lim_{r\rightarrow 0}\sup_{Q\in\partial \Omega}\frac{\sigma(B(Q,r))}{\omega_{n}r^{n-1}} =1.
\end{equation}
\end{defn}

For the purpose of this paper we assume that $\Omega\subset \R^{n}$ is a bounded
domain.  We consider elliptic
operators $L$ of the form
\begin{equation}\label{div-tt}
Lu=\textrm{div}(A(X)\nabla u)
\end{equation}
defined in the domain $\Omega$ with symmetric coefficient matrix $A(X)=(a_{ij}(X))$ and such that there
are $\lambda, \Lambda>0$ satisfying
\begin{equation}\label{ellipticity}
\lambda |\xi|^2\le
 \sum_{i,j=1}^{n}a_{ij}(X)\xi_i\xi_j\le
 \Lambda |\xi|^2
\end{equation}
for all $X \in \Omega$ and $\xi \in \R^{n}$.

We say that a function $u$ in $\Omega$ is a solution to $Lu=0$ in $\Omega$ provided that $u\in W_{\rm{loc}}^{1,2}(\Omega)$ and for all $\phi\in C^{\infty}_c(\Omega)$
$$\int_{\Omega}\langle A(x)\nabla u,\nabla\phi\rangle dx =0.$$

A domain $\Omega$ is called regular for the operator $L$, if for
every $g\in C(\partial \Omega)$, the generalized solution of the
classical Dirichlet problem  with boundary data $g$ is a function $u\in C(\overline{\Omega})$. Let $\Omega$ be a regular domain for $L$ as above and $g\in C(\partial \Omega)$. The Riesz Representation Theorem ensures that there exists a family of regular Borel probability measures
$\{\omega^X_L\}_{X\in\Omega}$ such that
$$u(X)=\int_{\partial\Omega}g(Q)d\omega^X_L(Q).$$
For $X\in \Omega$, $\omega^X_L$ is called the $L-$elliptic measure of $\Omega$ with pole $X$. When no
confusion arises, we will omit the reference to $L$ and simply
called it as the elliptic measure.

To state our results  we introduce the notion of perturbation of an operator. Consider two elliptic operators $L_i={\rm{div}}(A_i\nabla\ )$ for $i=0,1$ defined on a chord arc domain $\Omega\subset \R^{n}$. We say that $L_1$ is a perturbation of $L_0$ if the deviation function
\begin{equation}\label{eqn:tt-a}
a(X)=\sup\{|A_1(Y)-A_0(Y)|: Y\in B(X,\delta(X)/2)\}
\end{equation}
where
$\delta(X)$ is the distance of $X$ to $\partial \Omega$, satisfies the following Carleson measure property: there exists a constant $C>0$ such that
\begin{equation}\label{normfkp}
\sup_{0<r<\rm{diam}\Omega}\sup_{Q\in\partial \Omega}
\bigg\{\frac{1}{\sigma(B(Q,r))}\int_{B(Q,r)\cap \Omega}\frac{a^2(X)}{\delta(X)}dX\bigg\}^{1/2}\le
 C.
\end{equation}

Note that in this case $L_1=L_0$ on $\partial \Omega$. For $i=0,1$ we denote by $G_i(X,Y)$  the Green's function of $L_i$ in $\Omega$ with pole at $X$ and by $\omega^X_i$ the corresponding elliptic measure.

We now recall some of the results concerning the regularity of the elliptic measure of perturbation operators
in Lipschitz domains. The results in the literature are more general than those quoted below.

\begin{thm}\label{dahlberg1986}\cite{d}
Let $\Omega=B(0,1)$. If $L_0=\Delta$, $a(X)$ is as in (\ref{eqn:tt-a}),
\begin{equation}\label{carlesonnorm}
h(Q,r)=\bigg\{\frac{1}{\sigma(B(Q,r))}\int_{B(Q,r)\cap \Omega}\frac{a^2(X)}{\delta(X)}dX\bigg\}^{1/2}
\end{equation}
and
$$
\lim_{r \rightarrow 0}\sup_{|Q|=1}h(Q,r)=0,
$$
then the elliptic kernel of $L_1$, $k=d\omega_{L_1}/d\sigma\in B_q(d\sigma)$ for all $q>1$.
\end{thm}

In \cite{feff}, R.Fefferman investigated an alternative to the smallness condition on $h(Q,r)$ above, and considered a pointwise requirement on the quantity $A(a)(Q)$.

\begin{thm}\cite{feff}\label{fefferman1989}
Let $\Omega= B(0,1)$ and $L_0=\Delta$. Let $\Gamma(Q)$ denote a non-tangential cone with vertex $Q\in \partial \Omega$ and
$$A(a)(Q)=\bigg(\int_{\Gamma(Q)}\frac{a^2(X)}{\delta^n(X)}dX\bigg)^{1/2},$$
where $a(X)$ is as in (\ref{eqn:tt-a}).
If $\|A(a)\|_{L^{\infty}}\le
 C$ then $\omega\in A_{\infty}(d\sigma)$.
\end{thm}

The main results in \cite{d} and in \cite{feff} are proved using Dahlberg's idea of introducing a differential inequality for a family of elliptic measures. In \cite{fkp}, Fefferman, Kenig and Pipher presented a new direct proof of these results and we will show here that this proof extends beyond the class of Lipchitz domains. This requires a careful reworking of many of the technical steps in the \cite{fkp} proof, and the development of the required new analytic tools for CADs.

\begin{thm}\label{fkp1991}\cite{fkp}
Let $\Omega$ be a Lipschitz domain. Let $L_1$ be such that (\ref{normfkp})  holds then
 $\omega_{1}\in A_\infty(d\sigma)$ whenever $\omega_{0}\in A_\infty(d\sigma)$.
\end{thm}

\begin{thm}\label{fkp1991b}\cite{fkp}
Let $\Omega$ be a Lipschitz domain. Let $\omega_0$, $\omega_1$ denote the $L_0-$elliptic measure and the $L_1-$elliptic measure respectively in $\Omega$ with pole $0\in\Omega$.
There exists an $\varepsilon_0>0$, depending on the ellipticity constants and the dimension, such that if
$$\sup_{\Delta \subseteq \partial \Omega}\bigg\{\frac{1}{\omega_0(\Delta)}\int_{T(\Delta)}a^2(X)\frac{G_0(X)}{\delta^2(X)}dX\bigg\}^{1/2}\le
 \varepsilon_0,$$
then $\omega_1\in B_2(\omega_0)$.
Here $T(\Delta)=\overline B(Q,r)\cap \Omega$ is the Carleson region associated to the surface ball $\Delta(Q,r)=B(Q,r)\cap \partial \Omega$,  and $G_0(X)=G_0(0,X)$ denotes the Green's function for $L_0$ in $\Omega$ with pole at $0\in \Omega$.
\end{thm}

In the recent paper \cite{mt}, Theorem \ref{fkp1991} was generalized to chord arc domains with small constant in the case $L_0=\Delta$. More precisely,

\begin{thm}\label{mainpertu}\cite{mt}
Let $\Omega$ be a chord arc domain. Let $L_0=\Delta$ and $L_1$ be such that (\ref{normfkp}) holds. There exists $\delta(n)>0$ such that if $\Omega \subset \R^{n}$ is a $(\delta,R)-$CAD with $0<\delta\le
\delta(n)$
then $\omega_{1}\in A_\infty(d\sigma)$.
\end{thm}

The purpose of the present paper is to extend the result above to perturbation operators on ``rough domains".
In particular we will show the following result.

\begin{thm}\label{mainthm1}
Let $\Omega$ be a chord arc domain. There exists an $\varepsilon_0
>0$, depending also on the ellipticity constants and the dimension, such that if
\begin{equation}\label{condThm2.11}
\sup_{\Delta \subseteq \partial \Omega}\bigg\{\frac{1}{\omega_0(\Delta)}\int_{T(\Delta)}a^2(X)\frac{G_0(X)}{\delta^2(X)}dX\bigg\}^{1/2}\le
 \varepsilon_0
,
\end{equation}
then $\omega_1\in B_2(\omega_0)$.
\end{thm}

The various constants that will appear in the sequel may vary from formula to formula, although for simplicity we use the same letter(s). If we do not give
any explicit dependence for a constant, we mean that it depends
only on the usual parameters such as ellipticity constants, NTA constants, the Ahlfors regularity constant, the dimension  and the NTA character of the domain. Moreover throughout the paper we shall use the notation $a\lesssim b$ to mean that there is a constant $c>0$ such that $a\le
 cb$. Similarly $a\simeq b$ means that $a\lesssim b$ and $b\lesssim a$.

Next we recall the main theorems we will use about the boundary behavior of
$L-$elliptic functions in non-tangentially accessible (NTA)
domains for uniformly elliptic divergence form operators $L$ with bounded measurable coefficients. We refer the reader to \cite{k1} for the definitions and more details regarding elliptic operators of divergence form
defined in NTA domains.

\begin{lemma}\label{lem2.3}
Let $\Omega$ be an NTA domain, $Q\in\partial\Omega$, $0<2r<R$, and
$X\in\Omega\backslash B(Q,2r)$. Then
$$
C^{-1}<\frac{\omega^{X}(B(Q,r))}{r^{n-2}G(A(Q,r),X)}<C,
$$
where $G(A(Q,r),X)$ is the $L-$Green function of $\Omega$ with pole $X$, and $\omega^X$ is the corresponding elliptic measure.
\end{lemma}

\begin{lemma}\label{lem2.4}
Let $\Omega$ be an NTA domain with constants $M>1$ and $R>0$,
$Q\in\partial\Omega$, $0<2r<R$, and $X\in\Omega\backslash
B(Q,2Mr)$. Then for $s\in [0,r]$
$$
\omega^{X}(B(Q,2s))\le C \omega^{X}(B(Q,s)),
$$
where $C$ only depends on the NTA constants of $\Omega$.
\end{lemma}

\begin{lemma}\label{lem2.5}
Let $\Omega$ be an NTA domain, and $0<Mr<R$. Suppose that $u,v$ vanish continuously on
$\partial\Omega\cap B(Q,Mr)$ for some $Q\in\partial\Omega$, $u,v\ge
 0$ and
$Lu=Lv=0$ in $\Omega$. Then there exists a
constant $C>1$ (only depending on the NTA constants) such that for
all $X\in B(Q,r)\cap\Omega$,
$$
C^{-1}\frac{u(A(Q,r))}{v(A(Q,r))}\le
 \frac{u(X)}{v(X)}\le
C\frac{u(A(Q,r))}{v(A(Q,r))}.
$$
\end{lemma}
\begin{lemma}\label{1510}
Let $\mu\in A_\infty(d\omega)$, $0\in\Omega$, and set 
 $$S_\alpha(u)(Q)=\bigg(\int_{\Gamma_\alpha(Q)}|\nabla u(X)|^2 {\delta(X)}^{2-n}dX\bigg)^{1/2}\ \ {\rm{and}}\ \ N_\alpha(u)(Q)=\sup\{|u(X)|: X\in \Gamma_\alpha(Q)\} $$
where for $Q\in\partial\Omega$
\begin{equation}\label{cone-a}
\Gamma_\alpha(Q)=\{X \in \Omega: |X-Q|< (1+\alpha)\delta(X)\}
\end{equation} 
Then if $Lu=0$ and $0<p<\infty$,
$$\bigg(\int_{\partial\Omega}(S_\alpha(u))^pd\mu\bigg)^{1/p}\le
 C_{\alpha,p}\bigg(\int_{\partial\Omega}(N_\alpha(u))^pd\mu\bigg)^{1/p}.$$
If in addition $u(0)=0$ then
$$\bigg(\int_{\partial\Omega}(N_\alpha(u))^pd\mu\bigg)^{1/p}\le
 C_{\alpha,p}\bigg(\int_{\partial\Omega}(S_\alpha(u))^pd\mu\bigg)^{1/p}.$$
 \end{lemma}

\section{Nontangential behavior in CADs.}\label{nontan}

In this section we study the space of functions defined on chord arc domains whose non-tangential maximal function is well behaved
Our goal is to extend the theory of tent spaces developed by Coifman, Meyer and Stein \cite{cms} in the upper half space  to chord arc domains. It will play a crucial role in the proof of Theorem \ref{mainthm1}.

We first study the notion of global $\gamma$-density with respect to a set $\mathcal F\subset \partial\Omega$ for a doubling measure $\mu$ supported on $\partial\Omega$. In this paper $\mu$ will either be the elliptic measure of a divergence form operator defined on $\Omega$ or the surface measure to $\partial\Omega$. Please note that in contrast with the classical case we do not restrict the definition to the case 
$\mathcal F$ closed.

\begin{defn}\label{g-density} Let $\mathcal F\subset \partial\Omega$, and let $\gamma\in (0,1)$. A point $Q\in \partial\Omega$ has global 
$\gamma$-density with respect to $\mathcal F$ for a doubling measure $\mu$ if for $\rho\in (0, \rm{diam}\, \Omega)$
\begin{equation}\label{eqn-g-density}
\frac{\mu(B(Q,\rho)\cup \mathcal F)}{\mu(B(Q,\rho)}\ge\gamma.
\end{equation}
Let $\mathcal F^\ast_\gamma$ the set of points of global $\gamma$-density of $\mathcal F$.
\end{defn}

\begin{lemma}\label{3.10tt}
let $\Omega$ be a CAD with surface measure $\sigma$ and let
$\Delta(P,s) =  B(P,s) \cap \partial\Omega$ be the
surface ball centered at $P$.
Given $\alpha>0$ there exists $\gamma\in (0,1)$ close to 1 and $\lambda_0>0$ such that if $\mathcal F\subset\partial\Omega$ and $P\in \mathcal F_\gamma^\ast$ then for $P\in \Delta(P, (2+\alpha)r)$
\begin{equation}\label{3.200tt}
\sigma(\Delta(Q,\beta r)\cup\mathcal F)\ge \lambda_0 (\beta r)^{n-1},
\end{equation}
where $\beta=\min\{1,\alpha\}$.
\end{lemma}

\begin{proof}
Assume $\sigma(\Delta(Q,\beta r)\cup\mathcal F)<\lambda_0 (\beta r)^{n-1}$. Since $\Delta(Q,\beta r)\subset \Delta(P, (2+\alpha+\beta)r)\subset \Delta(P, (3+2\alpha)r)$ and  $P\in \mathcal F_\gamma^\ast$ then 
\begin{eqnarray}\label{3.201tt}
\gamma\sigma(\Delta(P,(3+2\alpha)r)&\le& \sigma(\Delta(P,(3+2\alpha)r\cup\mathcal F)\\
&\le &\sigma(\Delta(P,(3+2\alpha)r)\backslash \Delta(Q,\beta r)) + \sigma(\Delta(Q,\beta r))\nonumber\\
&\le & \sigma(\Delta(P,(3+2\alpha)r)\left [1-\frac{1}{C^2}\cdot(\frac{\beta}{3+2\alpha})^{n-1} +C^2\lambda_0(\frac{\beta}{3+2\alpha})^{n-1}\right],\nonumber
\end{eqnarray}
where $C$ denotes the Ahlfors regularity constant. For $\lambda_0=1/{2C^4}$, (\ref{3.201tt}) implies that $\gamma\le 1--\frac{1}{2C^2}\cdot(\frac{\beta}{3+2\alpha})^{n-1} <1$ which is a contradiction if $\gamma$ is close enough to 1.
\end{proof}

Please note that so far we have not assumed that the set $\mathcal F$ is closed. The following proposition requires the set $\mathcal F$ to be  closed. It holds for a general doubling measure supported on $\partial\Omega$ but will only be applied to either surface measure or elliptic
measure.

\begin{prop}\label{size-g-dens-set}
Let $\Omega$ be a CAD, and let $\mu$ be a doubling measure on $\partial\Omega$. Let $\mathcal F\subset\partial\Omega$ be a closed set. 
Then $\mathcal F_\gamma^\ast\subset \mathcal F$ and 
\begin{equation}\label{eqn-size-g}
\mu\left((\mathcal F_\gamma^\ast)^c\right)\le C\mu(\mathcal F^c).
\end{equation}
Here the constant $C$ depends on $\gamma$ and on the doubling constant of $\mu$.
\end{prop}

\begin{proof} Since $\mathcal F$ is closed it is clear that $\mathcal F_\gamma^\ast\subset \mathcal F$. Let $\mathcal O=\mathcal F^c$
and $\mathcal O^\ast=(\mathcal F_\gamma^\ast)^c$. If $Q \in O^{\ast}$ by definition there exists a radius $\varrho_Q >0$ such that
$$
 \frac{\mu(\Delta(Q, \varrho_Q)\cap \mathcal{F})}{\mu(\Delta(Q, \varrho_Q))}< \gamma.
 $$
By Besicovitch (see \cite{eg})
$$
O^*\subset \bigcup_{i=1}^{N_n}\bigcup_j \Delta(Q_j^i,\varrho_j^i)
$$
where $\Delta(Q_j^i,\varrho_j^i)\cap \Delta(Q_l^i,\varrho_l^i)=\emptyset$ for $j\neq l$. Therefore
$$
\mu(O^*)\le
 \sum_{i=1}^{N_n}\sum_j\mu(\Delta(Q_j^i,\varrho_j^i))
$$
and
\begin{eqnarray}
  \sum_{i=1}^{N_n}\sum_j\mu(\Delta(Q_j^i,\varrho_j^i)) &=& \sum_{i=1}^{N_n}\sum_j\mu(\Delta(Q_j^i,\varrho_j^i)\cap\mathcal{F})+\mu(\Delta(Q_j^i,\varrho_j^i)\cap O)\nonumber \\
  &\le
 & \sum_{i=1}^{N_n}\sum_j\gamma\mu(\Delta(Q_j^i,\varrho_j^i))+\mu(\Delta(Q_j^i,\varrho_j^i)\cap O)\nonumber
\end{eqnarray}
hence
\begin{equation}\label{Besic}
\mu(O^*)\le
 C \sum_{i=1}^{N_n}\sum_j\mu(\Delta(Q_j^i,\varrho_j^i)\cap O)\le
 C\mu(O).
\end{equation}
\end{proof}

\begin{defn}\label{defofN}
Let $\Omega$ be a CAD. We denote by $\mathcal{N}$ a linear space of Borel measurable functions $F$ such that
$$\mathcal{N}=\{F:\Omega\rightarrow\R \ {\rm{such\ that}}\ N(F) \in L^1(d\sigma)\}$$
where $N(F)(Q)=\sup\{|F(X)|: X\in \Gamma(Q)\}$ and $\Gamma(Q)=\Gamma_1(Q)$ as defined in (\ref{cone-a}).
\end{defn}

\begin{preremark}
The set $\mathcal{N}$ with the norm given by $||F||_{\mathcal{N}}= ||N(F)||_{L^1(\partial \Omega)}$
is a Banach space.
\end{preremark}

The following proposition shows that the definition of the space $\mathcal N$  above does not depend on the aperture of the cone used.

\begin{prop}\label{Ndep} Let $\Omega$ be a CAD.
Let $\mu$ be a doubling measure supported on $\partial\Omega$. For $Q\in\partial\Omega$ let
$$
N_{\alpha}F(Q)=\sup_{X\in\Gamma_\alpha(Q)}|F(X)|,
$$
where $\Gamma_\alpha(Q)$ is as defined in (\ref{cone-a}).
Then given $\alpha,\ \beta>0$ there exists a constant $C$ depending on $\alpha$, $\beta$ and the doubling constant of $\mu$ such that
for all $\lambda>0$
\begin{equation}\label{mmN}
\mu(\{X\in\partial\Omega\ :\ N_{\alpha}F(X)>\lambda\})\le
 C\mu(\{X\in\partial\Omega\ :\ N_{\beta}F(X)>\lambda\}).
\end{equation}
Hence for $1\le p<\infty$
\begin{equation}\label{angle-ind}
\int |N_\alpha F|^p\, d\mu\le C \int |N_\beta F|^p\, d\mu\
\end{equation}
 \end{prop}

\begin{proof}
If $\alpha\le \beta$ the inequality (\ref{mmN}) is automatic thus we may assume that $\alpha>\beta$.  To prove (\ref{mmN}) we would like to 
apply Proposition \ref{size-g-dens-set}. We claim that for $\gamma\in(0,1)$ close enough to 1 the set $\{X\in \partial\Omega\ :\ N_{\alpha}F(X)>\lambda\}$ is contained in the complement of the set of points of global $\gamma-$density with respect to 
$\{X\in \partial\Omega\ :\ N_{\beta}F(X)>\lambda\}^c$. It is straightforward to show that  the set $\{X\in \partial\Omega\ :\ N_{\beta}F(X)>\lambda\}$ is open, which ensures that  Proposition \ref{size-g-dens-set} applied to $\mathcal F=\{X\in \partial\Omega\ :\ N_{\beta}F(X)>\lambda\}^c$ combined with the previous claim yields (\ref{mmN}). To prove the claim assume that $N_{\alpha}F(Q)>\lambda$ for $Q\in\partial\Omega$
there exists $Y\in\Gamma_{\alpha}(Q)$ such that $F(Y)\ge
 \lambda$ and $|Q-Y|< (1+\alpha)\delta(Y)$. Now let $Q_Y\in\partial\Omega$ such that $|Y-Q_Y|=\delta(Y)$ then $\Delta(Q_Y,\beta \delta(Y))\subset\{P\in\partial\Omega : N_{\beta}F(P)> \lambda\}\cap \Delta(Q,(\alpha+\beta+2)\delta(Y))$. In fact, if $P\in \Delta(Q_Y,\beta \delta(Y))$ then $|P-Y|\le
|P-Q_Y|+|Q_Y-Y|< (1+\beta)\delta(Y)$ and $F(Y)>\lambda$. Therefore since $\mu$ is doubling
\begin{eqnarray}\label{tt-comp}
\frac{\mu\big(\{P:N_{\beta}F(P)>\lambda\}\cap\Delta(Q,(\alpha+\beta+2)\delta(Y))\big)}{\mu\big(\Delta(Q,(\alpha+\beta+2)\delta(Y))\big)}&\ge
& \frac{\mu\big(\Delta(Q_Y,\beta\delta(Y))\big)}{\mu\big(\Delta(Q,(\alpha+\beta+2)\delta(Y))\big)}\\
&\ge
& \frac{\mu\big(\Delta(Q_Y,(2+\alpha+\beta)\delta(Y))\big)}{\mu\big(\Delta(Q,(\alpha+\beta+2)\delta(Y))\big)}\nonumber\\
&\ge
& \frac{\mu\big(\Delta(Q,\beta\delta(Y))\big)}{\mu\big(\Delta(Q,(\alpha+\beta+2)\delta(Y))\big)}\nonumber\\
&\ge & C_0,\nonumber
\end{eqnarray}
where $C_0$ depends on $\alpha$, $\beta$ and the doubling constant of $\mu$.  Note that (\ref{tt-comp}) shows that for $Q\in \partial\Omega$ such that $N_{\alpha}F(Q)>\lambda$, $Q$ is not a global $\gamma$-density point with respect to $\{P\in\partial\Omega:N_{\beta}F(P)>\lambda\}$ whenever $\gamma\in (1-C_0/2, 1)$, which proves our claim.
\end{proof}

One of the goals of this section is to study the dual of the space $\mathcal{N}$. To achieve this we still need to understand better
the geometry of $\Omega$ and the structure of its boundary. To this effect we first prove a Whitney decomposition type lemma for open subset of $\partial\Omega$. Then we define the ``tent" over an open subset of $\partial\Omega$. Finally we define Carleson measures 
on $\Omega$.

\begin{lemma} \label{one}
Let $F \subset \partial \Omega$ be a closed nonempty set on $\partial \Omega$. There exist a family of balls $\{B_k\}$ with
$B_k=B(X_k, r_k),$ $X_k \in \partial \Omega$ and constants $1<c^*<c^{**}$ such that if $B^*_k=c^*B_k=B(X_k,c^*r_k)$, $B^{**}_k=c^{**}B_k$ then
\begin{itemize}
  \item   $B_k \cap B_j =\emptyset$,\ for \ $k\neq j$
  \item $\bigcup_k B^{\ast}_k = F^c \cap \partial \Omega=O$
  \item $B^{\ast \ast}_k \cap F\neq \emptyset$.
\end{itemize}
In addition if we define
$$ Q_k= B^{\ast}_k \cap (\bigcup_{j<k}Q_j)^c \cap (\bigcup_{j>k}B_j)^c,$$ 
then $B_k \subset Q_k \subset B^{\ast}_k$, the $Q_k$'s are disjoint and
$\bigcup_{k=1}^{\infty} Q_k=O.$
\end{lemma}

\begin{proof}
Consider $0<\varepsilon<1/6$ and let
$d(X)= \sup \{d: B(X, d) \cap \partial \Omega \subset O\}$
for $X \in \partial \Omega$. Let us choose a maximal disjoint subcollection of $\{B(X, \varepsilon d(X))\}_{X \in O}.$ For this countable subcollection
$\{B_k\}_{k=1}^{\infty},$ where $B_k:=B(X_k, \varepsilon d(X_{k}))$ and $X_k \in \partial \Omega$ we consider
$B^{\ast}_k=B(X_k, \frac{d(X_{k})}{2}))$ and $B^{\ast\ast}_k=B(X_k, 2 d(X_{k})).$ Clearly $i)$ and $iii)$ hold, moreover $B^{*}_k \subset O.$
To show that $$O \subset \bigcup_{k\ge
1}B^{*}_k$$ we take $Y \in O$. By the selection of $\{B_k\}$ there exists $k$ such that
\begin{equation} \label{ineq1}
B(Y, \varepsilon d(Y))\cap B(X_k, \varepsilon d(X_k)) \neq \emptyset,
\end{equation}
therefore $|Y-X_k|<\varepsilon d(Y)+\varepsilon d(X_k)$. Moreover $d(Y)\le
 |X_k-Y|+d(X_k)$ which implies $d(Y)\le
\frac{1+\varepsilon}{1-\varepsilon}d(X_k)$ and as a consequence
$|Y-X_k|<3\varepsilon d(X_k)<\frac{d(X_k)}{2}$ since $\varepsilon<1/6$.

By construction $B_1 \subset Q_1 \subset B^{*}_1$. Assume that for $k\ge
 2$ and $j\le k-1$
$B_j \subset Q_j \subset B^{*}_j$ and note that $B_k \subset (\bigcup_{j>k}B_j)^c\cup B_k^*.$
By definition and using the hypothesis of induction we have for $j<k$
$$ 
Q^{c}_j= (B^{\ast}_j)^c \cup (\bigcup_{i<j}Q_j) \cup (\bigcup_{i>j}B_i)\supset  (\bigcup_{i\not = j}B_i)\supset B_k
$$
and 
$
\bigcap_{j<k}Q^{c}_j \supset \bigcup_{i \ge
 k} B_i \supset B_k$ which ensures that $B_k\subset Q_k\subset B_k^*$. 

It is clear that $\bigcup_{k=1}^{\infty} Q_k\subset =\bigcup_{k=1}^{\infty} B_k^*=O$. 
For $X \in O$ consider two cases.
Either there exists $j$ such that $X \in B(X_j, \varepsilon d(X_j)) =: B_j\subset  B^{\ast}_j$, $X \notin Q_i$ for $i< j$, 
$X \notin B_i$ for $i\neq j$ and therefore $X \in Q_j.$ Or for all $j$, $X \notin B_j$. In this case, there exists a $k_0$ such that $X \in B^{\ast}_{k_0}$ but $X \notin B^{\ast}_k,$
for $k<k_0$. Hence $X \notin Q_k$ with $k<k_0,$ which implies $X \in Q^{*}_{k_0}$ and
$O \subset \bigcup_k Q_k.$
\end{proof}

To define the notion of ``tent" over an open subset of $\partial\Omega$ we first look at ``fans" of cones over subsets on $\partial\Omega$. Let 
$\mathcal F\subset \partial\Omega$ for $\alpha>0$ define
\begin{equation}\label{mm-Variablecone}
R_{\alpha}(\mathcal{F})  =\bigcup_{Q \in \mathcal{F}}\Gamma_{\alpha}(Q)
\end{equation}
where $\Gamma_\alpha(Q)$ is as define in (\ref{cone-a}). We denote $R_{1}(\mathcal{F})$ by $R(\mathcal{F})$.
Given an open set $O\subset\partial\Omega$ the tent over $O$ is defined as
\begin{equation}\label{tent}
T(O)=\Omega\setminus R(\mathcal{F}).
\end{equation}

\begin{lemma} \label{two}
Let O be the open set defined by $$O=\{Q\in\partial\Omega:  N(F)(Q)> \alpha\} ,$$ then $$ T(O)\subseteq\bigcup_{P \in O} T(\Delta(P, dist(P, O^c))).$$
\end{lemma}

\begin{proof} Recall that $T(\Delta(Q,r))=\overline B(Q,r)\cap \Omega$. 
Let $Y\in T(O)$ then $Y\notin R(\mathcal{F})$, hence $NF(Q_Y)> \alpha$ where by $Q_Y\in\partial\Omega$ we denote the boundary point satisfying $|Y-Q_Y|=\delta(Y)$. Now if $P \in \Delta(Q_Y,\delta(Y))$ then $|P-Y|< 2\delta(Y)$, $Y\in\Gamma(P)$, and since $Y\notin R(\mathcal{F})$ then $P\notin\mathcal{F}$, thus $P\in O$, i.e. $\Delta(Q_Y, \delta(Y))\subset O$
which implies $\delta(Y)\le {\rm{dist}}(Q_Y, O^c)$.
\end{proof}

\begin{defn}\label{carleson-meas}
Let $\Omega$ be a CAD.
For a Borel measure $\mu$ on $\overline{\Omega}$ we define for $Q\in\partial\Omega$,
$$
C(\mu)(Q)= \sup_{Q \in \Delta}\frac{1}{\sigma(\Delta)}\int_{T(\Delta)}d\mu
$$
and denote by
$$\mathcal{C}=\{\mu\ {\rm{is\ Borel\ in\ }}\overline{\Omega}: ||\mu||_{\mathcal{C}}= \sup_{Q \in \partial \Omega}C(\mu)(Q)<\infty\}.$$
$\mathcal{C}$ is the collection of Carleson measures on $\Omega$.
\end{defn}

\begin{lemma} \label{three}
Assume that $\mu$ is a positive measure on $\overline{\Omega}$ such that $||\mu||_{\mathcal{C}} \le
 1$ i.e. $\mu(T(B)) \le
 \sigma(B)$ for all balls B. Then, for every
open set $O \subset \partial \Omega$ $$\mu(T(O))) \le
 C \sigma(O). $$
\end{lemma}

\begin{proof}
As in the classical setting we appeal to the Whitney decomposition lemma \ref{one}. If the $B_k$'s are as in Lemma \ref{one} then we claim that
\begin{equation} \label{star2}
T(O) \subset \bigcup_k T(B^{\ast\ast}_k).
\end{equation}
Thus 
$$
\mu(T(O))) \le
 \sum_k \mu(T(B^{\ast\ast}_k))) \le
 C \sum_k \sigma(B^{\ast\ast}_k)
\le
 C \sum_k \sigma(B_k) \lesssim \sum_k \sigma(Q_k) \lesssim \sigma(O)
$$
where we have used the fact that $\sigma$ is Ahlfors regular in fact the doubling properties of $\sigma$ are enough). To prove (\ref{star2}), consider a point
$Z \in T(O)\subseteq \bigcup_{P \in O}T(\Delta(P, \dist(P, O^c))),$ that is $Z \in T(\Delta(P, \dist(P, O^c))$ for some $P \in O$, and there exists
a $k$ such that $P \in B^{*}_k$ and $|P-X_k|< \frac{d_k}{2}.$ Now if $Y \in O^c$ is such that $d_k=|Y-X_k|$ 
$ |Y-P| \le
 |Y-X_k|+|X_k-P|<\frac{3d_k}{2}$, $\dist(P, O^c)<\frac{3d_k}{2}$ and
$ |Z-X_k| \le
 |Z-P|+|P-X_k|<\dist(P, O^c)+\frac{d_k}{2}<2d_k,$
that is  $Z \in T(\Delta(X_k, 2d_k))=T(B^{**}_k).$
\end{proof}

We are now ready to study the relationship between the spaces $\mathcal{N}$ and $\mathcal{C}$. First we prove the analogue of Proposition 3 in \cite{cms}. 

\begin{prop} \label{thm1}
Let $\Omega$ be a CAD. If $F \in \mathcal{N}$ and $\mu \in \mathcal{C}$ then
\begin{equation}\label{thm3.3a}
\bigg|\int_{\Omega}F(X)d\mu(X)\bigg| \lesssim \int_{\partial \Omega} NF(Q)C(\mu)(Q)d\sigma.
\end{equation}
\end{prop}

\begin{proof}
Assume that $F \ge
 0$ and consider the open set
$ O =\{P \in \partial \Omega: NF(P)>\alpha\}$. Using the notation in Lemma \ref{one}
and the fact that $\sigma$ is Ahlfors regular we have for $X \in Q_k$
$$ \mu(T(B^{\ast\ast}_k)) \le
 C(\mu)(X)\sigma(B^{\ast\ast}_k) \le
 C(\mu)(X)\sigma(B_k)
\le
 \int_{Q_k}C(\mu)(X) d\sigma.$$
By Lemma \ref{three} and the fact that the $Q_k$'s are disjoint we have that
$$ 
\mu(\{|F(X)|>\alpha\}) \le
 \sum_k \mu(T(B^{\ast\ast}_k)) \le
 C \sum_k \int_{Q_k}C(\mu)(X) d\sigma
\le
 C \int_{\{NF(X)>\alpha\}}C(\mu)(X) d\sigma. $$
Integrating over $\alpha$ and using Fubini yields (\ref{thm3.3a}).
\end{proof}

\begin{cor}
Let $\mu\in\mathcal{C}$. Let F be a function defined on $\Omega$ be such that $NF \in L^p (d\sigma),$ for some
$p \in (0,\infty)$ fixed. Then,
\begin{equation}\label{star3}
\int_{\Omega}|F(X)|^p d\mu \le
 C \int_{\partial \Omega}|NF(Q)|^pC(\mu)(Q)d\sigma .
\end{equation}
\end{cor}

\begin{proof}
Inequality (\ref{star3}) follows from Proposition \ref{thm1} if we replace $|F(X)|$ by $|F(X)|^p$.
\end{proof}

We now present a couple of integration lemmas. They provide control of boundary integrals in terms of solid integrals on CAD, via Fubini.
In what follows the function $A(X)$ is a non-negative measurable function in $\Omega$. In Section \ref{Square} , we will take 
$A(X)$ to be be the square function.

\begin{lemma} \label{four}
Let $\Omega$ be a CAD. Given $\alpha>0$ if $\mathcal F\subset\partial\Omega$ and
$A$ is a non-negative measurable function in $\Omega$
then
\begin{equation}\label{tt-fub1}
\int_{\mathcal{F}} \bigg(\int_{\Gamma_{\alpha}(Q)}A(X)dX\bigg)d\sigma(Q) \le
 C_{\alpha}\int_{R_{\alpha}(\mathcal{F})}A(X)\delta(X)^{n-1}dX,
\end{equation}
where $R_{\alpha}(\mathcal{F})$ is given by (\ref{mm-Variablecone}).
\end{lemma}

\begin{proof}
By Fubini's Theorem
\begin{equation}\label{fub1}
\int_{\mathcal{F}} \bigg(\int_{\Gamma_{\alpha}(Q)}A(X)dX\bigg)d\sigma(Q)=
  \int_{\mathcal{F}} \int_{\Omega}A(X)\chi_{\Gamma_{\alpha}(Q)}(X)dX d\sigma(Q)=
  \int_{\Omega} \int_{\mathcal{F}}A(X)\chi_{\Gamma_{\alpha}(Q)}(X)d\sigma(Q)dX.
\end{equation}
If $\chi_{\Gamma_{\alpha}(Q)}(X)=1$ then $|X-Q| < (1+\alpha)\delta(X)$ and if $Q_X \in \partial \Omega$ is such that $|X-Q_X|=\delta(X)$ then
$|Q_X-Q|\le
 |X-Q_X|+|X-Q|< (2+\alpha)\delta(X),$
and
\begin{equation}\label{cart1}
\int_{\mathcal{F}} \chi_{\Gamma_{\alpha}(Q)}(X)d\sigma(Q)\le
 \sigma(\Delta(Q_X, (2+\alpha)\delta(X)))\le
 C_{\alpha}\delta(X)^{n-1}.
\end{equation}
Combining (\ref{cart1}) and (\ref{fub1}) we obtain
\begin{eqnarray*}
  \int_{\Omega} \int_{\mathcal{F}}A(X)\chi_{\Gamma_{\alpha}(Q)}(X)d\sigma(Q)dX &\le
& \int_{\Omega} \int_{\mathcal{F}}A(X)\chi_{\Gamma_\alpha(Q)}(X)\chi_{\Delta(Q_X, (2+\alpha)\delta(X))}(Q)d\sigma(Q)dX \\
   &\le
& \int_{R_{\alpha}(\mathcal{F})} A(X)\bigg(\int_{\mathcal{F}}\chi_{\Delta(Q_X, (2+\alpha)\delta(X))}(Q)d\sigma(Q)\bigg)dX \\
   &\le
& C_\alpha\int_{R_{\alpha}(\mathcal{F})} A(X)\delta(X)^{n-1}dX.
\end{eqnarray*}
\end{proof}

\begin{lemma} \label{six}
Let $\Omega$ be a CAD. Given $\alpha>0$ there exists $\gamma\in (0,1)$ close to 1 such that if $\mathcal F\subset\partial\Omega$ and
$A$ is a non-negative measurable function in $\Omega$
then
\begin{equation}\label{tt-fub2}
\int_{R_{\alpha}(\mathcal{F_\gamma^\ast})}A(X)\delta(X)^{n-1}dX,
\le C_\alpha
\int_{\mathcal{F}} \bigg(\int_{\Gamma_{\beta}(Q)}A(X)dX\bigg)d\sigma(Q),
\end{equation}
where $\beta=\min\{1,\alpha\}$.
\end{lemma}

\begin{proof} 
If $\chi_{\Gamma_{\beta}(Q)}(X)=0$ then $|X-Q|\ge
(1+\beta)\delta(X),$ and $|Q-Q_X| \ge
\beta\delta(X).$
Hence $\chi_{\Gamma_{\beta}(Q)}(X) \ge
 \chi_{\Delta(Q_X, \beta\delta(X)}(Q)$. Fubini's theorem yields
\begin{eqnarray}\label{3.2001tt}
\int_{\mathcal F}\int_{\Gamma_\beta(Q)} A(X)dXd\sigma(Q) &=&
 \int_{\partial\Omega} \int_{\mathcal{F}}  A(X)\chi_{\Gamma_{\beta}(Q)}(X) d\sigma(Q) dX \\
&   \ge
 & \int_{\Omega}A(X) \int_{\mathcal{F}}  \chi_{\Delta(Q_X, \beta\delta(X)}(Q) d\sigma(Q) dX\nonumber\\
&\ge &\int
_{R_{\alpha}(\mathcal{F}^{\ast}_\gamma)} A(X) \int_{\mathcal{F}}  \chi_{\Delta(Q_X, \beta\delta(X)}(Q) d\sigma(Q) dX.
\nonumber
\end{eqnarray}
Note that if $X\in R_{\alpha}(\mathcal{F}^{\ast}_\gamma)$ there is $P\in (\mathcal{F}^{\ast}_\gamma$ such that $X\in\Gamma_\alpha(P)$
and $Q_X\in\Delta(P, (2+\alpha)\delta(X))$ then applying (\ref{3.200tt}) in (\ref{3.201tt}) we obtain
\begin{equation}
\int_{\mathcal F}\int_{\Gamma_\beta(Q)} A(X)dXd\sigma(Q) \ge C_\alpha \int_{R_{\alpha}(\mathcal{F}^{\ast}_\gamma)} A(X)\delta(X)^{n-1}dX.
\end{equation}
\end{proof}

\begin{cor}
Let $\Omega$ be a CAD. Given $\alpha>0$ there exists $\gamma\in (0,1)$ close to 1 such that if $\mathcal F\subset\partial\Omega$ and
$f$ is a measurable function in $\Omega$
then
\begin{equation}\label{3.202tt}
\int_{\mathcal F_\gamma^\ast}\int_{\Gamma_\alpha(Q)}\frac{f^2(x)}{\delta(X)^n}dXd\sigma(Q)\le C_\alpha\int_{R_{\alpha}(\mathcal{F}^{\ast}_\gamma)}\frac{f^2(X)}{\delta(X)}dX\le C_\alpha  \int_{\mathcal F}\int_{\Gamma(Q)}\frac{f^2(x)}{\delta(X)^n}dXd\sigma(Q).
\end{equation}
\end{cor}

\begin{proof} Combining Lemma \ref{six} applied to $\mathcal F$ and Lemma \ref{four} applied to $\mathcal F_\gamma^\ast$ 
with $A(X)=\frac{f^2(X)}{\delta(X)}$ we obtain (\ref{3.202tt}).
\end{proof}

\section{Square functions in CADs}\label{Square}

Next we focus our attention on the tent spaces $T^p$ defined for chord arc domains, following the theory developed by Coifman, Meyer and Stein in \cite{cms}.
Suppose that $f$ is a measurable function defined on $\Omega$. For $\alpha>0$ and $Q \in \partial \Omega$, we define 
\begin{equation}\label{tt-square}
 A^{(\alpha)}(f)(Q)= \bigg( \int_{\Gamma_\alpha(Q)} {f(X)}^2 \frac{dX}{{\delta(X)}^{n}}\bigg)^{1/2}.
 \end{equation}
The square function of $f$ is defined as $A(f)=A^{(1)}(f)$.
By analogy with the space $\mathcal{N}$ defined in Section \ref{nontan}, we denote by $T^p$ for $1\le p<\infty$ the space of all Borel measurable functions given by
\begin{equation}\label{N1}
   T^p= \{ f\in L^2(\Omega): A(f) \in L^p(\sigma)
\}.
\end{equation}
\begin{preremark}
The space $T^p$ as defined above with the norm $||f||_ {T^p}=||A(f)||_{L^p(\sigma)}$ is a Banach space.
\end{preremark}

We define operator $C(f):\partial \Omega \rightarrow \mathbb{R}$ by
\begin{equation}\label{tt-carleson}
 C(f)(Q):= \sup_{Q \in \Delta}\bigg(\frac{1}{\sigma(\Delta)} \int_{T(\Delta)}f(X)^2\frac{dX}{\delta(X)}\bigg)^{1/2}
 \end{equation}
where $\Delta$ is a surface ball and $T(\Delta)$ is the tent over it.
We also introduce the space
\begin{equation}\label{Tinfty}
T^{\infty}=\{f\in L^2(\Omega):C(f)\in L^{\infty}(\sigma)\},
\end{equation}
with the norm $||f||_{T^{\infty}}=||C(f)||_{L^{\infty}}$.

\begin{thm} \label{theorem2}

\begin{description}
  \item[(a)] Whenever $g\in T^1$ and $ C(f) \in L^\infty(\sigma)$ then
               $$ \int_{\Omega}|f(X)g(X)|\frac{dX}{\delta(X)} \le
 C ||C(f) ||_{L^\infty}||g||_{T^1}.$$
  \item[(b)] More precisely
               $$ \int_{\Omega}|f(X)g(X)|\frac{dX}{\delta(X)} \le
 C \int_{\partial\Omega}C(f)(Q) A(g)(Q)d\sigma(Q) .$$
\end{description}
\end{thm}
\begin{proof}
Without loss of generality we may assume that both $f$ and $g$ are non-negative. For any $\tau >0$ we define the truncated cone
\begin{equation}\label{truncated-cone}
 \Gamma^{\tau}(Q)=\{X \in \Omega: |X-Q| < 2 \delta(X), \ \delta(X) \le \tau\}
 \end{equation}
and let $$ A_{\tau}(f)(Q):=\bigg( \int_{\Gamma^{\tau}(Q)}f(X)^2\frac{dX}{\delta(X)^n}\bigg)^{1/2}.$$
Note that $ A_{\tau}(f)$ increases with $\tau$, is constant for $\tau> {\rm{diam}}\,\Omega$ and $A_{\infty}(f)=A(f).$ 
Given $f$ define
the $``$stopping time$"$ $\tau(Q)$ which is given for $Q \in \partial \Omega$ by
$$ 
\tau(Q)= \sup \{\tau >0: A_{\tau}(f)(Q) \le
 \Lambda C(f)(Q)\}
$$
where $\Lambda$ is a large constant to be determined later. $\Lambda$ is only allowed to depend on $n$, the NTA constants and the Ahlfors regularity constant.\\
\emph{Claim: There exists a constant $c_0>0$ such that for every $Q_0 \in \partial \Omega$ and $0< r\le
 {\rm{diam}}\,\Omega$
$$ \sigma\bigg(\{Q \in \Delta(Q_0, r):\tau(Q) \ge
 r\}\bigg) \ge
 c_0 \sigma(\Delta(Q_0, r)).$$}
We assume that the previous claim holds and we prove that part (b) is satisfied by showing that for $H \ge
 0$
\begin{equation}\label{sect21}
   \int_{\Omega} H(X)\delta(X)^{n-1} dX \le
 C_1 \int_{\partial \Omega}\bigg\{\int_{\Gamma^{\tau(Q)}(Q)}H(X)dX \bigg\} d\sigma(Q)
\end{equation}
where $C_1$ depends on $c_0$ and the Ahlfors constant. Applying Fubini's Theorem we obtain
$$\int_{\partial \Omega}\int_{\Gamma^{\tau(Q)}(Q)}H(X)dX d\sigma(Q) = \int_{\Omega}\int_{\partial \Omega}H(X)\chi_{\Gamma^{\tau(Q)}(Q)}(X)d\sigma(Q)dX.$$
Note that if $Q \in \Delta(Q_X,\delta(X))$ and $\delta(X) \le
 \tau(Q)$ then
$ |Q-X|\le
 |Q-Q_X|+|Q_X-X|<2\delta(X)$ which implies that $X \in \Gamma^{\tau(Q)}(Q)$ and
$\chi_{\Gamma^{\tau(Q)}(Q)}(X) \ge
 \chi_{\Delta(Q_X,\delta(X)) \cap \{\tau(Q) \ge
 \delta(X)\}}(X).$
Therefore using the claim and the fact that $\sigma$ is Ahlfors regular we obtain
\begin{eqnarray*}
  \int_{\Omega}\int_{\partial \Omega}H(X)\chi_{\Gamma^{\tau(Q)}(Q)}(X)d\sigma(Q)dX &\ge
& \int_{\Omega}\int_{\partial \Omega}H(X)\chi_{\Delta(Q_X,\delta(X)) \cap \{\tau(Q) \ge
 \delta(X)\}}(Q)d\sigma(Q)dX \\
   &\ge
& \int_{\Omega}H(X) \sigma(\{Q \in \Delta(Q_X,\delta(X)):\tau(Q) \ge
 \delta(X)\})dX \\
   &\ge
& \int_{\Omega}H(X) C_0 \sigma(B(Q_X,\delta(X)))dX \\
   &\ge
& C_1 ^{-1}\int_{\Omega}H(X) \delta(X)^{n-1}dX.
\end{eqnarray*}
To prove part (b) we take $H(X)=f(X)g(X)\delta(X)^{-n}$ in the inequality (\ref{sect21}),
$$ \int_{\Omega}f(X)g(X)\frac{dX}{\delta(X)} \le
 C_1 \int_{\partial \Omega}\bigg(\int_{\Gamma^{\tau(Q)}(Q)}f(X)g(X)\delta(X)^{-n}dX\bigg) d\sigma(Q)$$
and then we use the Cauchy-Schwartz inequality in order to obtain
$$ \int_{\Gamma^{\tau(Q)}(Q)}f(X)g(X)\delta(X)^{-n}dX \le
 \left(\int_{\Gamma^{\tau(Q)}(Q)}\frac{f^2(X)}{\delta(X)^{n}}dX\right)^{1/2} \left(\int_{\Gamma^{\tau(Q)}(Q)}\frac{g^2(X)}{\delta(X)^{n}}dX\right)^{1/2}$$
therefore
\begin{equation}\label{H4.1}
\int_{\Omega}f(X)g(X)\frac{dX}{\delta(X)} \le
 C_1 \int_{\partial \Omega} A_{\tau(Q)}(f)(Q)A_{\tau(Q)}(g)(Q)d\sigma(Q).
\end{equation}
By the definition of $\tau(Q)$,
$$ A_{\tau(Q)}(f)(Q) \le
 \Lambda C(f)(Q)\ \hbox{ and }\  A_{\tau(Q)}(g)(Q) \le
 A(g)(Q)$$ hence
$$ \int_{\Omega}f(X)g(X)\frac{dX}{\delta(X)} \le
 C \int_{\partial \Omega} C(f)(Q)A(g)(Q)d\sigma(Q)$$ as required in part (b). In order to complete the proof, we need to prove the claim stated above.

\emph{Proof of Claim:} For $Q_0 \in \partial \Omega$ consider $\Delta=\Delta(Q_0, r)$ and $\widetilde{\Delta}=\Delta(Q_0, 3r).$ Note that $\bigcup_{Q \in \Delta}\Gamma^{r}(Q)\subset T(\widetilde{\Delta}).$ Indeed, if $X \in \Gamma^{r}(Q)$ for $Q \in \Delta$ then
$|X-Q|<2 \delta(X)$ and $\delta(X)\le r$, that is $|X-Q_0|<|Q_0-Q|+|Q-X|<r+2 \delta(X) \le
 3r$
which implies $X \in B(Q_0, 3r)\cap \Omega=T(\widetilde{\Delta}).$
Thus for $Q\in\Delta$
\begin{eqnarray*}
  \int_{\Delta}A_{r}^2(f)(Q)d\sigma(Q)&=& \int_{\Delta}\int_{\Gamma^{r}(Q)}\frac{f^2(X)}{\delta(X)^{n}}dX d\sigma(Q) \\
   &=& \int_{\Omega}\int_{\Delta}\frac{f^2(X)}{\delta(X)^{n}} \chi_{\Gamma^{r}(Q)}(X)d\sigma(Q)dX\\
  &\le
& \int_{\Omega}\frac{f^2(X)}{\delta(X)^{n}} \chi_{B(Q_0, 3r)}(X)\sigma(\Delta(Q_X, 3\delta(X)))dX \\
   &\le
& C \int_{T(\widetilde{\Delta})}\frac{f^2(X)}{\delta(X)}dX.
\end{eqnarray*}
Since $\sigma(\widetilde{\Delta}) \le
 c \sigma(\Delta)$ for any $Q\in \Delta$
$$ \frac{1}{\sigma(\Delta)} \int_{\Delta}{A}_{r}^2(f)(Q)d\sigma(Q) \lesssim \frac{1}{\sigma(\widetilde{\Delta}) } \int_{T(\widetilde{\Delta})}\frac{f^2(X)}{\delta(X)}dX  \lesssim C^2(f)(Q) C'\inf_{Q \in \Delta}C(F)(Q).$$
If $ \sigma(\{Q \in \Delta: \tau(Q) \ge
 r\}) < c_0 \sigma(B)$ then $ \sigma(\{Q \in \Delta: \tau(Q) < r\}) > (1-c_0) \sigma(\Delta)$ and
$$ \int_{\Delta}{A}_{r}^2(f)(Q)d\sigma(Q) \ge
 \int_{\Delta \cap \{\tau(Q) < r\}}{A^2}_{r}(f)(Q)d\sigma(Q) >
 {\Lambda}^2\int_{\Delta \cap \{\tau(Q) < r\}}C^2(f)(Q)d\sigma(Q)$$
$$ \ge
 {\Lambda}^2 \inf_\Delta C^2(f)(Q)\sigma(\Delta \cap \{\tau(Q) < r\}) \ge
 {\Lambda}^2 (1-c_0)\inf_\Delta C^2(f)(Q) \sigma(\Delta) $$
which would imply
$ {\Lambda}^2 (1-c_0)\inf_\Delta C^2(f)(Q) <
 C' \inf_{\Delta}C^2(f)(Q) $ or ${\Lambda}^2 (1-c_0)\le
 C'$ which is a contradiction if we take
$\Lambda$ large and $c_0=3/4$ fixed. This concludes the proof of the claim, thus that of Theorem \ref{theorem2}.
\end{proof}

\begin{preremark}
As in \cite{cms}, Theorem \ref{theorem2} can be used to identify the dual of $T^1$ with those $F$ for which $C(F) \in L^\infty(\sigma)
.$
\end{preremark}

\begin{preremark}
Note that if $1<p, q<\infty,$ are such that
$\frac{1}{p}+\frac{1}{q}=1$, $f \in T^p$, $g \in T^q$ then using (\ref{H4.1}) and H\"{o}lder's inequality we have that
$$ \int_{\Omega} \frac{f(X)g(X)}{\delta(X)}dX  \lesssim ||f||_{T^p}||g||_{T^q} .$$
Similarly (b) in Theorem  \ref{theorem2} ensures that
\begin{equation}\label{ac+1}
 \int_{\Omega} \frac{f(X)g(X)}{\delta(X)}dX  
\lesssim ||C(f)||_{L^p}||g||_{T^q}. 
\end{equation}
It will be proved in Theorem \ref{relAC} that for $2< p<\infty$ $A(f)\in L^p(\sigma)$ if and only if $C(f)\in L^p(\sigma)$.
\end{preremark}

As in \cite{cms},  we prove that the definition of tent spaces is independent of the aperture of the cone used.  The following proposition
is also crucial for the forthcoming results in Section \ref{theory_of_weights} (see in particular Remarks \ref{Nequiv} and \ref{N-hat-tilde}).

\begin{prop} \label{prop31}
Using the notation in (\ref{tt-square}) we have that for 
 $0<p<\infty$ 
 \begin{equation}\label{tt-a-1}
||A^{(\alpha)}(f)||_{L^p(\sigma)
} \thickapprox ||A(f)||_{L^p(\sigma)
}.
\end{equation}
\end{prop}

To prove Proposition \ref{prop31} we assume that $\alpha>1$. We note that in this case $A^{(\alpha)}(f)\ge A(f)$. We show that there exists a constant $C(\alpha, p)$ such that $||A^{(\alpha)}(f)||_{L^p} \le  C(\alpha, p) ||A(f)||_{L^p}.$ This proves (\ref{tt-a-1}) for $\alpha>1$. The case
$\alpha\le 1$ is proved the same way reversing the roles of $\alpha$ and 1.
The following lemma which is straightforward on Lipschitz domains requires a proof on a CAD.

\begin{lemma} \label{lem31}
For $f\in T^1$ and $\lambda>0$ the set $ \mathcal{F}=\{Q \in \partial \Omega: A^{(\alpha)}(f)(Q) \le
 \lambda\}$ is closed.
\end{lemma}

\begin{proof}
To prove that $\mathcal F^c=\{Q \in \partial \Omega:A^{(\alpha)}(f)(Q)> \lambda\}$ is open we show that give $Q\in \partial\Omega$ such that $A^{(\alpha)}(f)(Q)> \lambda$ there exist $\eta >0$ and $\epsilon >0$ such that if $P \in \partial \Omega$ with $|P-Q|<\epsilon\eta$, then
$$ \int_{\Gamma_{\alpha}(P)\backslash B(P, \eta)}\frac{f^2(X)}{\delta(X)^n}dX >\lambda^2  .$$

Since $A^{(\alpha)}(f)(Q)> \lambda$ there exists $\eta>0$ so that 
$$\int_{\Gamma_\alpha(Q)\setminus B(Q,\eta)}\frac{f^2(X)}{\delta(X)^n}dX>\left(\frac{A^{(\alpha}(f)(Q)+\lambda}{2}\right)^2 .$$
Observe that $$\left|\int_{\Gamma_{\alpha}(P)\backslash B(P, \eta)}\frac{f^2(X)}{\delta(X)^n}dX-\int_{\Gamma_{\alpha}(Q)\backslash B(Q, \eta)}\frac{f^2(X)}{\delta(X)^n}dX\right|\le
 \int_{D}\frac{f^2(X)}{\delta(X)^n}dX$$
where $D=\bigg(\Gamma_\alpha(Q)\setminus B(Q,\eta)\bigg)\triangle \bigg(\Gamma_\alpha(P)\setminus B(P,\eta)\bigg)$.
If $|X-Q|< (1+\alpha)\delta(X)$ and $|X-Q|\ge
 \eta$ then $\delta(X)\ge
 \frac{\eta}{1+\alpha}$. If $P\in B(Q,\epsilon\eta)$ and $X\notin \Gamma_\alpha(P)\setminus B(P,\eta)$ then $|X-Q|\ge
 (1+\alpha)(1-\epsilon)\delta(X)$. Thus we need to study sets of the form
\begin{equation}\label{setV}
V_P=\bigg\{X \in \Omega: |X-P|\ge
 \eta\ ;\ \delta(X) (1+\alpha)(1-\epsilon)\le
 |X-P|< (1+\alpha)\delta(X)\bigg\}
\end{equation}
for $P\in B(Q,\epsilon\eta)$ and prove that they have small $\mathcal{H}^n$-measure. 
Note that for $\epsilon<1/2$
\begin{equation}\label{hn-p-q}
V_P\subset V'_\epsilon=\bigg\{X \in \Omega: |X-Q|\ge
 \eta/2\ ;\ \delta(X) (1+\alpha)(1-\epsilon)^2\le
 |X-Q|< (1+\alpha)^2\delta(X)\bigg\}.
\end{equation}
Note that $D\subset V_P\cup V_Q\subset V'$.
We show that given $\alpha>0,$ and $\delta>0$ there exists $\beta>0$ such that 
\begin{equation}\label{smallV}
\mathcal{H}^n(\{X \in \Omega: |X-Q|\ge
 \eta/2\ ;\ \delta(X) (1+\alpha-\beta)\le
 |X-Q|\le
 (1+\alpha+\beta)\delta(X) \})<\delta,
\end{equation}
which ensures that  given $\alpha>0,$ and $\delta>0$ there exists $\epsilon>0$ such that  
$\mathcal{H}^n(V'_\epsilon)<\delta$.
Fix $\alpha >0$ and take $\beta>0$ small, such that $\alpha-\beta >\alpha/2.$ Consider the set
$$V=\Omega \cap \{|X-P|\ge
 \eta/2,\ \frac{1}{1+\alpha+\beta}\le
 \frac{\delta(X)}{|X-P|}\le
 \frac{1}{1+\alpha-\beta} \}\subset \Omega\backslash B(Q,\eta/2).$$
$F(X)=\frac{\delta(X)}{|X-Q|}-1$ is a non-positive Lipschitz function on $\Omega\backslash B(Q,\eta/2)$.
By the co-area formula we have
\begin{equation}\label{tt-hn}
\mathcal{H}^n(V)=\int^{0}_{-1} \left(\int_{F^{-1}(t)}\frac{1}{JF}\chi_V d \mathcal{H}^{n-1}\right)dt = \int^{\frac{1}{1+\alpha-\beta}-1}_{\frac{1}{1+\alpha+\beta}-1} \left(\int_{F^{-1}(t)}\frac{1}{JF} \chi_V d \mathcal{H}^{n-1}\right)dt.
\end{equation}
Since
$$
 \int^{0}_{-1}\left( \int_{F^{-1}(t)}\frac{1}{JF}\chi_{\Omega\backslash B(Q,\eta/2)} d \mathcal{H}^{n-1} \right) dt \le  \mathcal{H}^n(\Omega)\le C(\Omega)<\infty,
 $$
given $\alpha>0$ there exists $\beta>0$ small, such that
\begin{equation}\label{tt-hn1}
 \int^{\frac{1}{1+\alpha-\beta}-1}_{\frac{1}{1+\alpha+\beta}-1} \left(\int_{F^{-1}(t)}\frac{1}{JF} \chi_{\Omega\backslash B(Q,\eta/2)} d \mathcal{H}^{n-1}\right)dt< \delta.
 \end{equation}
Note that (\ref{hn-p-q}), (\ref{tt-hn}) and (\ref{tt-hn1}) yield (\ref{smallV}).

Since $f\in L^2(\Omega)$ given $\epsilon'>0$ there exists $\delta>0$ so that if  (\ref{smallV}) holds then
$$\int_D\frac{f^2(X)}{\delta(X)^n}dX\le
2^n \eta^{-n}\int_Df^2(X)dX<\eta^{-n}\epsilon'$$
thus
$$\int_{\Gamma_\alpha(P)\setminus B(P,\eta)}\frac{f^2(X)}{\delta(X)^n}dX>\int_{\Gamma_\alpha(Q)\setminus B(Q,\eta)}\frac{f^2(X)}{\delta(X)^n}dX-\eta^{-n}\epsilon'
\left(\frac{A^{(\alpha}(f)(Q)+\lambda}{2}\right)^2-\eta^{-n}\epsilon'.$$
Since $A^{(\alpha)}(f)(Q)> \lambda$ we can choose $\epsilon'>0$ so that
$$(A^{(\alpha)}(f)(P))^2\ge
\int_{\Gamma_\alpha(P)\setminus B(P,\eta)}\frac{f^2(X)}{\delta(X)^n}dX>\lambda^2.$$
\end{proof}

\emph{Proof of Proposition \ref{prop31}:} We fix $\lambda>0$ and let $$\mathcal{F}=\{Q \in \partial \Omega: A(f)(Q) \le
 \lambda\},\
O= \mathcal{F}^{C}=\{Q \in \partial \Omega: A(f)(Q) > \lambda\}.$$ Since $\mathcal F$ is a closed set  $\mathcal{F}^{*}_\gamma\subset \mathcal{F}$ (see Proposition \ref{size-g-dens-set}). Let  $O^{\ast}= (\mathcal{F}^{\ast}_\gamma)^{C}.$ 
Since $\sigma$ is Ahlfors regular it is doubling and (\ref{eqn-size-g}) ensures that
\begin{eqnarray*}
  \sigma\bigg(\{Q \in \partial \Omega: A^{(\alpha)}(f)(Q) > \lambda\}\bigg) &\le
& \sigma\bigg( (\mathcal{F}_\gamma^{\ast})^{C}\}\bigg)+ \sigma\bigg(\{Q \in \mathcal{F}_\gamma^{\ast}: A^{(\alpha)}(f)(Q) > \lambda\}\bigg) \\
   &\le
& \sigma(O^{\ast})+ \frac{C_{\alpha, \gamma}}{\lambda^2}\int_{\mathcal{F}} (A(f)(Q))^2 d\sigma(Q)\\
&\le &C \sigma\bigg(\{Q \in \partial \Omega: A(f)(Q) > \lambda\}\bigg) +\frac{C}{\lambda^2}\int_{\{ A(f) \le
 \lambda\}} (A(f)(Q))^2 d\sigma(Q).
\end{eqnarray*}

Multiplying both sides by $p \lambda^{p-1}$ and integrating,  we obtain
\begin{eqnarray}\label{eqn4.10t}
 p \int_{0}^{\infty}\sigma(\{Q \in \partial \Omega: A^{(\alpha)}(f) > \lambda\})\lambda^{p-1} d\lambda& \le&
C p \int_{0}^{\infty}\sigma(\{Q \in \partial \Omega: A(f) > \lambda\})\lambda^{p-1} d\lambda\\
&+&
+C p \int_{0}^{\infty}\lambda^{p-3}\int_{\{ A(f) \le
 \lambda\}} (A(f)(Q))^2 d\sigma d\lambda.\nonumber
 \end{eqnarray}
If $p<1$, Fubini's Theorem applied to the second term yields
\begin{eqnarray}
  \int_{0}^{\infty}\lambda^{p-3}\int_{\{ A(f) \le
 \lambda\}} (A(f)(Q))^2 d\sigma d\lambda &=& \int_{0}^{\infty}\int_{\partial \Omega} \chi_{\{ Af \le
 \lambda\}}(A(f)(Q))^2 \lambda^{p-3} d\sigma d\lambda \\
   &=& \int_{\partial \Omega}(Af(Q))^2 \int_{Af(Q)}^{\infty}  \lambda^{p-3} d\lambda d\sigma \nonumber\\
   &=& C_p\int_{\partial \Omega}(Af(Q))^2 (Af(Q))^{p-2}  d\sigma(Q).\nonumber
\end{eqnarray}
Thus for $p<2$ $ ||A^{(\alpha)}f||^p_{L^p} \le
 C ||Af||^p_{L^p}.$

For the case $p \ge
 2$, if $\frac{1}{r}+\frac{2}{p}=1$ observe that
\begin{equation}\label{tt-dual}
||A^{(\alpha)}f||^2_{p} = \sup_{\psi}\left\{\int_{\partial \Omega}(A^{(\alpha)}f)^2 \psi d\sigma:\ \psi \in L^{r}(\sigma),\ ||\psi||_{L^r} \le 1\right\}
\end{equation}
Note that
$$ \int_{\partial \Omega}(A^{(\alpha)}f)^2 \psi d\sigma \le
 \left(\int_{\partial \Omega}(A^{(\alpha)}f)^p\right)^{2/p}\left(\int_{\partial \Omega}\psi^r \right)^{1/r}.$$
Also if $X \in \Gamma_{\alpha}(Q)$ then $|X-Q| < (1+\alpha)\delta(X)$ and $|Q-Q_X| < (2+\alpha)\delta(X)$ therefore
\begin{eqnarray}\label{star24}
  \int_{\partial \Omega}(A^{(\alpha)}(f)(Q))^2 \psi(Q) d\sigma(Q) &=& \int_{\partial \Omega} \int_{\Gamma_{\alpha}(Q)}f^2(X)\frac{dX}{\delta(X)^n} \psi(Q) d\sigma(Q)\\
   &=& \int_{\partial \Omega} \int_{\Omega}\frac{f^2(X)}{\delta(X)^n} \chi_{\Gamma_{\alpha}(Q)}(X) \chi_{\Delta(Q_X, (2+\alpha)\delta(X))}(Q)\psi(Q) dX d\sigma \nonumber\\
   &=& \int_{\Omega} \left(\int_{\partial \Omega} \chi_{\Delta(Q_X, (2+\alpha)\delta(X))}(Q)\psi(Q) \chi_{\Gamma_{\alpha}(Q)}(X)  d\sigma(Q) \right) \frac{f^2(X)}{\delta(X)^n}dX \nonumber \\
   &\le
& \int_{\Omega} \frac{f^2(X)}{\delta(X)^n} \int_{\Delta(Q_X, (2+\alpha)\delta(X))} \psi(Q) d\sigma(Q) dX \nonumber\\
   &\le
& C_{\alpha}\int_{\Omega} M_{(2+\alpha)\delta(X)} \psi(Q_X) \frac{f^2(X)}{\delta(X)} dX \nonumber
\end{eqnarray}
where $$ M_s \psi(P)=\frac{1}{s^{n-1}}\int_{\Delta(P, s)} \psi(Q)d\sigma(Q).$$
Let $$ \psi^{\ast}(P)= \sup_{s>0}\frac{1}{s^{n-1}}\int_{\Delta(P, s)} \psi(Q)d\sigma(Q)$$
then $$ M_s(M_{\beta s} \psi)(P)=\frac{1}{s^{n-1}}\int_{\Delta(P, s)} \frac{1}{(\beta s)^{n-1}}\int_{\Delta(Q,\beta s)} \psi(X)d\sigma(X)d\sigma(Q).$$
If $\beta >1$ and $Q \in \Delta(P, s)$ observe that $\Delta(P, (\beta-1)s) \subset \Delta(Q, \beta s)$ and
\begin{eqnarray*}
 M_s(M_{\beta s} \psi)(P)& \ge
 & \frac{1}{s^{n-1}}\int_{\Delta(P, s)} \frac{1}{(\beta s)^{n-1}}\int_{\Delta(P,(\beta-1)s)} \psi(X)d\sigma(X)d\sigma(Q) \\
   &\ge
 & C \frac{1}{(\beta s)^{n-1}}\int_{\Delta(P,(\beta-1)s)} \psi(X)d\sigma(X) \\
   &\ge
 & C_{\beta}M_{(\beta-1)s} \psi.
\end{eqnarray*}
That is for $\beta >1$
\begin{equation}\label{tt4.6}
M_{(\beta-1)s} \psi \le
 C_{\beta}M_s(M_{\beta s} \psi) \le
 C_{\beta}M_s \psi^{*}
\end{equation}
since $M_{\beta s} \psi \le
 C \psi^{\ast} $. Plugging (\ref{tt4.6}) into (\ref{star24}) with $s=\delta(X)$, $\beta-1=2+\alpha$ we obtain
\begin{eqnarray*}
  \int_{\partial \Omega}(A^{(\alpha)}(f)(Q))^2 \psi(Q) d\sigma(Q) & \lesssim & C_{\alpha} \int_{\Omega}M_{(2+\alpha)\delta(X)} \psi(Q_X)\frac{f^2(X)}{\delta(X)} dX\nonumber\\
   & \le
 & C_{\alpha} \int_{\Omega}M_{\delta(X)} \psi^{\ast}(Q_X)\frac{f^2(X)}{\delta(X)} dX\nonumber\\
   &\le
&  C_{\alpha} \int_{\Omega}\frac{f^2(X)}{\delta(X)^n} \int_{\Delta(Q_X, \delta(X))}\psi^{\ast}(Q)d\sigma(Q) dX\nonumber\\
   &\le
&  C_{\alpha} \int_{\Omega}\int_{\partial \Omega}\frac{f^2(X)}{\delta(X)^n}\chi_{\Delta(Q_X,\delta(X))}(Q) \chi_{\Gamma(Q)}(X) \psi^{\ast}(Q)d\sigma(Q)dX \nonumber\\
   & \le
 & C_{\alpha}\int_{\partial \Omega} \left(\int_{\Gamma(Q)}\frac{f^2(X)}{\delta(X)^n}dX \right) \psi^{\ast}(Q)d\sigma(Q)\nonumber  \\
   &\le
 &  C_{\alpha}\int_{\partial \Omega} (Af)^2(Q)\psi^{\ast}(Q)d\sigma(Q)\nonumber \\
   &\le
&  C_{\alpha}||Af||^2_{L^p}||\psi^{\ast}||_{L^r} \le
 C_{\alpha}||Af||^2_{L^p}||\psi||_{L^r}\le
 C_{\alpha}||Af||^2_{L^p}\nonumber
\end{eqnarray*}
where we have used the fact that if $Q\in \Delta(Q_X,\delta(X))$ then $|X-Q|\le
 |Q-Q_X|+|X-Q_X|\le
 2\delta(X)$ and the fact that the maximal function of $\psi$ is bounded in $L^r(\sigma)$.
Taking the supremum over all $\psi$ yields that  $ ||A^{(\alpha)}f||_{L^p} \le
 C_{\alpha}||Af||_{L^p}$ for $2\le p < \infty.$
\hfill $\Box$

\begin{defn}
A $T^1$ atom is a function $a(X)$ which is supported in $T(\Delta)$, $\Delta=B(Q_0,r)\cap\partial \Omega$ for $Q_0\in\partial \Omega$ and
\begin{equation}\label{tt4.9'}
\int_{T(\Delta)}a^2(X)\frac{dX}{\delta(X)}\le
 \frac{1}{\sigma(\Delta)}.
\end{equation}
\end{defn}
Observe that if $a$ is supported in $B(Q_0,r)\cap\Omega$ then $A(a)$ given by
$$A(a)=\bigg(\int_{\Gamma(Q)}\frac{a^2(X)}{\delta(X)^n}dX\bigg)^{1/2}$$
is supported in $\Delta(Q_0,3r)$. Indeed, if $X\in\Gamma(Q)$ and $|Q-Q_0|\ge
 3r$ then $|X-Q_0|\ge
 r$ which gives $a(X)=0$ thus $A(a)(Q)=0$. Using (\ref{tt4.9'}) we estimate,
\begin{eqnarray}\nonumber
                 \int_{\partial \Omega}A^2(a)(Q)d\sigma(Q) &=& \int_{\partial\Omega}\int_{\Gamma(Q)}\frac{a^2(X)}{\delta(X)^n}dXd\sigma(Q)=\int_{\partial\Omega}\int_{\Omega}\chi_{\Gamma(Q)}(X)\frac{a^2(X)}{\delta(X)^n}dXd\sigma(Q) \\
                  &\le
 & \int_{\Omega}\int_{\partial\Omega}\chi_{\Delta(Q_X,3\delta(X))}(Q)\chi_{\Gamma(Q)}(X)\frac{a^2(X)}{\delta(X)^n}d\sigma(Q)dX \nonumber \\
                  &\le
& C\int_{\Omega}\frac{a^2(X)}{\delta(X)}dX=C\int_{T(\Delta)}\frac{a^2(X)}{\delta(X)}dX\lesssim \frac{1}{\sigma(\Delta)}   \nonumber        \end{eqnarray}
which yields
$$\int_{\partial\Omega}A(a)(Q)d\sigma(Q)\le
 \bigg(\int_{\partial\Omega}A^2(a)(Q)d\sigma(Q)\bigg)^{1/2}\bigg(\sigma(\Delta)\bigg)^{1/2}\le
 C_n$$
where $C_n$ depends on the Ahlfors regularity constant. Thus, if $a$ is a $T^1$ atom, then $a\in T^1$ and $||a||_{T^1}=||A(a)||_{L^1}\le
 C_n$.

We are now ready to prove the duality relation between $T^1$ and $T^\infty$ (see (\ref{N1}) and (\ref{Tinfty})).

\begin{thm}\label{thm32}
If $G\in(T^1)^*$ then there exists a $g\in T^\infty$ such that for every $f\in T^1$
$$|G(f)|\simeq\bigg|\int_{\Omega}f(x)g(x)\frac{dX}{\delta(X)}\bigg|.$$
\end{thm}

\begin{proof} We first notice that Theorem \ref{theorem2} shows that every $g\in T^\infty$ induces an element in $(T^1)^*$.
Let $G\in (T^1)^*$ and note that if $K$ is a compact set in $\Omega$ and $f$ is supported in $K$ with $f\in L^{2}(K)$ then $f\in T^1$. 
First we consider $K=\overline{B(X_0,r)}$ with ${\rm{dist}(K,\partial\Omega)}\ge
 \epsilon_0$. If $X\in \Gamma(Q)\cap \overline{B(X_0,r)}$ then
\begin{eqnarray*}
  |Q-Q_{X_0}| &\le
& |Q-X|+|X-X_0|+|X_0-Q_{X_0}| < 2\delta(X)+r+\delta(X_0)\\
   &\le
& 2|X-Q_{X_0}|+r+\delta(X_0)\le
 2|X-X_0|+2|X_0-Q_{X_0}|+r+\delta(X_0)\\
 &\le&
 3r+3\delta(X_0)<4r+4\delta(X_0)
\end{eqnarray*}

Thus
\begin{eqnarray}
\int_{\partial\Omega}A(f)(Q)d\sigma(Q) &=&\int_{\partial\Omega}\bigg(\int_{\Gamma(Q)\cap \overline{B(X_0,r)}}\frac{f^2(X)}{\delta(X)^n}dX\bigg)^{1/2}d\sigma(Q)\nonumber\\
&\le&
 \int_{\Delta(Q_{X_0,4r+4\delta(X_0)})}\bigg(\int_{\Gamma(Q)\cap B(X_0,r)}\frac{f^2(X)}{\delta(X)^n}dX\bigg)^{1/2}d\sigma(Q)\nonumber\\
&\lesssim& \epsilon_0^{-n/2}(r+\delta(X_0))^{n-1}\bigg(\int_{\overline{B(X_0,r)}}f^2(X)\bigg)^{1/2}.\nonumber
\end{eqnarray}

If $K\subseteq \bigcup_{i=1}^m\overline{B(X_i,r_i)}$ with $\overline{B(X_i,r_i)}\subset\subset \Omega$, ${\rm{dist}}(B(X_i,r_i),\partial\Omega)\ge
 r_i>\frac{1}{2}{\rm{dist}}(K,\partial\Omega)=\varepsilon_K$ and $r_i\le
{\rm{diam}}\,K$ then $\delta(X_i)\le
{\rm{diam}}\,K+{\rm{dist}}(K,\partial\Omega)$ and
$$\int_{\partial\Omega}A(f)(Q)d\sigma(Q)\le
 \int_{\partial\Omega}\sum_{i=1}^m\bigg(\int_{\Gamma(Q)\cap B(X_i,r_i)}\frac{f^2(X)}{\delta(X)^n}dX\bigg)^{1/2}d\sigma(Q)$$
$$\le
 \sum_{i=1}^mC_K\varepsilon_K^{-n/2}\bigg(\int_{\overline{B(X_i,r_i)}}f^2(X)dX\bigg)^{1/2}(r_i+\delta(X_i))^{n-1}\le
 C_K||f||_{L^2(K)},$$
therefore for $f$ compactly supported
$$|G(f)|\le
 C||f||_{T^1}\le
 C_K||f||_{L^2(K)}.$$
Thus $G$ induces a bounded linear functional on $L^2(K)$ which can be represented by a $g_K\in L^2(K)$. Taking an increasing family of such $K$ which exhaust $\Omega$, gives us a function $g\in L^2_{{\rm{loc}}}(\Omega)$ and 
\begin{equation}\label{4.101}
G(f)=\int_\Omega f(X)g(X)\frac{dX}{\delta(X)}
\end{equation}
whenever $f\in T^1$ with compact support in $\Omega$.

Let $a\in T^1$ be an atom supported on $T(\Delta)$ then
$$|G(a)|\le
 ||G||||a||_{T^1}\le
 C_n||G||.$$
For the atom
$$a_m=g\chi_{T(\Delta)\cap\{\delta(X)>r/m\}} \bigg(\sigma(\Delta))\int_{T(\Delta)\cap\{\delta(X)>r/m\}}\frac{g^2(X)}{\delta(X)}dX\bigg)^{-1/2}$$
where $\Delta=B(Q,r)\cap \partial \Omega$, we have
$$|G(a_m)|=\bigg(\frac{1}{\sigma(\Delta)}\int_{T(\Delta)\cap\{\delta(X)>r/m\}}\frac{g^2(X)}{\delta(X)}dX\bigg)^{1/2}\le
 C_n||G||$$
and if $m\rightarrow \infty$
$$\bigg(\frac{1}{\sigma(\Delta)}\int_{T(\Delta)}\frac{g^2(X)}{\delta(X)}dX\bigg)^{1/2}\le
 C_n||G||$$
which shows that $C(g)\in L^\infty$. This representation of $G$ (as in (\ref{4.101}) can be extended to all of $T^1$ since the subspace of the functions with compact support is dense in $T^1$.
\end{proof}

\section{Duality of $T^p$ spaces}
The main purpose of the present section is study the dual spaces to $T^p$ spaces for $1< p< \infty$. The main result is contained in the following theorem.

\begin{thm} \label{thm3} Let $1< p< \infty$. The dual of $T^p$ is the space $T^q$ with $\frac{1}{p}+\frac{1}{q}=1$. More precisely if
$G \in (T^p)^{\ast}$ then there exists  $g \in T^q$ such that for every $f \in T^p$
$$ G(f)= \int_{\Omega}f(X)g(X)\frac{dX}{\delta(X)}\ \ \hbox{ and }\ \  ||G||\simeq ||g||_{T^q}
.$$
\end{thm}

\begin{proof}
As in \cite{cms} we first study the case $p=2$. Note that from Lemmas \ref{four} and \ref{six} there exists a constant $C_n$ such that
\begin{equation}\label{5.tt.1}
 C_n^{-1}\int_{\Omega}f^2(X)\frac{dX}{\delta(X)} \le
 {||f||^2}_{T^2} \le
 C_n\int_{\Omega}f^2(X)\frac{dX}{\delta(X)} 
 \end{equation}
 Given $g\in T^2$ the proof of Theorem \ref{theorem2} (see (\ref{H4.1})) yields that the operator $G$ defined by
$$ G(f)= \int_{\Omega}f(X)g(X)\frac{dX}{\delta(X)} $$
satisfies
\begin{equation}\label{5.tt.2}
|G(f)|
\le
 C \int_{\partial \Omega}A(f)(Q)A(g)(Q)d\sigma(Q) \le C\|f\|_{T^2}\|g\|_{T^2}.
 \end{equation}
Thus $G\in ({T^2})^{\ast}$.

Consider $T^2$ with the norm induced by the inner product
$$ \langle f,g\rangle= \int_{\Omega}\frac{f(X)g(X)}{\delta(X)}dX$$
then $(T^2,\langle,\rangle)$ is a Hilbert space. Given $G\in ({T^2})^{\ast}$ by the Riesz Representation Theorem there exists $g \in (T^2, \langle\ ,\ \rangle)$ such that $$ G(f)= \int_{\Omega}f(X)g(X)\frac{dX}{\delta(X)}. $$ By (\ref{5.tt.1}) and (\ref{5.tt.2}) we have that
$$ C_n^{-1}||g||_{T^2} \le
 ||G|| \le
  C_n ||g||_{T^2}.$$

Consider the case $1<p<2$. Let $G\in ({T^p})^{\ast}$. For $f \in L^2(\Omega)$ with compact support $K\subset\subset\Omega$,let
$ S =\{Q \in \partial \Omega: \Gamma(Q)\cap K \neq \emptyset \}$ and $\varepsilon_K={\rm{dist}}(K,\partial\Omega)$
then
$|G(f)|\le  ||G||||f||_{T^p}$
and
\begin{equation}\label{H-E}
||f||^{p}_{T^p}\le
 \int_S\bigg(\int_{\Gamma(Q)\cap K}\frac{f^2(X)}{\delta(X)^n}dX\bigg)^{p/2}d\sigma\le
 \bigg(\int_K\frac{f^2(X)}{\delta(X)^n}dX\bigg)^{p/2}\sigma(S)\le
  C_K \bigg(\int_{K}\frac{f^2(X)}{\delta^2(X)}dX\bigg)^{1/2}
\end{equation}
Thus $|G(f)|\le
 C_K||f/\delta||_{T^{2}}$. By the Riesz Representation Theorem there exists $g$ which is locally in $L^2(\Omega)$ such that
$$
G(f)= \int_{\Omega}\frac{f(X)g(X)}{\delta(X)}dX
$$
 whenever $f \in L^2(\Omega)$ and has compact support in $\Omega$. Note that for every $K\subset\subset\Omega$, $f\in L^2(\Omega)$ by Theorem \ref{theorem2} (b)
\begin{eqnarray}\label{tt5.1}
  \int_{K}f^2(X)dX &\le
 & C_K \int_{\Omega}\frac{(f\chi_K)(X)(f\chi_K)(X)}{\delta(X)}dX \\
                   &\le
 & C_K \int_{\partial\Omega}A(f\chi_K)(Q)C(f\chi_K)(Q)d\sigma(Q)  \nonumber\\
                   &\le
 & C_K \bigg(\int_{\partial\Omega} A^p(f\chi_K)(Q) d\sigma(Q)\bigg)^{1/p} \bigg(\int_{\partial\Omega} C^q(f\chi_K)(Q) d\sigma(Q)\bigg)^{1/q} 
\end{eqnarray}
where $1/p+1/q=1$. By the definition of $C(f\chi_K)$ and if $\delta_K(Q)$ denotes the distance of Q to the set $K$, 
$$ 
C(f\chi_K)(Q)=\sup_{Q \in \Delta}\left(\frac{1}{\sigma(\Delta)} \int_{T(\Delta)}\frac{(f\chi_K)^2(X)}{\delta(X)}dX\right)^{1/2} \lesssim \left(\int_{T(\Delta)}{f^2(X)}{} dX \right)^{1/2}\varepsilon_K^{-\frac{}{2}}$$
and
\begin{equation}\label{tt5.2}
\int_{\partial \Omega}(C(f\chi_K)(Q))^qd\sigma(Q) \le
 C_K \left(\int_K f^2(X)dX\right)^{q/2}.
\end{equation}
Combining (\ref{tt5.1}) and (\ref{tt5.2}) we have for $p>1$
$$ \left(\int_K f^2(X)dX\right)^{1/2} \le  
 C(K,p) \|f\chi_K\|_{T^p}.$$

Observe that the set $\{f \in T^p:\ f{\rm{\ is\  compactly\ supported\ in}} \ \Omega\}$ is dense in $T^p$. Indeed, let us choose an increasing family of compact sets $\{K_n\}$ which exhaust $\Omega$. For $f \in T^p$ consider $f_m=f \chi_{K_m} \in T^p$. Then
$$ ||f_m-f ||_{T^p}=||A(f_m-f) ||_{L^p(\sigma)}$$ 
and
$$
||A(f_m-f) ||_{L^p(\sigma)
}=\left(\int_{\partial \Omega}\bigg[\int_{\Gamma(Q)}\frac{(f_m-f)^2(X)}{\delta(X)^n}dX \bigg]^{p/2}d\sigma\right)^{1/2}
=\left(\int_{\partial\Omega}( E_m(f)(Q))^{p/2}d\sigma(Q)\right)^{1/2}
$$
where $$E_m(f)(Q)=\int_{\Gamma(Q)\cap K_m^c}\frac{f^2(X)}{\delta(X)^n}dX. $$
Note that $0 \le
 E_m(f)(Q) \le
 E_{m-1}(f)(Q)\le
 \ldots \le
 E(f)(Q)=\int_{\Gamma(Q)}\frac{f^2(X)}{\delta(X)^n}dX$. 
 Since $f\in T^p$, $E_m(f) \rightarrow 0$ a.e. $Q \in \partial \Omega$ and by the Dominated
Convergence Theorem 
$$ 
\lim_{m\rightarrow \infty}\int_{\partial \Omega}E_m(f)(Q) d\sigma(Q) =0\ \ \hbox{ and }\ \ 
\lim_{m\rightarrow \infty}||A(f_m-f) ||_{L^p}=0.
$$

We claim that for $g$ as above there exists $C'>0$ such that 
\begin{equation}\label{star48}
  ||A(g_K) ||_{L^q} \le
 C'||G||
\end{equation}
where  $g_K=g\chi_K$, and $K$ is any compact subset of $\Omega$.
The key point is that $C'$ is a constant independent of the choice of the set $K$. Note that this ensures that 
\begin{equation}\label{5.tt.7}
\|g\|_{T^q}\le C' \|G\|.
\end{equation}

Let $r$ denote the exponent dual to $q/2$, $\frac{1}{r}+\frac{2}{q}=1$. Then, as in the proof of Proposition \ref{prop31} (see (\ref{tt-dual})),
\begin{equation}\label{5.tt.6}
||A(g_K)||_{L^q}^2=\sup_{\psi}\left\{\int_{\partial\Omega}(A(g_K)(Q))^2\psi(Q)d\sigma(Q): \psi\ge 0, \ \psi\in L^r(\partial\Omega),\ \
||\psi||_{L^r}\le 1\right\}
\end{equation}

 As in the proof of Proposition \ref{prop31} to obtain (see (\ref{star24}))
\begin{equation}\label{5.tt.4}
\sup_{\psi}\int_{\partial\Omega}(A(g_K)(Q))^2\psi(Q)d\sigma(Q)\le
 C\sup_{\psi}\int_{\Omega}\frac{g_K^2(X)}{\delta(X)}M_{3\delta(X)}\psi(Q_X)dX = C\sup_{\psi} G(h_\psi),
 \end{equation}
where $h_\psi(X)=g_K(X)M_{3\delta(X)}\psi(Q_X)$. Note that
$$M_{3\delta(X)}\psi(Q_X)=\frac{1}{(3\delta(X))^{n-1}}\int_{\Delta(Q_X,3\delta(X))}\psi(Y)d\sigma(Y)\le
 CM_{6\delta(X)}\psi(Q)\le
 CM\psi(Q)$$
\begin{equation}\label{5.tt.3}
A(h_\psi)(Q)=\bigg(\int_{\Gamma(Q)}\frac{g_K^2(X)M_{3\delta(X)}^2\psi(Q_X)}{\delta(X)^n}dX\bigg)^{1/2}
\lesssim M\psi(Q)A(g_K)(Q)
\end{equation}
where $M\psi$ denotes the maximal function of $\psi$.

Integrating (\ref{5.tt.3}), noting that $\frac{p}{r}+\frac{p}{q}=1$ and applying H\"{o}lder's inequality,   we conclude that
\begin{eqnarray}\label{5.tt.8}
\|h_\psi\|_{T^p}&=&\left(
  \int_{\partial\Omega}(A(h_\psi)(Q))^pd\sigma(Q) \right)^{1/p}\\
   &\lesssim& \bigg(\int_{\partial\Omega}(M\psi(Q))^rd\sigma(Q)\bigg)^{1/r}\bigg(\int_{\partial\Omega}(A(g_K)(Q))^qd\sigma(Q)\bigg)^{1/q}
   \nonumber\\
 & \lesssim& ||M\psi||_{L^r}||A(g_K)||_{L^q} \lesssim ||\psi||_{L^r}||A(g_K)||_{L^q} \nonumber
\end{eqnarray}

Since $h_\psi$ is compactly supported $h_\psi\in T^p$
\begin{equation}\label{5.tt.5}
|G(h_\psi)|\le
 ||G||||h_\psi||_{T^p}\lesssim ||G||\cdot||\psi||_{L^r}||g_K||_{T^q}
\end{equation}

Combining (\ref{5.tt.4}), (\ref{5.tt.5}) and (\ref{5.tt.8})
$$||A(g_K)||^2_{L^q}=\|g_K\|_{T^q}^2\le
 C ||G||\cdot ||g_K||_{T^q}.$$
where $C$ is independent of $K$. This proves (\ref{star48}) and (\ref{5.tt.7}). The density of compactly supported functions $f$ in $T^p$, 
ensures that $G(f)=\int_{\Omega}\frac{f(X)g(X)}{\delta(x)}dX$. Applying H\"older's inequality to (\ref{5.tt.2}) we conclude that $\|G\|\le \|g\|_{T^q}$. This combined with (\ref{5.tt.7}) guarantees that $\|G\|\simeq\|g\|_{T^q}$, which completes the proof in the case $1<p<2$.

To prove Theorem \ref{thm3} for any $p\in (2,\infty)$, it is enough to show that  for $1<p<2$, $T^p$ is reflexive.
By the Eberlein-Smulyan Theorem (see \cite{y}) it is enough to show that whenever $f_n \in T^p,$ $||f_n||_{T^p} \le
 1$ there exists a subsequence which converges weakly in $T^p$. If $\{f_n\} \in T^p$ with $||f_n||_{T^p} \le
 1,$ we have $$ \left(\int_{K}f^2_n(X)dX\right)^{1/2} \le
 C_K ||f_n||_{T^p} \le
 C_K$$
therefore taking a compact exhaustion of $\Omega$ we show that there exists a subsequence $\{f_{n_j}\} $ such that $f_{n_j} \rightharpoonup f$ in $L^2(K),$ for all $K \subset\subset \Omega$. Let $G\in (T^p)^\ast$, by the proof above there exists $g\in T^q$,
where $\frac{1}{p}+\frac{1}{q}=1$ such that $G(f)=\int_{\Omega}\frac{f(X)g(X)}{\delta(x)}dX$. 
Given $\varepsilon >0$ there exists a compact set K such that $ ||A(g-g_K)||_{T^q}=||g-g_K||_{T^q}<\varepsilon$ where $g_K=g\chi_K$.
Note that
\begin{eqnarray*}
  G(f_{n_j})-G(f_{n_i}) &=& \int_{\Omega}\frac{(f_{n_j}-f_{n_i})(X)}{\delta(X)}g(X) dX \\
  &=& \int_{\Omega}\frac{(f_{n_j}-f_{n_i})(X)}{\delta(X)} g_K(X) dX +\int_{\Omega}\frac{(f_{n_j}-f_{n_i})(X)}{\delta(X)}(g-g_K)(X) dX.
\end{eqnarray*}
Since  $f_{n_i} \rightharpoonup f$ in $L^2(K)$ and $ g_K(X)/\delta(X) \in L^2(K)$,
$$ \int_{\Omega}\frac{(f_{n_j}-f_{n_i})(X)}{\delta(X)} g_K(X) dX \longrightarrow 0
\ \ \hbox{ as }\ \ i,j \rightarrow \infty.
$$
Also by (\ref{H4.1})
\begin{eqnarray*}
  \bigg|\int_{\Omega}\frac{(f_{n_j}-f_{n_i})(X)}{\delta(X)}(g-g_K)(X) dX\bigg| 
     &\lesssim & \int_{\partial \Omega} A(f_{n_j}-f_{n_i})(Q)A(g-g_K)(Q)d\sigma(Q) \\
   &\le
 &  ||f_{n_j}-f_{n_i}||_{T^p}||g-g_K||_{T^q}< 2\varepsilon.
\end{eqnarray*}
Thus $\{G(f_{n_i})\}$ converges, which ensures that $\{f_{n_i}\}$ converges weakly in $T^p$.
\end{proof}

\section{Relation between integrals on cones $A$ and Carleson's function $C$}
\renewcommand{\thesection}{\arabic{section}}
In this section we study, as in \cite{cms}, the relation between the functionals $A$ and $C$. We show that if $2<p< \infty$ then
$||A(f)||_{L^p} \simeq ||C(f)||_{L^p}. $ 

\begin{thm} \label{relAC}
\begin{description}
  \item[a)] If $0<p< \infty$ then
$$||A(f)||_{L^p} \le
 C_p||C(f)||_{L^p}.$$
  \item[b)] If $2<p\le
 \infty$ then
$$||C(f)||_{L^p} \le
 C_p||A(f)||_{L^p}.$$
\end{description}

\end{thm}

In the proof of Theorem \ref{relAC} uses the following ``good-$\lambda$" inequality.

\begin{lemma} \label{fixaper}
There exist a fixed aperture $\alpha>1$ and a constant $C>0$ so that for $0< \gamma \le
 1$ and $0<\lambda< \infty$
 \begin{equation}\label{6.1}
    \sigma(\{Q \in \partial \Omega: A(f)(Q)>2\lambda ;\ C(f)(Q) \le
 \gamma\lambda\}) \le
 C{\gamma}^2 \sigma(\{Q \in \partial \Omega: A^{(\alpha)}(f)(Q)>\lambda\}).
 \end{equation}
\end{lemma}
\begin{proof}
Let $\bigcup Q_k$ be a Whitney decomposition of $\{A^{(\alpha)}(f)(Q)>\lambda\}$ as in Lemma \ref{one}.
For each $k$, there exists $Y_k \in \{A^{(\alpha)}(f)(Q)\le
\lambda\}$ such that $\dist(Y_k, Q_k) \le
 c_0 {\rm{diam}}\,
Q_k.$
Since $\alpha>1$ then $A ^{(\alpha)}(f)\ge
 A(f)$ and the set $\{A(f)(Q)>2\lambda\}$ is contained in the set $\{A^{(\alpha)}(f)(Q)>\lambda\}.$\\
To prove (\ref{6.1}) it enough to show that
$$ \sigma(\{X\in Q_k: A(f)(X)>2\lambda ;\ C(f)(X) \le
 \gamma\lambda\}) \le
 c {\gamma}^2 \sigma(Q_k).$$
The construction in Lemma \ref{one} for $\{A^{(\alpha)}(f)(Q)\le
\lambda\}$ yields a family of balls $\{B_k\}$ such that $B_k=B(X_k, \frac{1}{24}d(X_k)),$
$B^{*}_k=B(X_k, \frac{1}{2}d(X_k))$, $B^{**}_k=B(X_k, 2d(X_k))$ and $B_k \subset Q_k \subset B^{*}_k$ for $X_k \in \{A^{(\alpha)}(f)>\lambda\} $ and
$d(X_K)=\dist(X_K, \{A^{(\alpha)}(f)\le
 \lambda\}).$
Nota that $\frac{1}{12}d(X_k)\le
 {\rm{diam}}\,Q_k \le
 d(X_k)= : 2r_k$ and that there exists $Y_k\in \{A^{(\alpha)}(f)(Q)\le
\lambda\}$ such that $\dist(Y_k, Q_k)\le
 4 r_k$. If $P \in Q_k$ then $|P-Y_k|\le
 5 d(X_k).$ Define $f=f_1+f_2$ where
$$\left\{
  \begin{array}{ll}
    f_1(X)=f(X)\chi_{\{\delta(X) \ge
 r_k\}} &  \\
    f_2(X)=f(X)\chi_{\{\delta(X) < r_k\}}. &
  \end{array}
\right.$$
Note that $A(f) \le
 A(f_1)+A(f_2).$ For $P \in Q_k$, $|P-Y_k|\le 5r_k$, where $Y_k\in \{A^{(\alpha)}(f)(Q)\le
\lambda\}$ 
 and
$$ A(f_1)(P)^2=  \int_{\Gamma(P)\cap\{\delta(X) \ge
 r_k\}} \frac{f^2(X)}{\delta(X)^n}dX \le \int_{\Gamma_6(Y_k)\cap\{\delta(X) \ge
 r_k\}} \frac{f^2(X)}{\delta(X)^n}dX
 $$
 Thus for $\alpha \ge
 6$ and $P\in Q_k$ we obtain $$A(f_1)(P)^2 \le
 A^{(\alpha)}(f)(Y_k)^2\le \lambda^2.$$
Thus for $P \in Q_k$  if  $A(f)(P) \ge 2 \lambda$ 
 $A(f_1)(P) \le
 \lambda, $  and $ 2 \lambda \le
 A(f)(P) \le
 A(f_1)(P) + A(f_2)(P) $ which ensures that $A(f_2)(P) \ge
 \lambda$, i.e.  $\{P \in Q_k: A(f)(P)> 2\lambda \} \subset \{P \in Q_k: A(f_2)(P)\ge
 \lambda \}.$
By the definition 
$$ A(f_2)(P)^2= \int_{\Gamma(P)\cap\{\delta(X) < r_k\}} \frac{f^2(X)}{\delta(X)^n}dX.$$
Lemma \ref{four} combined the the Ahlfors regularity of $\sigma$ yields 
\begin{eqnarray}\label{a2-c}
\frac{1}{\sigma(B^{\ast}_{k})}\int_{B^{\ast}_{k}} (A(f_2)(P))^2 d\sigma(P)&\le&
  \frac{1}{\sigma(B^{\ast}_{k})}\int_{B^{\ast\ast}_{k}} \frac{f^2(X)}{\delta(X)}dX
   \le
   \frac{C}{\sigma(B^{\ast\ast}_{k})}\int_{B^{\ast\ast}_{k}} \frac{f^2(X)}{\delta(X)}dX\\
   &\le&
 C \inf_{P \in B^{\ast}_{k}}(C(f)(P))^2.\nonumber
 \end{eqnarray}
On the other hand if the set
$ \{X \in Q_k: A(f)(X)>2\lambda ;\ C(f)(X) \le
 \gamma\lambda\}$ is nonempty, there exists $P_0 \in Q_k \subset B^{\ast}_{k}$ such that
$A(f)(P_0)>2\lambda$ and $ C(f)(P_0) \le
 \gamma\lambda,$
thus (\ref{a2-c}) yields
$$\frac{1}{\sigma(B^{\ast}_{k})}\int_{B^{\ast}_{k}} (A(f_2)(P))^2 d\sigma(P)\le
 C \gamma^2\lambda^2.$$
In this case, using the Ahlfors regularity of $\sigma$
\begin{equation*}
  \sigma(\{P \in Q_k: A(f_2)(P)> \lambda \}) \le
  \sigma(\{P \in B^{\ast}_k: A(f_2)(P)> \lambda \})
  C \gamma^2\sigma(B_k) \le
  C \gamma^2\sigma(Q_k).
\end{equation*}
Hence 
$$ \sigma(\{P \in Q_k: Af(P)> 2 \lambda ;\ C(f)(P) \le
 \gamma\lambda\}) \le
 C \gamma^2\sigma(Q_k),$$
and since $\{Q_k\}$ is a disjoint cover of $\{A^{(\alpha)}(f)> \lambda\}$
\begin{eqnarray*}
  \sigma(\{Q \in \partial \Omega: A(f)>2\lambda ;\ C(f) \le
 \gamma\lambda\}) &\le
& \sum_k \sigma(\{X \in Q_k: A(f)>2\lambda ;\ C(f) \le
 \gamma\lambda\})\\
   &\le
&  \sum_k  C \gamma^2\sigma(Q_k)\le
  C \gamma^2 \sigma(\{Q \in \partial \Omega:A^{(\alpha)}(f)> \lambda\}).
\end{eqnarray*}
\end{proof}

\emph{Proof of Theorem \ref{relAC}.}  Note that Theorem \ref{thm3} combined with (\ref{ac+1}) yield part $a)$ for $1<p<\infty$ 
Note that Lemma \ref{fixaper}  ensures that for $\alpha$ big enough
\begin{eqnarray*}
  \sigma(\{A(f)> 2 \lambda \})&\le
 & \sigma(\{A(f)> 2 \lambda ;\ C(f)\le
 \gamma\lambda\}) + \sigma(\{C(f)> \gamma\lambda\}) \\
                                         &\le
 & C \gamma^2 \sigma(\{A^{(\alpha)}(f)> \lambda\})+\sigma(\{C(f)> \gamma\lambda\}).
\end{eqnarray*}
Multiplying both sides by $p \lambda^{p-1}$, integrating with respect to $\lambda$ and using Proposition \ref{prop31} we obtain
$$2^{-p}||A(f)||^{p}_{L^p} \le
 C \gamma^2 ||A^{(\alpha)}(f)||^{p}_{L^p}+C \gamma^{-p} ||C(f)||^{p}_{L^p} 
 C(\alpha, p) \gamma^2 ||A(f)||^{p}_{L^p} + C \gamma^{-p} ||C(f)||^{p}_{L^p} 
 .$$
 Choosing $\gamma>0$ small enough so that
$C \gamma^2 C(\alpha, p) 2^p <\frac{1}{2}$ we obtain
$$ ||A(f)||_{L^p} \le
C||C(f)||_{L^p} $$ provided that $||A(f)||_{L^p} < \infty.$ If $||A(f)||_{L^p} = \infty$
the result is obtained by applying the previous argument
to $f\chi_K$ where $K$ is selected from an increasing family of compact subsets which exhausts $\Omega.$

To prove part $b)$ of the Theorem, let $\Delta=\Delta(Q_0,r)$ and $t\Delta =\Delta(Q_0, tr)$ for $t>3$. Note that
$X\in T(\Delta)$ then $\Delta(Q_X,\delta(X))\subset t\Delta$, as in Lemma \ref{six} $\chi_{\Gamma(Q)}(X)\ge
 \chi_{\Delta(Q_X,\delta(Q))}(Q)$ thus
 \begin{eqnarray}\label{mm38}
  \int_{t\Delta}\left(\int_{\Gamma(Q)} \frac{f^2(X)}{\delta(X)^n}dX\right) d\sigma(Q) &=& \int_{\partial\Omega}\int_{\Omega} \frac{f^2(X)}{\delta(X)^n}\chi_{\Gamma(Q)}(X)\chi_{t\Delta}(Q)dX d\sigma(Q) \\
   &\ge
& \int_{\partial\Omega}\int_{\Omega} \frac{f^2(X)}{\delta(X)^n}\chi_{\Delta(Q_X,\delta(X))}(Q)\chi_{T(\Delta)}(X)dX d\sigma(Q)
\nonumber \\
   &\ge
& \int_{\Omega} \frac{f^2(X)}{\delta(X)^n}\sigma({\Delta(Q_X,\delta(X))})\chi_{T(\Delta)}(X)dX d\sigma(Q)\nonumber \\
   &\ge
& \int_{T(\Delta)} \frac{f^2(X)}{\delta(X)}dX.\nonumber
\end{eqnarray}
(\ref{mm38}) and the Ahlfors regularity of $\sigma$ ensure that
 $$ \frac{1}{\sigma(\Delta)}\int_{T(\Delta)} \frac{f^2(X)}{\delta(X)}dX \le
 \frac{C}{\sigma(\Delta)}\int_{t\Delta} (A(f)(Q))^2d\sigma(Q)  \le
 \frac{C}{\sigma(t\Delta)}\int_{t\Delta} (A(f)(Q))^2d\sigma(Q). $$
Therefore $(C(f)(Q))^2 \le
 C M(A(f)(Q))^2 $ which for $p>1$ ensures that
\begin{eqnarray*}
   \left(\int_{\partial \Omega} (C(f)(Q))^{2p}d\sigma(Q) \right)^{1/p} &\le
& C \left(\int_{\partial \Omega} (M(A(f)^2(Q)))^{p}(Q)d\sigma(Q) \right)^{1/p} \\
   &\le
& C \left(\int_{\partial \Omega} (A(f)(Q))^{2p}d\sigma(Q) \right)^{1/2p}.
\end{eqnarray*}
\hfill $\Box$

\section{Solvability of the Dirichlet problem in $L^p$ for perturbation operators on CADs}\label{theory_of_weights}
\renewcommand{\thesection}{\arabic{section}}

In this section we study the following question: given a second order divergence form elliptic symmetric operator $L_1$ which is a perturbation of an operator $L_0$ for which the Dirichlet problem can be solved in $L^p$ what can be said about the solvability of 
the Dirichlet problem in $L^q$ for $L_1$? As it was pointed out in the introduction this problem is well understood on Lipschitz domains. The goal of this section is to develop a similar theory for CADs. Given that we lack some of the tools available in the Lipschitz
case rather than following Dahlberg's steps we turn our attention to \cite{fkp}. Proposition \ref{p-fkp} below justifies this approach.

Assume that $L_0$ and $L_1$ are  second order divergence form elliptic symmetric operators as in Section 2. Assume also that $0\in\Omega$ and denote by $G_0(Y)$ the Green's function of $L_0$ in $\Omega$ with pole $0$, and by $\omega_0$ the corresponding elliptic measure. Let $a$ be the deviation function defined in (\ref{eqn:tt-a}).

\begin{prop}\label{p-fkp}
Let $\Omega$ be a CAD and that assume $\omega_0\in B_p(\sigma)$ for some $p>1$. 
Given $\epsilon>0$ there exists $\delta>0$ such that if
\begin{equation}\label{small-carl}
\sup_{\Delta \subseteq \partial \Omega}
\bigg\{\frac{1}{\sigma(\Delta)}\int_{T(\Delta)}\frac{a^2(X)}{\delta(X)}dX\bigg\}^{1/2}\le
 \delta,
\end{equation}
then
\begin{equation}\label{small-fkp}
\sup_{\Delta \subseteq \partial \Omega}\bigg\{\frac{1}{\omega_0(\Delta)}\int_{T(\Delta)}a^2(X)\frac{G_0(X)}{\delta^2(X)}dX\bigg\}^{1/2}\le
\epsilon.
 \end{equation}
\end{prop}

\begin{proof}
Let $\Delta_0=\Delta(Q_0,r_0)$ and $t\Delta_0=\Delta(Q_0,tr_0)$. Using Lemma \ref{six}, Lemma \ref{lem2.3}, Fubini and the notation for truncated cones introduced in (\ref{truncated-cone}) we have

\begin{eqnarray}\label{tt-7-1}
  \int_{T(\Delta_0)}a^2(X)\frac{G_0(X)}{\delta(X)^2}dX &\le
  & \int_{\partial\Omega}\int_{\Gamma(Q)}\frac{a^2(X)G_0(X)}{\delta(X)^{n+1}}\chi_{T(\Delta_0)}(X)dXd\sigma(Q) \\
   &\lesssim& \int_{3\Delta_0}\int_{\Gamma^{r_0}(Q)}\frac{a^2(X)}{\delta(X)^{n}}\frac{\omega_0(\Delta(Q_X,\delta(X)))}{\delta(X)^{n-1}}dXd\sigma(Q)\nonumber\\
   &\lesssim&\int_{7\Delta_0}\int_{\Gamma_5^{r_0}(P)}\frac{a^2(X)}{\delta(X)^n}dXd\omega_0(P)\nonumber\\
   &\lesssim& \int_{7\Delta_0} (A^{(5)}_{r_0}(P))^2 d\omega_0(P).\nonumber
   \end{eqnarray}
   
Since $\omega_0\in B_p(d\sigma)$ for some $p>1$, if $\frac{1}{p}+\frac{1}{q}=1$ and $k=\frac{d\omega_0}{d\sigma}$ then
\begin{eqnarray}\label{**G}
  \int_{t\Delta_0}(A_{r_0}^{(\alpha)}(a)(P))^2k(P)d\sigma(P) &\le
& \bigg(\int_{t\Delta_0}(A_{r_0}^{(\alpha)}(a)(P))^{2q}d\sigma(P)\bigg)^{1/q}\bigg(\int_{t\Delta_0}k^pd\sigma\bigg)^{1/p} \\
   &\le
& \bigg({\fint}_{t\Delta_0}(A_{r_0}^{(\alpha)}(a)(P))^{2q}d\sigma(P)\bigg)^{1/q}\omega_0(t\Delta_0)\nonumber
\end{eqnarray}
because
$$\bigg(\int_{t\Delta_0}k^pd\sigma\bigg)^{1/p}\le
 C \sigma(t\Delta_0)^{1/p}{\fint}_{t\Delta}kd\sigma \le C \omega_0(t\Delta)\sigma(t\Delta_0)^{-1/q}.$$

Combining (\ref{tt-7-1}), (\ref{**G}) and Lemma \ref{lem2.4} we obtain
\begin{equation}\label{tt-7-2}
\frac{1}{\omega_0(\Delta_0)} \int_{T(\Delta_0)}a^2(X)\frac{G_0(X)}{\delta(X)^2}dX \le C
\bigg({\fint}_{7\Delta_0}(A_{r_0}^{(5)}(a)(P))^{2q}d\sigma(P)\bigg)^{1/q}.
\end{equation}
We estimate $\bigg({\fint}_{7\Delta_0}(A_{r_0}^{(5)}(a)(P))^{2q}d\sigma(P)\bigg)^{1/q}$ by duality. Let $g\in L^p(\sigma)$ with $\frac{1}{p}+\frac{1}{q}=1$ and $g$ supported on $7\Delta_0$. Without lot of generality we may assume that $g\ge 0$.
\begin{eqnarray}\label{tt.7.200}
\int_{\partial\Omega}(A_{r_0}^{(5)}(a)(P))^2 g(P)\, d\sigma(P) &=& \int_{\po}\int_{\G_5(Q)}\frac{a^2(X)}{\delta(X)^n}\chi_{\{\delta(X)<r_0\}}
(X)
\chi_{7\Delta_0}(Q)g(Q)\, dX\, d\s(Q)\\
&\le&
 \int_{\po}\int_{\O}\frac{a^2(X)}{\delta(X)^n}\chi_{\{\delta(X)<r_0\}}
(X)\chi_{T(13\D_0)}(X)
\chi_{\D(Q_X,7\delta(X))}(Q)g(Q)\, dX\, d\s(Q)\nonumber\\
&\le &\int_{T(13\D_0)}\chi_{\{\delta(X)<r_0\}}
(X)\frac{a^2(X)}{\delta(X)^n}\left(\int_{\D(Q_X,7\d(X))}g(Q)\, d\s(Q)\right)\, dX\nonumber\\
&\le &C \int_{T(13\D_0)}\chi_{\{\delta(X)<r_0\}}
(X)\frac{a^2(X)}{\delta(X)}\left(\fint_{\D(Q_X,7\d(X))}g(Q)\, d\s(Q)\right)\, dX.\nonumber
\end{eqnarray}
Letting $F(X)=\chi_{\{\delta(X)<r_0\}}(X)\fint_{\D(Q_X,7\d(X))}g(Q)\, d\s(Q)$ and applying Proposition \ref{thm1} to the last term in 
(\ref{tt.7.200}) we obtain
\begin{equation}\label{tt.7.201}
\int_{\partial\Omega}(A_{r_0}^{(5)}(a)(P))^2 g(P)\, d\sigma(P)\le C\int_{\po}NF(Q)C(\frac{a^2}{\delta}\chi_{T(13\D_0)})(Q)\, d\s(Q),
\end{equation}
where
\begin{equation}\label{tt.7.202}
NF(Q)=\sup_{X\in\G(Q)}|F(X)|=\sup_{X\in\G^{r_0}(Q)}|F(X)|\le C M_{9r_0}g(Q).
\end{equation}
Here $ M_{9r_0}g$ denotes the truncated maximal function of $g$, i.e. $M_{9r_0}g(Q)=\sup_{0<r\le 9r_0}\fint_{\D(Q,r)}|g|\, d\s$. Note that if $|Q-Q_0|\ge 10r_0$ and $X\in \G^{r_0}(Q)$ then $|Q-Q_X|> 7r_0$ and $NF(Q)=0$. Moreover (\ref{small-carl}) yields 
$C(\frac{a^2}{\delta}\chi_{T(13\D_0)})(Q)\le C(\frac{a^2}{\delta})(Q)\le \d$. This combined with (\ref{tt.7.201}), (\ref{tt.7.202}), H\"older's inequality, the fact that $\sigma$ is Ahlfors regular and the maximal function theorem ensures that
\begin{eqnarray}\label{tt.7.203}
\int_{\partial\Omega}(A_{r_0}^{(5)}(a)(P))^2 g(P)\, d\sigma(P)&\le& C\d\int_{10\D_0}NF(Q)\, d\s(Q)\\
&\le& C\d\int_{10\D_0}M_{9r_0}g(Q)\, d\s(Q)\nonumber\\
&\le&  C\d\left(\int_{10\D_0}(M_{9r_0}g(Q))^p\, d\s(Q)\right)^{1/p}\s(10\D_0)^{1/q}\nonumber\\
&\le & C\delta \left(\int_{\po}g(Q)^p\, d\s(Q)\right)^{1/p}\s(7\D_0)^{1/q},\nonumber
\end{eqnarray}
which implies
\begin{equation}\label{tt.7.204}
\bigg({\fint}_{7\Delta_0}(A_{r_0}^{(5)}(a)(P))^{2q}d\sigma(P)\bigg)^{1/q}\le C\d.
\end{equation}
Note that (\ref{tt.7.204}) combined with (\ref{tt-7-2}) yields (\ref{small-fkp}) provided $C\d<\epsilon$.
\end{proof}

In this section we need to consider variants of the non-tangential maximal function of $u$. Define for $\alpha\in(0,1)$ and $\eta>0$
\begin{equation}\label{N-tilde}
\widetilde{N}^\eta_\alpha F(Q)=\sup_{X \in \Gamma_\eta(Q)} \left( \fint_{B(X,\alpha\delta(X)/8)}F^2(Z)dZ\right)^{1/2}.
\end{equation}
For simplicity $\widetilde{N}^1_\alpha F =\widetilde{N}_\alpha F$, $\widetilde{N}^\eta_1F=\widetilde{N}^\eta F$ and  $\widetilde{N}^1_1F=\widetilde{N} F$. Recall that $N_\eta F(Q)=\sup_{X\in\G_\eta(Q)}|F(X)|$.

\begin{preremark}\label{N-hat-tilde}
Let $\mu$ be a doubling measure on $\partial\Omega$ then
for $p\ge  1$, $\alpha,\beta\in(0,1)$ and $\eta>0$
$$||\widetilde{N}_\alpha F||_{L^p(\mu)}\sim ||\widetilde{N}_\beta F||_{L^p(\mu)}\sim||\widetilde{N}
^\eta_\alpha F||_{L^p(\mu)}.$$
\end{preremark}

\begin{proof} 
Note that Proposition \ref{Ndep} ensures that for $1\le p<\infty$, $\|\widetilde{N}_\alpha^\eta F\|_{L^p(\mu)}\sim \|\widetilde{N}_\alpha F\|_{L^p(\mu)}$. Moreover for  $\alpha>\beta$, $\widetilde{N}_{\beta}F(Q)\le (\alpha/\beta)^n\widetilde{N}_{\alpha}F(Q)$ thus it is enough to show $\|\widetilde{N}_\alpha F\|_{L^p(\mu)} \le C \|\widetilde{N}_\beta F\|_{L^p(\mu)}$.
We claim that for $\gamma=(2+\frac{\alpha}{8})(1-\frac{\alpha}{8})^{-1}-1$
\begin{equation}\label{mm-eqv0}
\widetilde{N}_\alpha F(Q)\le C_{n,\alpha,\beta} \widetilde{N}^{\gamma}_{\beta}F(Q),
\end{equation}
which yields the desired inequality.
Note that
\begin{eqnarray*}
  \fint_{B(X,\frac{\alpha}{8}\delta(X))} F^2(Z)dZ &=&\frac{C_{\alpha}}{\delta(X)^n}\int_{B(X,\frac{\alpha}{8}\delta(X))} F^2(Z)dZ \\
   &=& \frac{C_{\alpha}}{\delta(X)^n}\int_{B(X,\frac{\alpha}{8}\delta(X))\setminus B(X,\frac{\beta}{8}\delta(X))} F^2(Z)dZ+\frac{C_{\alpha}}{\delta(X)^n}\int_{B(X,\frac{\beta}{8}\delta(X))} F^2(Z)dZ.
\end{eqnarray*}
Covering the region $B(X,\frac{\alpha}{8}\delta(X))\backslash B(X,\frac{\beta}{8}\delta(X))$ by balls $B_i=B(Y_i,r)$ with radius $r=(1-\frac{\alpha}{8})\delta(x)\frac{\beta}{8}$ and $Y_i\in B(X,\frac{\alpha}{8}\delta(X))\backslash B(X,\frac{\beta}{8}\delta(X))$,
and noting that  the number of such balls only depends on $\alpha,\beta, n$ we have
\begin{equation}\label{mm-eqv1}
\fint_{B(X,\frac{\alpha}{8}\delta(X))} F^2(Z)dZ\le
 C_{\alpha,\beta,n}' \bigg(\frac{1}{\delta(X)^n}\sum_i \int_{B_i} F^2(Z)dZ+\fint_{B(X,\frac{\beta}{8}\delta(X))} F^2(Z)dZ\bigg).
\end{equation}
If $X\in \G(Q)$ and $Y\in B(X,\frac{\alpha}{8}\d(X))$ then $ (1-\frac{\alpha}{8})\d(X)\le \d(Y)\le  (1+\frac{\alpha}{8})\d(X)$ and $Y\in \G_\g(Q)$.
Hence
\begin{equation}\label{tt.7.17}
\fint_{B_i}F^2(Z)\, dZ\le C\sup_{Y\in\G_\g(Q)}\fint_{B(Y,r)}F^2(Z)\, dZ\le C\sup_{Y\in\G_\g(Q)}\fint_{B(Y,\frac{\beta}{8}\d(Y))}F^2(Z)\, dZ=\widetilde{N}_\beta^\gamma F(Q),
\end{equation}
which combined with (\ref{mm-eqv1}) yields (\ref{mm-eqv0}).
\end{proof}

\begin{preremark}\label{Nequiv}
Assume that $L_iu=0$ for $i=0$ or $i=1$. Then $||\widetilde{N}u||_{L^p(\s)}\sim ||Nu||_{L^p(\s)}$ for $1\le p<\infty$.
\end{preremark}
\begin{proof}
Since $\widetilde{N}u(Q)\le C N_{10/7}u(Q)$, Proposition \ref{Ndep} ensures that $\|\widetilde{N}u\|_{L^p(\s)}\le C\|Nu\|_{L^p(\s)}$.
Since $u$ is a solution for $L_i$ then $u^2$ is a subsolution for $L_i$ and Lemma 1.1.8 of \cite{k1} guarantees that
$$u^2(X)\le
 \sup_{B(X,\frac{\delta(X)}{16})}u^2(Y)\le
 C {\fint}_{B(X,\frac{\delta(X)}{8})}u^2(Z)dZ.$$
 Hence $Nu(Q)\le C\widetilde{N}u(Q)$ and $\|Nu\|_{L^p(\s)}\le C \|\widetilde{N}u\|_{L^p(\s)}$ follows.
 \end{proof}


We still need a few preliminaries before we can get to the proof of Theorem \ref{mainthm1}.
Recall that by assumption $0\in \Omega$. Let $R_0=\frac{1}{2^{30}}\min\{\delta(0),1\}$. The following calculation shows that we may assume that $a(X)=0$ for all $X\in \Omega$ such that $\delta(X)>4R_0$. Cover the boundary $\partial\Omega$ by balls $\{B(Q_i,R_0/2)\}_{i=1}^M$ such that $Q_i\in\partial\Omega$ and $|Q_i-Q_j|\ge
 \frac{R_0}{2}$ for $i\neq j$. Note that $M$ depends only on  $n$, $R_0$ and ${\rm{diam}}\, \Omega$. Let $\{\varphi_i\}_{i=1}^M$ be a partition of unity associated with this covering satisfying $0\le
\varphi_i\le
 1$, ${\rm{spt}}\varphi_i\subset B(Q_i,2R_0)$, $\varphi_i\equiv 1$ on $B(Q_i,R_0)$ and $|\nabla\varphi_i|\le
 4/R_0$. Define
$$\psi_i(X)=\left\{
  \begin{array}{ll}
    \bigg(\sum_{j=1}^M\varphi_j(X)\bigg)^{-1}\varphi_i(X) & {\rm{if}} \  \sum_{j=1}^M\varphi_j(X)\neq 0\\
    0 & {\rm{otherwise}}.
  \end{array}
\right.$$
Note that for $X\in (\partial \Omega, \frac{1}{2}R_0):=\{Y\in \R^n :\exists Q_Y\in\partial\Omega\ {\rm{with}}\ |Q_Y-Y|=\delta(Y)\le
 R_0/2\}$ there exists $Q_X\in\partial\Omega$ with $|Q_X-X|\le
 R_0/2$ and $i\in\{1,...,M\}$ such that $|Q_X-Q_i|<R_0/2$, thus $X\in B(Q_i,R_0)$ and $\varphi_i(X)=1$ therefore $\sum_{j=1}^{M}\psi_j(X)=1$. If $X\in \R^n \setminus (\partial \Omega, {2}R_0)$ then $\varphi_i(X)=0$ and $\sum_{j=1}^{M}\psi_j(X)=0$. Consider the matrix
 \begin{equation}\label{tt-a'}
 A'(X)=\bigg(\sum_{j=1}^{M}\psi_j(X)\bigg)A_1(X)+\bigg(1-\sum_{j=1}^{M}\psi_j(X)\bigg)A_0(X)
 \end{equation}
 and the corresponding operator $L'={\rm{div}}A'\nabla$. Note that $A'$ is symmetric and $L'$ is an elliptic second order divergence form operator with bounded coefficients in $\Omega$. 
Denote by $a'$ the deviation function 
$$a'(X)=\sup_{B(X,\delta(X)/2)}|A'(Y)-A_0(Y)|$$
\begin{lemma}\label{claim1-p4}
Let $A'$ be as in (\ref{tt-a'}) then
$a'(X)=0$
for $X\in \Omega$, with $\delta(X)>4R_0$.
\end{lemma}
\begin{proof}
For $X\in \Omega$ with $\delta(X)>4R_0$, if $Y\in B(X,\delta(X)/2)$ then $\d(Y)\ge \frac{\delta(X)}{2}>2R_0$,
$A'(Y)=A_0$ and $a'(X)=0$.
\end{proof}

\begin{lemma}\label{claim2-p4}
If $\omega'$ denotes the elliptic measure associated to $L'$ with pole at $0$, then $\omega_1\in B_p(\omega_0)$ if and only if $\omega'\in B_p(\omega_0)$.
\end{lemma}
\begin{proof}
Let $G'$ be the Green's function for $L'$ in $\Omega$. Note that for $X\in(\partial\Omega,\frac{R_0}{2})$, $A'(X)=A_1(X)$. For $r<R_0/4$ and $Q\in\partial\Omega$ the comparison principle for NTA domains yields that for $i=0,1$
\begin{equation}\label{tt7.B}
\frac{G_i(0,A(Q,r))}{r}\sim\frac{\omega_i(\Delta(Q,r))}{r^{n-1}}, \hbox{  and  }
\frac{G'(0,A(Q,r))}{r}\sim\frac{\omega'(\Delta(Q,r))}{r^{n-1}}.
\end{equation}
Moreover
\begin{equation}\label{tt7.A}
\frac{G_1(0,A(Q,r))}{G'(0,A(Q,r))}\sim 1.
\end{equation}
Combining (\ref{tt7.B}) and (\ref{tt7.A}) we have
\begin{equation}\label{tt7.C}
\frac{G_1(0,A(Q,r))}{G_0(0,A(Q,r))}\sim\frac{\omega_1(\Delta(Q,r))}{\omega_0(\Delta(Q,r))}\ {\hbox{ and }}\ \frac{G_1(0,A(Q,r))}{G'(0,A(Q,r))}\sim\frac{\omega_1(\Delta(Q,r))}{\omega'(\Delta(Q,r))}\sim 1
\end{equation}
which yields for every $Q\in\partial\Omega$ and for every $r<R_0/2$
\begin{equation}\label{tt7.D}
\frac{\omega'(\Delta(Q,r))}{\omega_0(\Delta(Q,r))}\sim \frac{\omega_1(\Delta(Q,r))}{\omega_0(\Delta(Q,r))}
\end{equation}
with constants that only depend on the NTA constants of $\Omega$. Letting $r$ tend to $0$ we obtain that for every $Q\in\partial\Omega$
\begin{equation}\label{tt7.E}
\frac{d\omega'}{d\omega_0}(Q)\sim \frac{d\omega_1}{d\omega_0}(Q).
\end{equation}
\end{proof}

\begin{lemma}\label{claim3-p7}
Assume that
\begin{equation}\label{tt7.101}
\sup_{\Delta\subset\partial\Omega}\bigg\{\frac{1}{\omega_0(\Delta)}\int_{T(\Delta)}a^2(Y)\frac{G_0(Y)}{\delta(Y)^2}dY\bigg\}^{1/2}<\varepsilon_0
\end{equation}
with $a(Y)=0$ for $Y\in\Omega$ and $\delta(Y)>4R_0$ where $R_0=\frac{1}{2^{30}}\min\{\delta(0),1\}$. Then there exists $C>0$ such that for $X\in\Omega$ with $\delta(X)>5R_0$
\begin{equation}\label{tt7.102}
\sup_{\Delta\subset\partial\Omega}\bigg\{\frac{1}{\omega_0^X(\Delta)}\int_{T(\Delta)}a^2(Y)\frac{G_0(X,Y)}{\delta(Y)^2}dY\bigg\}^{1/2}\le C\varepsilon_0
.
\end{equation}
Here $C$ depends on NTA constants of $\Omega$, the NTA character of $\Omega$ its diameter and $R_0$.
\end{lemma}

\begin{proof}
If $\Delta=\Delta(Q,r)$ with $r\le
9/2R_0$ and $\delta(X)>5R_0$ then for $Y\in T(\Delta)$ by the comparison principle and (\ref{tt7.B}) we have
\begin{equation}\label{tt7.103}
\frac{G_0(X,Y)}{G_0(Y)}\sim\frac{G_0(X,A(Q,r))}{G_0(A(Q,r))}\sim\frac{\omega^X_0(\Delta(Q,r))}{\omega_0(\Delta(Q,r))}
\end{equation}
hence
\begin{equation}\label{tt7.104}
\frac{1}{\omega_0^X(\Delta)}\int_{T(\Delta)}a^2(Y)\frac{G_0(X,Y)}{\delta(Y)^2}dY\sim\frac{1}{\omega_0(\Delta)}\int_{T(\Delta)}a^2(Y)\frac{G_0(Y)}{\delta(Y)^2}dY.
\end{equation}
If $r\ge 9/2R_0$ then
$$\frac{1}{\omega_0^X(\Delta)}\int_{T(\Delta)}a^2(Y)\frac{G_0(X,Y)}{\delta(Y)^2}dY=\frac{1}{\omega_0^X(\Delta)}\int_{T(\Delta)\cap(\partial\Omega,4R_0)}a^2(Y)\frac{G_0(X,Y)}{\delta(Y)^2}dY.$$
Covering $\partial\Omega$ by balls $\{B(Q,R_0/2)\}^{M}_{i=1}$, if $\Delta_i=B(Q_i,9/2R_0)\cap\partial\Omega$ we have, using (\ref{tt7.104}), that
\begin{equation}\label{tt7105}
\frac{1}{\omega_0^X(\Delta)}\int_{T(\Delta)\cap(\partial\Omega,4R_0)}a^2(Y)\frac{G_0(X,Y)}{\delta(Y)^2}dY\le
 \frac{1}{\omega_0^X(\Delta)}\sum_{i=1}^M\int_{T(\Delta_i)}a^2(Y)\frac{G_0(X,Y)}{\delta(Y)^2}dY
\end{equation}
$$\lesssim \sum_{i=1}^M\bigg(\int_{T(\Delta_i)}a^2(Y)\frac{G_0(Y)}{\delta(Y)^2}dY\bigg)\frac{\omega_0^X(\Delta_i)}{\omega_0(\Delta_i)}\frac{1}{\omega_0^X(\Delta)}\lesssim \varepsilon_0
$$
because $\omega_0$, $\omega_0^X$ are doubling, $\o^X(\D)\sim C\o^X(\D_i)$ and by (\ref{tt7.103}).
\end{proof}

The last preliminary concerns the existence of a family of dyadic cubes in $\po$ whose ``projections'" in $\O$ provide a good covering 
of $\O\cap (\po, 4R_0)$, with $R_0$ as above. Since $\Omega$ is a CAD in $\R^n$, both $\sigma=\mathcal{H}^{n-1}\res \partial\Omega$ and $\omega_0$ are doubling measures and therefore $(\partial\Omega, |\ |, \sigma)$ and $(\partial\Omega, |\ |, \omega_0)$ are spaces of homogeneous type. Here $|\ |$ denotes the Euclidean distance in $\R^n$. M. Christ's construction (see \cite{c}) ensures that there exists a family of dyadic  cubes $\{Q_\alpha^k\subset\partial\Omega:k\in\mathbb{Z}, \alpha\in I_k\}$, $I_k\subset \mathbb{N}$ such that for every $k\in\mathbb{Z}$
\begin{equation}\label{tt7.19}
    \sigma(\partial\Omega\setminus\bigcup_{\alpha}Q_\alpha^k)=0, \ \ \ \ \omega_0(\partial\Omega\setminus\bigcup_{\alpha}Q_\alpha^k)=0.
\end{equation}
Furthermore the following properties are satisfied:
\begin{enumerate}
\item{} If $l\ge
 k$ then either $Q_\beta^l\subset Q_\alpha^k$ or $Q_\beta^l\cap Q_\alpha^k=\emptyset.$

\item{} For each $(k,\alpha)$ and each $l<k$ there is a unique $\beta$ so that $Q_\alpha^k\subset Q_\beta^l$.

\item{} There exists a constant $C_0>0$ such that ${\rm{diam}}\,Q_\alpha^k\le
 C_08^{-k}$.

\item{} Each $Q_\alpha^k$ contains a ball $B(Z_\alpha^k,8^{-k-1})$.
\end{enumerate}

The fact that $B(Z_\alpha^k,8^{-k-1})\subset Q_\alpha^k$ implies that ${\rm{diam}}\, Q_\alpha^k\ge
 8^{-k-1}$. The Ahlfors regularity property of $\sigma$ combined with properties 3 and 4 ensure that there exists $C_1>1$ such that
\begin{equation}\label{tt7.24}
    C_1^{-1}8^{-k(n-1)}\le
 \sigma(Q_\alpha^k)\le
 C_18^{-k(n-1)}.
\end{equation}
In addition the doubling property of $\omega_0$ yields
\begin{equation}\label{tt7.25}
    \omega_0(B(Z_\alpha^k,8^{-k-1}))\sim\omega_0(Q_\alpha^k).
\end{equation}
For $k\in\mathbb{Z}$ and $\alpha\in I_k$ we define
\begin{equation}\label{tt7.26}
    I_\alpha^k=\{Y\in\Omega: \lambda8^{-k-1}<\delta(Y)<\lambda8^{-k+1}, \ \exists P\in Q_\alpha^k\ \ {\rm{so \ that}}\ \ \lambda8^{-k-1}<|P-Y|<\lambda8^{-k+1}\},
\end{equation}
where $\lambda>0$ is chosen so that for each $k$, the $\{I_\alpha^k\}_{\alpha\in I_k}$'s have finite overlaps and
\begin{equation}\label{tt7.26A}
    \Omega\cap(\partial\Omega,4R_0)\subset \bigcup_{\alpha,k\le
 k_0}I_\alpha^k.
\end{equation}
Here $k_0$ is chosen so that $4R_0<\lambda 8^{-k-1}$; i.e $k_0=[\frac{{\rm{log}}\lambda-{\rm{log}}32R_0}{{\rm{log}}8}]+1$. To see that such a $\lambda>0$ can be found, note that if $I_\alpha^k\cap I_\beta^k\neq\emptyset$ there exist $Y\in I_\alpha^k\cap I_\beta^k$, $P_\alpha\in Q_\alpha^k$ and $P_\beta\in Q_\alpha^k$ so that $$\lambda8^{-k-1}<\delta(Y), \ |P_\alpha-Y|, \ |P_\beta-Y|<\lambda 8^{-k+1}.$$
Thus $|P_\alpha-P_\beta|\le
 2\lambda 8^{-k+1}$ and for $P\in Q_\beta^k$,
\begin{eqnarray}\label{tt7.27}
  |P_\alpha-P| &\le
& |P_\alpha-P_\beta|+|P_\beta-P| 
   \le
 2\delta(Y)+{\rm{diam}}\,
Q_\beta^k \nonumber\\
   &\le
& 2\lambda 8^{-k+1}+C_08^{-k} 
   \le
 8^{-k}(16\lambda+C_0).
\end{eqnarray}
Thus (\ref{tt7.27}) yields that given $I_\alpha^k$, if $Q^k_\beta$ is such that $I_\alpha^k\cap I_\beta^k\neq\emptyset$ then $Q_\beta^k\subset B(P_\alpha, 8^{-k}(16\lambda+C_0))$ for some $P_\alpha\in Q_\alpha^k$. Since $\{Q^k_\beta\}_{\beta\in I_k}$ is a disjoint collection,(\ref{tt7.24}) yields that the number $N$ of cubes $Q_\beta^k$ so that $I_\alpha^k\cap I_\beta^k\neq\emptyset$ satisfies $NC^{-1}8^{-k(n-1)}\le  C8^{-k(n-1)}(16\lambda+C_0)^{n-1}$, i.e.
$N\le  C^2(16\lambda+C_0)^{n-1}$. To show that the $I_\alpha^k$'s cover $(\partial\Omega,4R_0)$ let $Y\in(\partial\Omega,4R_0)$, $\delta(Y)\le
 4R_0<\frac{1}{2^{28}}\min\{\delta(0),1\}$ by choosing $\lambda \ge
 \frac{1}{8}\max\{\delta(0),1\}+1+64C_0$ we have that $\delta(Y)<\frac{\lambda}{8}$. Thus there exists $k\ge
 2$ so that $\lambda8^{-k-1}<\delta(Y)<\lambda8^{-k+1}$ and $Q_Y\in\partial\Omega$ so that $|Q_Y-Y|=\delta(Y)$. Let $\rho_0=\frac{1}{2}\min\{\delta(Y)-\lambda8^{-k-1}, \lambda8^{-k+1}-\delta(Y)\}>0$. Since $\sigma(\partial\Omega\setminus\bigcup_{\alpha\in I_k}Q^k_\alpha)=0$ and $\sigma(\Delta(Q_Y,\rho_0))\ge
 C^{-1}\rho_0^{n-1}>0$ there exists $\alpha\in I_k$ so that $\Delta(Q_Y,\rho_0)\cap Q_\alpha^k\neq\emptyset$. Let $P_\alpha\in\Delta(Q_Y,\rho_0)\cap Q^k_\alpha$ then
\begin{equation}\label{tt7.29}
 \delta(Y)-\rho_0\le   |Q_Y-Y|-|P_\alpha-Q_Y|\le
 |P_\alpha-Y|\le
|P_\alpha-Q_Y|+|Q_Y-Y|\le \rho_0+\delta(Y)
\end{equation}
hence by the selection of $\rho_0$,
$$\frac{\delta(Y)+\lambda8^{-k-1}}{2}\le
 |P_\alpha-Y|\le
 \frac{\delta(Y)+\lambda8^{-k+1}}{2}.$$
Thus $Y\in I^k_\alpha$ provided that $\lambda$ is chosen as above.

Next we proceed with the proof of Theorem \ref{mainthm1} following the approach presented in \cite{fkp}. Note that (\ref{condThm2.11})  implies that $\varepsilon(X)=A_1(X)-A_0(X)\equiv 0$ on $\partial\Omega$, i.e. $L_0=L_1$ on $\partial\Omega$. Thus $L_1$ is regarded as a perturbation of $L_0$. Hence as in \cite{fkp} the strategy consists of regarding the solution to $L_1$ with given boundary data as a perturbation of the solution to $L_0$ with the same boundary data. We consider the Dirichlet problem
\begin{equation}\label{ttm7.4}
\left\{
  \begin{array}{ll}
    L_1u_1=0 & {\rm{in}} \ \Omega \\
    u_1 \ \big{.\mid}_{\partial \Omega}=f & \in L^2(\omega_0).
  \end{array}
\right.
\end{equation}
We need to show the following apriori estimate
\begin{equation}\label{ttm7.5}
||N(u_1)||_{L^{2}(\omega_0)}\le
 ||f||_{L^{2}(\omega_0)}
\end{equation}
which is equivalent to the statement that $\omega_1\in B_2(\omega_0)$. Assume that $f\in C(\partial\Omega)$ and $u_1$ is a solution of (\ref{ttm7.4}). Let $u_0$ satisfy
\begin{equation}\label{ttm7.6}
\left\{
  \begin{array}{ll}
    L_0u_0=0 & {\rm{in}} \ \Omega \\
    u_0 =f & {\rm{on}}\  \partial\Omega.
  \end{array}
\right.
\end{equation}
then
$$||N(u_0)||_{L^{2}(\omega_0)}\le
 ||f||_{L^{2}(\omega_0)}$$
since $Nu_0(Q)\le
 C M_{\omega_0}(f)(Q)$ and $u_1$ is related to $u_0$ by the formula
$$u_1(X)=u_0(X)+\int_{\Omega}G_0(X,Y)L_0 u_1(Y)dY=u_0(X)+F(X).$$
Integration by parts shows that
$$F(X)=\int_{\Omega}G_0(X,Y)(L_0-L_1)u_1(Y)dY=\int_{\Omega} \nabla_YG_0(X,Y)\varepsilon(Y)\nabla u_1(Y) dY$$
where $\varepsilon(Y)=A_1(Y)-A_0(Y)$.

As in \cite{fkp}, the proof of Theorem \ref{mainthm1} follows from the two lemmas below (lemmas \ref{lem2.9} and \ref{lem2.10}). We start with the analogue of Lemma 2.9 of \cite{fkp}.
\begin{lemma}\label{lem2.9}
Let $\Omega$ be a CAD and assume that (\ref{condThm2.11}) holds. Then there exists $C>1$ and $M>1$ such that for $Q_0\in\partial\Omega$
\begin{equation}\label{ttlem7.6A}
\widetilde{N}F(Q_0)\le
 C\varepsilon_0
M_{\omega_0}(S_M(u_1))(Q_0)
\end{equation}
and
\begin{equation}\label{ttlem7.6B}
\widetilde{N}_{1/2}(\delta|\nabla F|)(Q_0)\le
 C\varepsilon_0
\bigg[M_{\omega_0}(S_M(u_1))(Q_0)+\widetilde{N}(\delta|\nabla F|)(Q_0)\bigg].
\end{equation}
Therefore
\begin{equation}\label{ttlem7.6C}
\int_{\partial\Omega}\widetilde{N}F(Q)^2+\widetilde{N}(\delta|\nabla F|)(Q)^2d\omega_0(Q)\le
 C\varepsilon_0
^2\int_{\partial\Omega}S^2u_1(Q)d\omega_0(Q).
\end{equation}
\end{lemma}

Here $M_{\o_0}$ denotes the Hardy-Littlewood maximal function with respect to $\o_0$, and $S_\alpha (u)$ denotes the square function of $u$ given by
\begin{equation}\label{tt.7a}
S_\alpha^2(u)(Q)=\int_{\Gamma_\alpha(Q)}|\nabla u(X)|^2\delta^{2-n}dX.
\end{equation}

\begin{proof}
The proof follows the same guidelines of Lemma 2.9 in \cite{fkp}. We estimate each term separately. First we show that there exists $M>1$ so that for $Q_0\in\partial\Omega$
\begin{equation}\label{tt7.10.6}
\widetilde{N}F(Q_0) \le
 C \varepsilon_0
 M_{\omega_0}(S_M (u_1))(Q_0).
\end{equation}
Let $X \in \Gamma(Q_0)$ and set $B(X)=B(X,\delta(X)/4) $. We split the potential $F$ into two pieces
\begin{equation}\label{tt7.107}
F(Z)=F_1(Z)+F_2(Z)
\end{equation}
where
\begin{equation}\label{tt7.108}
F_1(Z)=\int_{B(X)}\nabla_YG_0(Z,Y)\varepsilon(Y)\nabla u_1(Y) dY
\end{equation}
and
\begin{equation}\label{tt7.109}
F_2(Z)=\int_{\Omega\setminus B(X)}\nabla_YG_0(Z,Y)\varepsilon(Y)\nabla u_1(Y) dY.
\end{equation}
To estimate $\widetilde{N}F(Q_0)$ let $X\in\Gamma(Q_0)$ and note that
$$\int_{B(X,\frac{\delta(X)}{8})}F^2(Z)\frac{dZ}{\delta(X)^n}\lesssim {\fint}_{B(X,\frac{\delta(X)}{8})}F_1^2(Z)dZ+{\fint}_{B(X,\frac{\delta(X)}{8})}F_2^2(Z)dZ.$$
We look at each term on the right hand side separately. For $Y\in B(X)$, $\frac{3\delta(X)}{4}\le
 \delta(Y)\le
 \frac{5\delta(X)}{4}$, and either $\delta(X)<8R_0$ or $\delta(X)\ge
 8R_0$. If $\delta(X)\ge
 8R_0$ then $\delta(Y)\ge
 6R_0$ thus $\varepsilon(Y)=0$. If $\delta(X)<8R_0$ then $\delta(Y)<10R_0$ and $|Y|\ge
 8R_0$. In this case the Harnack principle ensures that $G_0(X)\sim G_0(Y)$. Furthermore, since $\omega_0$ is doubling, for $Y\in \Gamma_{5/4}(Q_0)$ the relationship between the Green's function and the elliptic measure on NTA domains yields
\begin{equation}\label{tt7.110}
\frac{G_0(X)}{\delta(X)}\sim\frac{\omega_0(\Delta(Q_0,\delta(X))}{\delta(X)^{n-1}}\sim \frac{\omega_0(\Delta(Q_0,\delta(Y))}{\delta(Y)^{n-1}}\sim \frac{G_0(Y)}{\delta(Y)}.
\end{equation}
Therefore for $Y_0\in B(X)$ either $\varepsilon(Y_0)=0$ or for $\delta(X)< 8R_0$ in this case (\ref{tt7.110}) and the doubling properties of $\o_0$ imply 
\begin{eqnarray}\label{tt7.9}
  |\varepsilon(Y_0)|  &\lesssim& \left( {\fint}_{B(X,\frac{\delta(X)}{8})} a ^2(Y) dY\right)^{1/2} \nonumber \\
  &\lesssim& \left( {\fint}_{B(X)} a ^2(Y) dY\right)^{1/2} \nonumber \\
   &\lesssim& \left( \int_{B(X)} \frac{a ^2(Y)}{\delta(Y)^2}G_0(Y)\frac{\delta(Y)}{G_0(Y)} \frac{dY}{\delta^{n-1}(X)}\right)^{1/2}\nonumber \\
   &\lesssim& \left( \frac{1}{\delta(X)^{n-1}}\int_{B(X)} \frac{a ^2(Y)}{\delta(Y)^2}G_0(Y)\frac{\delta(X)^{n-1}}{\omega_0(\Delta(Q_0,\delta(X)))} dY\right)^{1/2}\nonumber \\
   &\lesssim& \left( \frac{1}{\omega_0(\Delta(Q_0,\delta(X)))}\int_{B(X,\delta(X)/4)} \frac{a ^2(Y)}{\delta(Y)^2}G_0(Y)dY\right)^{1/2}\nonumber \\
   &\lesssim& \left( \frac{1}{\omega_0(\Delta(Q_0,\delta(X)))}\int_{T(\Delta(Q_0,3\delta(X)))} \frac{a ^2(Y)}{\delta(Y)^2}G_0(Y)dY\right)^{1/2}\nonumber \\
   &\lesssim& \left( \frac{1}{\omega_0(\Delta(Q_0,3\delta(X)))}\int_{T(\Delta(Q_0,3\delta(X)))} \frac{a ^2(Y)}{\delta(Y)^2}G_0(Y)dY\right)^{1/2} \lesssim \varepsilon_0
.
\end{eqnarray}

Let $\widetilde{G}_0(Z,Y)$ be the Green's function for $L_0$ in $2B(X)=B(X,\delta(X)/2)$. Let
\begin{equation}\label{tt7.111}
K(Z,Y)=G_0(Z,Y)-\widetilde{G}_0(Z,Y),\ \ \widetilde{F}_1(Z)=\int_{B(X)}\nabla_Y\widetilde{G}_0(Z,Y)\varepsilon(Y)\nabla u_1(Y) dY
\end{equation}
and
\begin{equation}\label{tt-hat}
\widehat{F}_1(Z)=F_1(Z)-\widetilde{F}_1(Z).
\end{equation}

\begin{equation}\label{tt7.400}
\left\{
  \begin{array}{lll}
    L_0 \widetilde{F}&=& {\rm{div}}[\varepsilon \nabla u_1 \chi_{B(X)}] \ {\rm{in}} \ 2B(X) \\
    \widetilde{F} & = & 0  \ {\rm on} \ \partial (2B(X)).
  \end{array}
\right.
\end{equation}
Using (\ref{tt7.9}), as in \cite{fkp}, we have that
\begin{eqnarray}\label{tt7.10}
  \int_{2B(X)}|\nabla \widetilde{F}_1|^2dZ &\le
& C \int_{2B(X)}A_0\nabla\widetilde{F}_1\nabla\widetilde{F}_1dZ=C\int \nabla\widetilde{F}_1 \varepsilon \nabla u_1\chi_B dZ \nonumber \\
   &\le
& \frac{1}{2}\int_{B(X)}|\nabla \widetilde{F}_1|^2dZ+C\varepsilon_0
^2\int_{B(X)}|\nabla u_1|^2dZ.
\end{eqnarray}
Combining Sobolev inequality and (\ref{tt7.10}) we obtain
\begin{equation}\label{tt7.11}
    \int_{2B(X)}|\widetilde{F}_1|^2dZ \le
 C\delta(X)^2 \int_{2B(X)}|\nabla \widetilde{F}_1|^2dZ\le
 C\varepsilon^2_p\delta(X)^2\int_{B(X)}|\nabla u_1|^2dZ.
\end{equation}
Thus since for $Z\in B(X)$, $\delta(Z)\sim\delta(X)$ (\ref{tt7.11}) yields
\begin{equation}\label{tt7.12}
    \bigg({\fint}_{B(X,\frac{\delta(X)}{8})}|\widetilde{F}_1|^2dZ\bigg)^{1/2}\le
 C\bigg({\fint}_{B(X,\frac{\delta(X)}{2})}|\widetilde{F}_1|^2dZ\bigg)^{1/2}\le
 C \varepsilon_0
 \bigg(\int_{B(X)}|\nabla u_1|^2\delta(Z)^{2-n}dZ\bigg)^{1/2}.
\end{equation}
If $X\in\Gamma(Q_0)$ and $Z\in B(X)$ then $Z\in\Gamma_2(Q_0)$ and from (\ref{tt7.12}) we conclude
\begin{equation}\label{tt7.13}
    \bigg({\fint}_{B(X,\frac{\delta(X)}{8})}|\widetilde{F}_1|^2dZ\bigg)^{1/2}\le
 C\bigg({\fint}_{B(X,\frac{\delta(X)}{2})}|\widetilde{F}_1|^2dZ\bigg)^{1/2}\le
 C \varepsilon_0
 S_2(u_1)(Q_0).
\end{equation}
We now estimate $\widehat{F}_1$ by writing
\begin{equation}\label{tt7.14half}
\widehat{F}_1=F_1-\widetilde{F}_1=\int_{B(X)}\nabla_Y K(Z,Y)\varepsilon\nabla u_1(Y)dY.
\end{equation}
That is
$$ |\widehat{F}_1(Z)| \le
 \varepsilon_0
 \int_{B(X)}|\nabla_Y K(Z,Y)||\nabla u_1(Y)|dY.$$
For fixed $Z\in B(X)$ we have that $L_0 K(Z,Y)=0$ in $2B(X).$  Apply Cauchy-Schwarz and Cacciopoli's inquality (to $K$) we obtain
\begin{equation}\label{tt7.15}
|\widehat{F}_1(Z)| \le
 \frac{C\varepsilon_0
}{\delta(X)}\left( \int_{\frac{3}{2}B(X)}|K(Z,Y)|^2dY\right)^{1/2} \left( \int_{B(X)}|\nabla u_1(Y)|^2 dY\right)^{1/2}.
\end{equation}
Since $K(Z, -)\ge 0$ Harnack's inequality yields,
\begin{equation}\label{tt7.15half}
\left( {\dashint} _{\frac{3}{2}B(X)}K(Z,Y)^2dY \right)^{1/2} \le
 C\left( {\dashint} _{\frac{3}{2}B(X)}K(Z,Y)dY \right)\le
 C {\dashint} _{\frac{3}{2}B(X)}|Z-Y|^{2-n}dY
\end{equation}
since $G_0(Z,Y) \lesssim \frac{1}{|Z-Y|^{n-2}}.$ Thus since for $Y\in B(X)$, $\delta(X)\sim\delta(Y)$ combining (\ref{tt7.15}) and (\ref{tt7.15half}) we have
\begin{eqnarray}\label{tt7.16}
  \bigg({\fint}_{2B(X)}|\widehat{F}_1(Z)|^2dZ\bigg)^{1/2} &\le
& \frac{C\varepsilon_0
}{\delta(X)}\bigg( \int_{2B(X)}{\bigg({\fint}_{\frac{3}{2}B(X)}\frac{dY}{|Z-Y|^{n-2}}\bigg)}^2dZ\bigg)^{1/2}\bigg(\int_{B(X)}|\nabla u_1(Y)|^2 dY\bigg)^{1/2} \nonumber\\
   &\le
& C\varepsilon_0
\delta(X)^{1-n/2}\bigg(\int_{B(X)}|\nabla u_1(Y)|^2 dY\bigg)^{1/2} \nonumber\\
   &\le
& C\varepsilon_0
\bigg(\int_{B(X)}|\nabla u_1(Y)|^2\delta(Y)^{2-n} dY\bigg)^{1/2} \nonumber\\
   &\le
& C \varepsilon_0
 S_2(u_1)(Q_0).
\end{eqnarray}
Combining (\ref{tt-hat}), (\ref{tt7.13}) and (\ref{tt7.16}) we obtain
\begin{equation}\label{tt7.17}
  \bigg({\fint}_{B(X,\frac{\delta(X)}{2})}|F_1(Z)|^2dZ\bigg)^{1/2}\le
 C \varepsilon_0
 S_2(u_1)(Q_0).
\end{equation}

Next we give a pointwise estimate for $F_2(Z)$ when $Z\in B(X,\delta(X)/8)$. Note that in this case $Z$ is away from the pole of the Green's function that appears as an integrand in the definition of $F_2$. To estimate $F_2(Z)$ for $Z\in B(X,\delta(X)/8)$ we consider two cases: $\delta(X)\le
 4R_0$ and $\delta(X)>4R_0$. In the second case we use Lemma \ref{claim3-p7}.

Assume that $\delta(X)\le
 4R_0$ and let $Q_X\in\partial\Omega$ be such that $|X-Q_X|=\delta(X)$. Let $\Omega_0=\Omega\cup B(Q_X,\frac{\delta(X)}{2})$ and $\Delta_0=\partial\Omega\cap B(Q_X,\delta(X)/2)$. For $j\ge
 1$ define $\Omega_j=\Omega\cap B(Q_X, 2^{j-1}\delta(X))$ with $j=1,...,N$ and $2^{14}R_0\le
 2^{N-1}\delta(X)<2^{15}R_0$. Let $\widetilde{R}_j=\Omega_{2j}\setminus \Omega_{2j-2}$, $\Delta_j=\partial\Omega\cap B(Q_X,2^{j-1}\delta(X))$ and $A_j=A(Q_X,2^{j-1}\delta(X))\in\Omega_j$. We now follow the argument that appears in \cite{fkp} using the dyadic surface cubes constructed by M. Christ and described above (see (\ref{tt7.19})) and their interior projections (see (\ref{tt7.26})). Note that $\Omega_0\subset \bigcup_{Q_\alpha^k\subset3\Delta_0}I_\alpha^k$. In fact if $Y\in\Omega_0$ then $\delta(Y)\le
 |Y-Q_X|<\frac{\delta(X)}{2}<2R_0$. As in the proof of (\ref{tt7.26A}) there exists $k\ge
 2$ so that $\frac{\lambda}{8}<8^k\delta(Y)<8\lambda$ and $Q_Y\in\partial\Omega$ with $|Y-Q_Y|=\delta(Y)$. For $\rho_0=\min\{\frac{\delta(Y)-\lambda8^{-k-1}}{2},\frac{\lambda8^{-k+1}-\delta(Y)}{2}\}$ there exists $Q_\alpha^k$ so that $P_\alpha\in Q_\alpha^k\cap\Delta(Q_Y,\rho_0)$ and $Y\in I_\alpha^k$. For any $P\in Q_\alpha^k$,
\begin{eqnarray}\label{tt7.201}
  |P-Q_X| &\le
& |P-P_\alpha|+|P_\alpha-Q_Y|+|Q_Y-Y|+|Y-Q_X| \nonumber\\
   &\le
& {\rm{diam}}\,
Q_\alpha^k+\rho_0+\delta(Y)+\frac{\delta(X)}{2} \nonumber\\
   &<& C_08^{-k}+\rho_0+\delta(Y)+\frac{\delta(X)}{2}  \nonumber\\
   &<&  \frac{8C_0}{\lambda}\delta(Y)+\delta(Y)+\frac{\delta(X)}{2}+\rho_0 \nonumber\\
   &<&  \frac{9}{8}\delta(Y)+\frac{\delta(X)}{2}+\frac{\delta(Y)}{2}\nonumber\\
   &<&  2\delta(Y)+\frac{\delta(X)}{2}< \delta(X)+\frac{\delta(X)}{2}<\frac{3\delta(X)}{2}
\end{eqnarray}
which implies that $Q_\alpha^k\subset3\Delta_0$. We now estimate $F_2(Z)$ for $Z\in B(X,\delta(X)/8)$ as follows
\begin{equation}\label{tt7.202}
    |F_2(Z)|\le
 \bigg|\int_{\Omega_0}\nabla_YG_0(Z,Y)\varepsilon(Y)\nabla u_1(Y)dY\bigg|+\sum_{j=1}^N\bigg|\int_{\widetilde{R}_j\cap(\Omega\setminus B(X))}\nabla_YG_0(Z,Y)\varepsilon(Y)\nabla u_1(Y)dY\bigg|
\end{equation}
$$+\int_{(\Omega\setminus B(X))\cap(\partial\Omega,4R_0)\setminus B(Q_X,2^{15}R_0) }\bigg|\nabla_YG_0(Z,Y)\varepsilon(Y)\nabla u_1(Y)\bigg|dY.$$
We estimate each term separately. To estimate the first term we note that
\begin{eqnarray}\label{tt7.303}
   |F_2^0(Z)|&=& \bigg|\int_{\Omega_0}\nabla_YG_0(Z,Y)\varepsilon(Y)\nabla u_1(Y)dY\bigg| \nonumber\\
   &\le
& \int_{\Omega_0}|\nabla_YG_0(Z,Y)||\varepsilon(Y)||\nabla u_1(Y)|dY \nonumber\\
   &\le
&  \lim_{\varepsilon\rightarrow0^+} \int_{\Omega_0\setminus (\partial\Omega,\varepsilon)}|\nabla_YG_0(Z,Y)||\varepsilon(Y)||\nabla u_1(Y)|dY.
\end{eqnarray}
The goal is to estimate $$F_2^\varepsilon(Z)=\int_{\Omega_0\setminus (\partial\Omega,\varepsilon)}|\nabla_YG_0(Z,Y)||\varepsilon(Y)||\nabla u_1(Y)|dY$$
independently of $\varepsilon>0$. In particular
\begin{equation}\label{tt7.304}
    F^\varepsilon_0
(Z)\le
 \sum_{\substack{Q_\alpha^k\subset3\Delta_0 \\ \varepsilon<\lambda8^{-k-1}}}\sup_{I_\alpha^k}|\varepsilon(Y)|\bigg(\int_{I_\alpha^k}|\nabla_Y G_0(Z,Y)|^2dY\bigg)^{1/2}\bigg(\int_{I_\alpha^k}|\nabla u_1(Y)|^2dY\bigg)^{1/2}.
\end{equation}
By Cacciopoli's inequality
\begin{equation}\label{tt7.204}
    \bigg(\int_{I_\alpha^k}|\nabla_Y G_0(Z,Y)|^2dY\bigg)^{1/2}\le
\frac{C}{{\rm{diam}}\,
Q_\alpha^k}\bigg(\int_{\widehat{I}_\alpha^k}| G_0(Z,Y)|^2dY\bigg)^{1/2}
\end{equation}
where $\widehat{I}_\alpha^k=\{Y\in\Omega:\exists Z\in I^k_\alpha,\ |Z-Y|<\frac{\delta(Z)}{8^4}\}$. By the comparison principle for NTA domains, the Harnack principle and the doubling properties of $\omega^Z$ and $\omega_0$ we have for $Y\in\widehat{I}^k_\alpha$
\begin{equation}\label{tt7.205}
    \frac{G_0(Z,Y)}{G_0(Y)}\sim\frac{G_0(Z,A_0)}{G_0(A_0)}\sim\frac{\omega^Z(\Delta_0)}{\omega_0(\Delta_0)}.
\end{equation}
Thus for $Y\in\widehat{I}^k_\alpha$ with $Q_\alpha^k\subset3\Delta_0$
\begin{equation}\label{tt7.206}
    \frac{G_0(Z,Y)}{G_0(Y)}\le
 \frac{C}{\omega_0(\Delta_0)}.
\end{equation}
Combining (\ref{tt7.304}), (\ref{tt7.204}), (\ref{tt7.205}) and (\ref{tt7.206}) we have that
\begin{equation}\label{tt7.207}
    |F_2^\varepsilon(Z)|\lesssim \sum_{\substack{Q_\alpha^k\subset3\Delta_0\\ k\le
 k_\varepsilon}}\frac{1}{\omega_0(\Delta_0)}\bigg(\int_{\widehat{I}^k_\alpha}\frac{G_0^2(Y)a^2(Y)}{\delta(Y)^2}dY\bigg)^{1/2}\bigg(\int_{I_\alpha^k}|\nabla u_1(Y)|^2dY\bigg)^{1/2}.
\end{equation}
Note that for $Y\in \widehat{I}^k_\alpha$, there exists $Z\in I_\alpha^k$ and $P_\alpha\in Q_\alpha^k$ so that $\frac{\lambda}{8}<|Z-P_\alpha|<\lambda 8$, $\lambda8^{-k-1}-\lambda8^{-k-3}<|Y-P_\alpha|<\lambda8^{-k+1}+\lambda8^{-k-3}$ and $\delta(Z)(1-8^{-4})<\delta(Y)<\delta(Z)(1+8^{-4})$. That is
$|P_\alpha-Q_Y|\le
 |P_\alpha-Z|+|Z-Y|\le
 \lambda8^{-k+1}+\delta(Z)8^{-4}\le
 \delta(Z)(64+8^{-4})\le
 65\delta(Y)$. Now using the doubling property of $\omega_0$ we have
\begin{equation}\label{tt7.208}
    G_0(Y)\sim\frac{\omega_0(\Delta(Q_Y,\delta(Y)))}{\delta(Y)^{n-2}}\sim\frac{\omega_0(\Delta(P_\alpha,\delta(Y)))}{\delta(Y)^{n-2}}.
\end{equation}
Recall that there exists $Z_\alpha^k\in\partial\Omega$ such that $\Delta(Z_\alpha^k,8^{-k-1})\subset Q_\alpha^k\subset \Delta(Z^k_\alpha, 2C_08^{-k})$ (see the construction of the $Q_\alpha^k$) and $|P_\alpha-Z_\alpha^k|\le
 {\rm{diam}}\,
\,
\,
Q_\alpha^k\le
 C_08^{-k}\sim\delta(Y)$. Again by the doubling property of $\omega_0$ we have that (\ref{tt7.208}) yields for $Y\in I_\alpha^k$
\begin{equation}\label{tt7.209}
    G_0(Y)\sim\frac{\omega_0(\Delta(Z_\alpha^k,8^{-k-1}))}{(8^{-k})^{n-2}}\sim\frac{\omega_0(Q_\alpha^k)}{({\rm{diam}}\,
Q_\alpha^k)^{n-2}}
\end{equation}
and combining (\ref{tt7.207}) with (\ref{tt7.209}) we have
\begin{equation}\label{tt7.210}
    |F_2^\varepsilon(Z)|\le
 \sum_{\substack{Q_\alpha^k\subset3\Delta_0 \\ k+1\le
\frac{{\rm{log}}\lambda-{\rm{log}}\varepsilon}{8}}}\frac{1}{\omega_0(\Delta_0)}\bigg(\int_{\widehat{I}^k_\alpha}\frac{G_0(Y)a^2(Y)}{\delta(Y)^2}dY\bigg)^{1/2}
    \bigg(\frac{\omega_0(Q_\alpha^k)}{({\rm{diam}}\,
Q_\alpha^k)^{n-2}}\int_{I_\alpha^k}|\nabla u_1(Y)|^2dY\bigg)^{1/2}.
\end{equation}

To finish the estimate of $F_2^\varepsilon(Z)$ for $Z\in B(X,\delta(X)/8)$ we use a ``stopping time" argument. For $j\in \mathbb{Z}$ let $M\ge
64(1+C_0)$ and
$$O_j=\{Q\in3\Delta_0: T_\varepsilon u_1(Q)=\bigg(\int_{(\Gamma_{M}(Q)\setminus B_{2\varepsilon}(Q))\cap(\partial\Omega,4R_0)}|\nabla u_1(Y)|^2\delta(Y)^{2-n}dY\bigg)^{1/2}>2^j\}.$$
We say that a surface cube $Q_\alpha^k$ in the dyadic grid belongs to $J_j$ if
\begin{equation}\label{tt7.211}
\omega_0(Q_\alpha^k\cap O_j)\ge
\frac{1}{2}\omega_0(Q_\alpha^k)\ \ {\rm{and}}\  \ \omega_0(Q_\alpha^k\cap O_{j+1}) < \frac{1}{2}\omega_0(Q_\alpha^k)
\end{equation}
and it belongs to $J_\infty$ if
\begin{equation}\label{tt7.211A}
    \omega_0(Q_\alpha^k\cap \{T_\varepsilon u_1(Q)=0\})\ge
\frac{1}{2}\omega_0(Q_\alpha^k).
\end{equation}
Note that there exists $0<c_0<1$ depending on the doubling constant of $\o_0$ so that for $\widetilde{O}_j=\{M_{\omega_0}(\chi_{O_j})>c_0\}$, if $Q_\alpha^k\in J_j$ then $Q_\alpha^k\subset \widetilde{O}_j$ and
\begin{equation}\label{tt-7.69A}
\omega_0(Q_\alpha^k\cap \widetilde{O}_j\setminus O_{j+1})\ge
\frac{1}{2}\omega_0(Q_\alpha^k).
\end{equation}
In fact, for $Q_\alpha^k$ there exists $Z^k_\alpha\in Q_\alpha^k$ so that
\begin{equation}\label{tt7.212}
    \Delta(Z_\alpha^k,8^{-k-1})\subset Q_\alpha^k \subset \Delta(Z_\alpha^k,2C_08^{-k}).
\end{equation}
Moreover if $Q_\alpha^k\in J_j$ for $P\in Q_\alpha^k$, $|Z_\alpha^k-P|\le
 {\rm{diam}}\,
Q_\alpha^k\le
 C_08^{-k}$ thus
\begin{equation}\label{tt7.213}
    \Delta(Z_\alpha^k,2C_08^{-k})\subset \Delta(P,3C_08^{-k})\subset \Delta(Z_\alpha^k,4C_08^{-k})
\end{equation}
and by (\ref{tt7.212}) and the doubling property of $\omega_0$ we have
\begin{eqnarray}\label{tt7.214}
  M_{\omega_0}(\chi_{O_j})(P) &\ge
& \frac{\omega_0(\Delta(P,3C_08^{-k})\cap O_j)}{\omega_0(\Delta(P,3C_08^{-k})} \nonumber\\
   &\ge
& \frac{\omega_0(\Delta(Z_\alpha^k,2C_08^{-k})\cap O_j)}{\omega_0(\Delta(Z_\alpha^k,4C_08^{-k}))} \nonumber\\
   &\gtrsim& \frac{\omega_0(Q_\alpha^k\cap O_j)}{\omega_0(\Delta(Z_\alpha^k,8^{-k-1}))} \nonumber\\
   &\gtrsim& \frac{\omega_0(Q_\alpha^k\cap O_j)}{\omega_0(Q_\alpha^k)}\ge
 c_0.
\end{eqnarray}
We conclude that if $Q_\alpha^k\in J_j$ then $Q_\alpha^k \subset \widetilde{O}_j$. Since  $O_{j+1}\subset O_j\subset \widetilde{O}_j$
\begin{eqnarray}\label{tt-7.72A}
\omega_0(Q_\alpha^k\cap\widetilde{O}_j\setminus O_{j+1}) &=&\omega_0(Q_\alpha^k\cap O_{j+1}^c)\\
&=&\omega_0(Q_\alpha^k)-\omega_0(Q_\alpha^k\cap O_{j+1})>\frac{1}{2}\omega_0(Q_\alpha^k),\nonumber
\end{eqnarray}
which ensures that $Q_\alpha^k\subset\{Q\in \partial\Omega: M_{\omega_0}(\chi_{\widetilde{O}_j\setminus O_{j+1}})(Q)>c_0\}=U_j$. 
A weak type inequality for $M_{\omega_0}$ applied to 
$(\chi_{\widetilde{O}_j\setminus O_{j+1}}$ and $\chi_{O_j}$
yields
\begin{equation}\label{tt-7.72B}
\omega_0(U_j)\le
 C\omega_0(\widetilde{O}_j\setminus O_{j+1})\le
 C\omega_0(O_j).
 \end{equation} 
Note that for each $\varepsilon>0$ $T_\varepsilon u_1(Q)$ is bounded. Thus for $Q_\alpha^k\subset 3\Delta_0$ either
$T_\varepsilon u_1\equiv 0$ or
there exists $j_0$ so that $$2^{j_0-1}\le
 {\fint}_{Q_\alpha^k}T_\varepsilon u_1(Q)d\omega_0(Q)<2^{j_0}.$$
 In the first case $Q_\alpha^k\in J_\infty$, in the second $\omega_0(Q_\alpha^k\cap O_{j})<\frac{1}{2}\omega_0(Q_\alpha^k)$ for $j\ge
 j_0$.
 Furthermore either there exists $j<j_0$ so that (\ref{tt7.211}) is satisfied or for all $l\in \mathbb{Z}$, $\omega_0(Q_\alpha^k\cap O_l)<\frac{1}{2}\omega_0(Q_\alpha^k)$ which implies that $\omega_0(Q_\alpha^k\cap \{T_\varepsilon u_1(Q)=0\})\ge
 \frac{1}{2}\omega(Q_\alpha^k)$. In this case $Q_\alpha^k\in J_\infty$ and $$Q_\alpha^k\subset \{Q\in \partial\Omega: M_{\omega_0}(\chi_{\{T_\varepsilon u_1=0\}})(Q)>c_0\}=U_\infty.$$
As above a weak type inequality on the maximal function yields that $\omega_0(U_\infty)\le
 C\omega_0(O_\infty)$ where $O_\infty=\{Q\in \partial\Omega: T_\varepsilon u_1(Q)=0\}$. Note that if $Q_\alpha^k\in J_\infty$, $\omega_0(Q_\alpha^k\cap O_\infty)\ge
 \frac{1}{2}\omega_0(Q_\alpha^k)$.

We now go back to our estimate of $F_2^\varepsilon$ for $Z\in B(X,\delta(X/8)$. Combining (\ref{tt7.210}), Cauchy-Schwarz inequality, and letting $k_\varepsilon=\frac{{\rm{log}}\lambda-{\rm{log}}\varepsilon}{8}-1$ we have
$$|F_2^\varepsilon(Z)|\le
 \frac{1}{\omega_0(\Delta_0)}\sum_j\sum_{k\le
 k_\varepsilon}\sum_{Q_\alpha^k\in J_j}\bigg(\int_{{I}^k_\alpha}\frac{G_0(Y)a^2(Y)}{\delta(Y)^2}dY\bigg)^{1/2}
    \bigg(\frac{\omega_0(Q_\alpha^k)}{({\rm{diam}}\,
Q_\alpha^k)^{n-2}}\int_{I_\alpha^k}|\nabla u_1(Y)|^2dY\bigg)^{1/2}$$
$$+\frac{1}{\omega_0(\Delta_0)}\sum_{k\le
 k_\varepsilon}\sum_{Q_\alpha^k\in J_\infty}\bigg(\int_{{I}^k_\alpha}\frac{G_0(Y)a^2(Y)}{\delta(Y)^2}dY\bigg)^{1/2}
    \bigg(\frac{\omega_0(Q_\alpha^k)}{({\rm{diam}}\,
Q_\alpha^k)^{n-2}}\int_{I_\alpha^k}|\nabla u_1(Y)|^2dY\bigg)^{1/2}$$
$$\lesssim \frac{1}{\omega_0(\Delta_0)}\sum_j\bigg(\sum_{k\le
 k_\varepsilon}\sum_{Q_\alpha^k\in J_j}\int_{{I}^k_\alpha}\frac{G_0(Y)a^2(Y)}{\delta(Y)^2}dY\bigg)^{1/2}\bigg(\sum_{k\le
 k_\varepsilon}\sum_{Q_\alpha^k\in J_j}
    \frac{\omega_0(Q_\alpha^k)}{({\rm{diam}}\,
Q_\alpha^k)^{n-2}}\int_{I_\alpha^k}|\nabla u_1(Y)|^2dY\bigg)^{1/2}$$
$$+\frac{1}{\omega_0(\Delta_0)}\bigg(\sum_{k\le
 k_\varepsilon}\sum_{Q_\alpha^k\in J_\infty}\int_{{I}^k_\alpha}\frac{G_0(Y)a^2(Y)}{\delta(Y)^2}dY\bigg)^{1/2}\bigg(\sum_{k\le
 k_\varepsilon}\sum_{Q_\alpha^k\in J_\infty}
    \frac{\omega_0(Q_\alpha^k)}{({\rm{diam}}\,
Q_\alpha^k)^{n-2}}\int_{I_\alpha^k}|\nabla u_1(Y)|^2dY\bigg)^{1/2}.$$
\begin{equation}\label{tt7.215}
\end{equation}
Note that if $Q_\alpha^k$, $Q_\beta^l\in J_j$ and $Q_\alpha^k\cap Q_\beta^l\neq \emptyset$ then either $Q_\alpha^k\subset Q_\beta^l$ or $Q_\beta^l \subset Q_\alpha^k$. Since $Q_\alpha^k \subset \Delta(Z_\alpha^k, C_08^{-k})$ by construction, then for $Y\in I^l_\beta$ there exists $P\in Q_\beta^l$ so that $\lambda/8<8^k|P-Y|<8\lambda$ and $|Y-Z_\alpha^k|\le
 C_08^{-k}+\lambda8^{-k+1}$ thus $Y\in T(\Delta(Z_\alpha^k, (C_0+8\lambda)8^{-k}))$. Furthermore since $\Delta(Z_\alpha^k, 8^{-k-1})\subset Q_\alpha^k$, $\omega_0$ is doubling, the $I_\alpha^k$'s have finite overlap and (\ref{tt-7.72A}) we conclude that \begin{eqnarray}\label{tt7.216}
  \sum_{k\le
 k_\varepsilon}\sum_{Q_\alpha^k\in J_j}\int_{{I}^k_\alpha}\frac{G_0(Y)a^2(Y)}{\delta(Y)^2}dY &\lesssim& \sum_{\substack{Q_\alpha^k\in J_j \\ Q_\alpha^k \ {\rm{disjoint}}}}\int_{T(\Delta(Z_\alpha^k, (C_0+8\lambda)8^{-k}))}\frac{G_0(Y)a^2(Y)}{\delta(Y)^2}dY \nonumber\\
   &\lesssim& \varepsilon^2_p \sum_{\substack{Q_\alpha^k\in J_j \\ Q_\alpha^k \ {\rm{disjoint}}}}  \omega_0(Q_\alpha^k) \lesssim \varepsilon^2_p \omega_0(O_j).
\end{eqnarray}
Similarly we obtain
\begin{equation}\label{tt7.217}
  \sum_{k\le
 k_\varepsilon}\sum_{Q_\alpha^k\in J_\infty}\int_{{I}^k_\alpha}\frac{G_0(Y)a^2(Y)}{\delta(Y)^2}dY \lesssim \varepsilon^2_p \sum_{\substack{Q_\alpha^k\in J_\infty \\ Q_\alpha^k \ {\rm{disjoint}}}}  \omega_0(Q_\alpha^k) \lesssim \varepsilon^2_p \omega_0(O_\infty).
\end{equation}
To estimate the other term note that $I_\alpha^k \subset \Gamma_M(P)$ for all $P\in Q_\alpha^k$ and $M>64+8C_0$. In fact if $Y\in I_\alpha^k$, $\delta(Y)>\lambda 8^{-k-1}$ and there is $P'\in Q_\alpha^k$ so that $|Y-P'|<\lambda8^{-k+1}$, thus for $P\in Q_\alpha^k$ $|Y-P|<\lambda8^{-k+1}+{\rm{diam}}\,
Q_\alpha^k\le
 \lambda8^{-k+1}+C_08^{-k}<\delta(Y)(1+M)$ and $Y\in\Gamma_M(P)$. So if $Q_\beta^l \subset Q_\alpha^k$, $I_\alpha^k\subset \Gamma_M(P)$ for every $P\in Q_\beta^l$ and since the $I^k_\alpha$'s have finite overlap then denoting by $S_M^{\varepsilon,l}=(\Gamma_M(Q)\setminus B_\varepsilon(Q))\cap (\partial\Omega,4R_0)\cap\{\frac{\lambda}{8}<8^l\delta(Y)<8\lambda\}$ we have using (\ref{tt-7.72A})
$$\sum_{k\le
 k_\varepsilon}\sum_{Q_\alpha^k\in J_j}
    \frac{\omega_0(Q_\alpha^k)}{({\rm{diam}}\,
Q_\alpha^k)^{n-2}}\int_{I_\alpha^k}|\nabla u_1(Y)|^2dY\lesssim \sum_{k\le
 k_\varepsilon}\sum_{Q_\alpha^k\in J_j}
    \omega_0(Q_\alpha^k)\int_{I_\alpha^k}|\nabla u_1(Y)|^2\delta(Y)^{2-n}dY$$
    \begin{eqnarray}\label{tt7.218}
       &\lesssim&  \sum_{\substack{Q_\alpha^k\in J_j \\ Q_\alpha^k \ {\rm{disjoint}} \\ k\le
 k_\varepsilon}}\sum_{\substack{Q_\beta^l\subset Q_\alpha^k \\ Q_\beta^l\in J_j}}\omega_0(Q_\beta^l)\int_{I_\beta^l}|\nabla u_1(Y)|^2\delta(Y)^{2-n}dY\nonumber\\
       &\lesssim&  \sum_{\substack{Q_\alpha^k\in J_j \\ Q_\alpha^k \ {\rm{disjoint}} \\ k\le
 k_\varepsilon}}\sum_{\substack{Q_\beta^l\subset Q_\alpha^k \\ Q_\beta^l\in J_j}}\omega_0(\widetilde{O}_j\setminus O_{j+1}\cap Q_\beta^l)\int_{I_\beta^l}|\nabla u_1(Y)|^2\delta(Y)^{2-n}dY\nonumber\\
       &\lesssim&  \sum_{\substack{Q_\alpha^k\in J_j \\ Q_\alpha^k \ {\rm{disjoint}} \\ k\le
 k_\varepsilon}}\sum_{\substack{Q_\beta^l\subset Q_\alpha^k \\ Q_\beta^l\in J_j}}\int_{\widetilde{O}_j\setminus O_{j+1}\cap Q_\beta^l}\int_{S_M^{\varepsilon,l}}|\nabla u_1(Y)|^2\delta(Y)^{2-n}dYd\omega_0(Q)\nonumber\\
       &\lesssim&  \sum_{\substack{Q_\alpha^k\in J_j \\ Q_\alpha^k \ {\rm{disjoint}} \\ k\le
 k_\varepsilon}}\sum_{l}\int_{\widetilde{O}_j\setminus O_{j+1}\cap Q_\alpha^k}\int_{S_M^{\varepsilon,l}}|\nabla u_1(Y)|^2\delta(Y)^{2-n}dYd\omega_0(Q)\nonumber\\
       &\lesssim&  \sum_{\substack{Q_\alpha^k\in J_j \\ Q_\alpha^k \ {\rm{disjoint}} \\ k\le
 k_\varepsilon}}\int_{\widetilde{O}_j\setminus O_{j+1}\cap Q_\alpha^k}T_\varepsilon u_1(Q)^2d\omega_0(Q)\nonumber\\
       &\lesssim&  \int_{\widetilde{O}_j\setminus O_{j+1}}T_\varepsilon u_1(Q)^2d\omega_0(Q).
    \end{eqnarray}
Note that if $Q_\alpha^k\in J_\infty$ for $k\le
 k_\varepsilon$  since $I_\alpha^k\subset \Gamma_M(P)\setminus B_\varepsilon(P)$ then
$$\int_{I_\alpha^k}|\nabla u_1|^2\delta(Y)^{2-n}dY\le
 T_\varepsilon u_1(Q)^2=0.$$
Combining this remark with (\ref{tt-7.72B}), (\ref{tt7.215}), (\ref{tt7.216}) and (\ref{tt7.218}) we conclude that for $Z\in B(X,\delta(X)/8)$
\begin{eqnarray}\label{tt7.219}
  |F_2^\varepsilon(Z)| &\lesssim& \frac{\varepsilon_0
}{\omega_0(\Delta_0)}\sum_j\omega_0(O_j)^{1/2}\bigg(\int_{\widetilde{O}_j\setminus O_{j+1}}T_\varepsilon u_1(Q)^2d\omega_0(Q)\bigg)^{1/2} \nonumber\\
   &\lesssim& \frac{\varepsilon_0
}{\omega_0(\Delta_0)}\sum_j2^j\omega_0(O_j)^{1/2}\omega_0(\widetilde{O}_j\setminus O_{j+1})^{1/2} \nonumber\\
   &\lesssim&  \frac{\varepsilon_0
}{\omega_0(\Delta_0)}\sum_j2^j\omega_0(O_j)^{1/2}\omega_0(U_j)^{1/2}\nonumber\\
   &\lesssim&  \frac{\varepsilon_0
}{\omega_0(\Delta_0)}\sum_j2^j\omega_0(O_j)\nonumber\\
   &\lesssim&  \frac{\varepsilon_0
}{\omega_0(\Delta_0)}\int_{\Delta_0}T_\varepsilon u_1(Q)d\omega_0(Q).
\end{eqnarray}
The last inequality comes from the fact that $\sum2^j\omega_0(O_j)=\sum2^j\omega_0(O_j\backslash O_{j+1}) + \sum_j2^j\omega_0(O_j)$ which ensures that $\sum 2^j\omega_0(O_j)=\frac{1}{2}\sum2^j\omega_0(O_j\backslash O_{j+1}) $.
It is important to note that at each step the constants involved are independent of $\varepsilon>0$. Combining (\ref{tt7.303}), (\ref{tt7.219}) as well as the doubling property of $\omega_0$ we have that
\begin{eqnarray}\label{tt7.220}
  |F_2^0(Z)| &\le
& C\lim_{\varepsilon\rightarrow 0} \frac{\varepsilon_0
}{\omega_0(\Delta_0)}\int_{\Delta_0}T_\varepsilon u_1(Q)d\omega_0(Q)\nonumber \\
   &\le
& C\frac{\varepsilon_0
}{\omega_0(\Delta_0)}\int_{\Delta_0}S_M(u_1)(Q)d\omega_0(Q)\nonumber \\
   &\le
& C \varepsilon_0
M_{\omega_0}(S_M(u_1))(Q).
\end{eqnarray}

To estimate the second part of (\ref{tt7.202}) recall that for $j\ge
 1$, $\widetilde{R}_j=\Omega_{2j}\setminus\Omega_{2j-2}$, $\Delta_j=\Delta(Q_X, 2^{j-1}\delta(X))$, $\Omega_j=B(Q_X, 2^{j-1}\delta(X))\cap\Omega$ and $A_j=A(Q_X, 2^{j-1}\delta(X))\in \Omega_j$. Denote $R_j=\widetilde{R}_j\setminus B(X)$. To estimate
\begin{equation}\label{tt7.221}
    F_2^j(Z)=\int_{R_j}\nabla_Y G_0(Z,Y)\varepsilon(Y)\nabla u_1(Y)dY
\end{equation}
for $Z\in B(X,\delta(X)/8)$ divide $R_j$ as follows
\begin{equation}\label{tt7.222}
    R_j=R_j\cap(\partial\Omega,2^{2j-6}\delta(X))\cup R_j\setminus (\partial\Omega,2^{2j-6}\delta(X)).
\end{equation}
Let $V_j=R_j\cap(\partial\Omega,2^{2j-6}\delta(X))$ and $W_j=R_j\setminus (\partial\Omega,2^{2j-6}\delta(X))$. Note that $V_j\subset \bigcup_{Q_\alpha^k\subset 3\Delta_{2j}\setminus\frac{1}{3}\Delta_{2j-2}}I_\alpha^k$. In fact if $Y\in V_j$ then $2^{2j-3}\delta(X)\le
|Y-Q_X|<2^{2j-1}\delta(X)$, $\delta(Y)<4R_0$ and there exists $k$ such that $8^{-k-1}\lambda<\delta(Y)<8^{-k+1}\lambda$ and $Y\in I_\alpha^k$ for some $\alpha$. For $\rho_0=\frac{1}{2}\min\{\delta(Y)-\lambda8^{-k-1}, \lambda8^{-k+1}-\delta(Y)\}$ there exists $P\in Q_\alpha^k\cap\Delta(Q_Y,\rho_0)$ such that
$$|P-Q_X|\ge
 |Q_X-Q_Y|-|P-Q_Y|\ge
 |Y-Q_X|-|Y-Q_Y|-|P-Q_Y|$$ $$\ge
 2^{2j-3}\delta(X)-\rho_0-\delta(Y)\ge
 2^{2j-3}\delta(X)-\frac{3}{2}\delta(Y)\ge
 (2^{2j-3}-2^{2j-5})\delta(X)> \frac{2^{2j-3}}{3}\delta(X).$$
Following the same pattern of the proof above we have for $Z\in B(X,\delta(X)/8)$
\begin{equation}\label{tt7.223}
    \int_{V_j}|\nabla_Y G_0(Z,Y)||\varepsilon(Y)||\nabla u_1(Y)|dY\le
 \lim_{\varepsilon\rightarrow 0}\int_{V_j\setminus(\partial\Omega,\varepsilon)}|\nabla_Y G_0(Z,Y)||\varepsilon(Y)||\nabla u_1(Y)|dY
\end{equation}
and
\begin{equation}\label{tt7.224}
    \int_{V_j\setminus(\partial\Omega,\varepsilon)}|\nabla_Y G_0(Z,Y)||\varepsilon(Y)||\nabla u_1|dY\lesssim \sum_{\substack{Q_\alpha^k\subset 3\Delta_{2j}\setminus\frac{1}{3}\Delta_{2j-2} \\ k\le
 k_\varepsilon}}
    \bigg(\int_{{I}^k_\alpha}\frac{G_0(Z,Y)^2a^2(Y)}{\delta(Y)^2}dY\bigg)^{1/2}\bigg(\int_{I_\alpha^k}|\nabla u_1|^2dY\bigg)^{1/2}.
\end{equation}
For $Y\in I_\alpha^k\cap V_j$ and $Z\in B(X,\delta(X)/8)$, (\ref{tt7.206}) and the vanishing properties of the Green's function at the boundary of an NTA domain yield
\begin{equation}\label{tt7.225}
    G_0(Z,Y)\lesssim \bigg(\frac{\delta(Z)}{\delta(X)2^{2j-1}}\bigg)^\beta G_0(A_{2j},Y)\lesssim 2^{-2\beta j}G_0(Y)\frac{1}{\omega_0(\Delta_{2j})}.
\end{equation}
Moreover for $Y\in I_\alpha^k$, $G_0(Y)\sim\frac{\omega_0(Q_\alpha^k)}{({\rm{diam}}\,
Q_\alpha^k)^{n-2}}$ as in (\ref{tt7.209}). Thus combining (\ref{tt7.224}), (\ref{tt7.225}) and the remark above we have
$$\int_{V_j\setminus(\partial\Omega,\varepsilon)}|\nabla_Y G_0(Z,Y)||\varepsilon(Y)||\nabla u_1|dY$$
\begin{equation}\label{tt7.226}
    \lesssim \sum_{\substack{Q_\alpha^k\subset 3\Delta_{2j}\setminus\frac{1}{3}\Delta_{2j-2} \\ k\le
 k_\varepsilon}}2^{-2\beta j}\frac{1}{\omega_0(\Delta_{2j})}
    \bigg(\int_{{I}^k_\alpha}\frac{G_0(Y)a^2(Y)}{\delta(Y)^2}dY\bigg)^{1/2}\bigg(\frac{\omega_0(Q_\alpha^k)}{({\rm{diam}}\,
Q_\alpha^k)^{n-2}}\int_{I_\alpha^k}|\nabla u_1|^2dY\bigg)^{1/2}.
\end{equation}
A ``stopping time" argument yields, as in (\ref{tt7.220}) that
\begin{equation}\label{tt7.227}
\int_{V_j}|\nabla_Y G_0(Z,Y)||\varepsilon(Y)||\nabla u_1(Y)|dY \le
 C\varepsilon_0
 2^{-2\beta j}M_{\omega_0}(S_M(u_1))(Q).
\end{equation}
To estimate the corresponding integral over $W_j$, cover $W_j$ with balls $B(X_{jl}, 2^{2j-8}\delta(X))$ such that $X_{jl}\in W_j$ and the
$B(X_{jl}, 2^{2j-10}\delta(X))$'s are disjoint. Since $X_{jl}\in W_j$ $2^{2j-6}\delta(X)\le
 \delta(X_{jl})\le
 2^{2j-1}\delta(X)$, the $B_{jl}=B(X_{jl}, 2^{2j-8}\delta(X))$'s are non-tangential balls and
\begin{eqnarray}\label{tt7.228}
 \int_{W_j}|\nabla_Y G_0(Z,Y)||\varepsilon(Y)||\nabla u_1|dY
      &\le
& \sum_{l}\sup_{B_{jl}}|\varepsilon(Y)|   \bigg(\int_{B_{jl}}\frac{G_0(Z,Y)^2}{\delta(Y)^2}dY\bigg)^{1/2}\bigg(\int_{B_{jl}}|\nabla u_1|^2dY\bigg)^{1/2} \nonumber\\
   &\le
&  \sum_{l}\bigg(\int_{B_{jl}}\frac{a^2(Y)G_0(Z,Y)^2}{\delta(Y)^2}dY\bigg)^{1/2}\bigg(\int_{B_{jl}}|\nabla u_1|^2dY\bigg)^{1/2}
\end{eqnarray}
For $Y\in B_{jl}$, $2^{2j-7}\delta(X)\le \delta(Y)\le 2^{2j} \delta(X)$ and $Z\in B(X,\delta(X)/8)$
\begin{equation}\label{tt7.229}
    G_0(Z,Y)\le
 C2^{-2j \beta}G_0(A_{2j},Y)\le
 C2^{-2j \beta}\frac{G_0(Y)}{\omega_0(\Delta_{2j})}
\end{equation}
and for $Y\in B_{jl}$, $G_0(Y)\le
 G_0(A_{2j})$. Thus combining this with (\ref{tt7.228}) and (\ref{tt7.229}) we obtain
$$\int_{W_j}|\nabla_Y G_0(Z,Y)||\varepsilon(Y)||\nabla u_1|dY$$
$$\le
 C\frac{2^{-2\beta j}}{\omega_0(\Delta_j)}  \sum_{l}\bigg(\int_{B_{jl}}\frac{a^2(Y)G_0(Y)}{\delta(Y)^2}dY\bigg)^{1/2}\bigg(\frac{G_0(A_{2j})}{(2^{2j}\delta(X))^{2-n}}\bigg)^{1/2}\bigg(\int_{B_{jl}}|\nabla u_1|^2\delta(Y)^{2-n}dY\bigg)^{1/2}$$
$$\le
 C\frac{2^{-2\beta j}}{\omega_0(\Delta_j)}\bigg({G_0(A_{2j})}{(2^{2j}\delta(X))^{n-2}}\bigg)^{1/2}  \bigg(\sum_{l}\int_{B_{jl}}\frac{a^2(Y)G_0(Y)}{\delta(Y)^2}dY\bigg)^{1/2}\bigg(\sum_l\int_{B_{jl}}|\nabla u_1|^2\delta(Y)^{2-n}dY\bigg)^{1/2}$$
\begin{equation}\label{tt7.231}
    \le
 C2^{-2\beta j}\bigg(\frac{1}{\omega_0(\Delta_{2j+1})}\int_{\Omega_{2j+1}}\frac{a^2(Y)G_0(Y)}{\delta(Y)^2}dY\bigg)^{1/2}\bigg(\int_{(\Omega_{2j+1}\setminus \Omega_{2j-2})\setminus (\partial\Omega,2^{2j-4}\delta(X))}|\nabla u_1|^2\delta(Y)^{2-n}dY\bigg)^{1/2}
\end{equation}
For $Y\in (\Omega_{2j+1}\setminus \Omega_{2j-2})\setminus (\partial\Omega,2^{2j-4}\delta(X))$ we have
$|Y-Q_0|\le
 |Y-Q_X|+|Q_X-Q_0|\le
 2^{2j}\delta(X)+|Q_X-X|+|X-Q_0|\le
 2^{2j}\delta(X)+\delta(X)+2\delta(X)\le
 2^4\delta(Y)+2^{-2j+4}\delta(Y)+2^{-2j+5}\delta(Y)$ thus $|Y-Q_0|\le
 64\delta(Y)$ and $Y\in\Gamma_{64}(Q_0)$, therefore for $M\ge 64$, (\ref{tt7.231}) yields
\begin{equation}\label{tt7.232}
    \int_{W_j}|\nabla_Y G_0(Z,Y)||\varepsilon(Y)||\nabla u_1|dY \le
 C2^{-2\beta j}\varepsilon_0
 S_M(u_1)(Q_0).
\end{equation}
Combining (\ref{tt7.220}), (\ref{tt7.227}) and (\ref{tt7.232}) we have for $\delta(X)\le
 4R_0$ and $Z\in B(X,\delta(X/8))$
$$\bigg|\int_{(\Omega\setminus B(X))\cap B(Q_X,2^{15}R_0)}\nabla_Y G_0(Z,Y)\varepsilon(Y)\nabla u_1(Y)dY\bigg|$$
\begin{eqnarray}\label{tt7.233}
   &\le
& \int_{\Omega_0}|\nabla_Y G_0(Z,Y)||\varepsilon(Y)||\nabla u_1|dY+\sum_{j=1}^{N}\int_{R_j}|\nabla_Y G_0(Z,Y)||\varepsilon(Y)||\nabla u_1|dY \nonumber\\
   &\le
& C\varepsilon_0
M_{\omega_0}(S_M(u_1))(Q_0)+ C\varepsilon_0
 \sum_{j=1}^{N}2^{-2\beta j}M_{\omega_0}(S_M(u_1))(Q_0)\nonumber\\
   &\le
&  C\varepsilon_0
M_{\omega_0}(S_M(u_1))(Q_0)
\end{eqnarray}
To complete the estimate for $F_2(Z)$ with $Z\in B(X,\delta(X)/8)$ it only remains to consider the integral
\begin{equation}\label{tt7.234}
    \int_{(\Omega\setminus B(X))\cap (\partial\Omega,4R_0)\setminus B(Q_X,2^{15}R_0)}\nabla_Y G_0(Z,Y)\varepsilon(Y)\nabla u_1(Y)dY.
\end{equation}
Note that $$(\Omega\setminus B(X))\cap (\partial\Omega,4R_0)\setminus B(Q_X,2^{15}R_0)\subset \bigcup_{Q_\alpha^k\subset \partial\Omega\setminus \Delta(Q_X,2^{14}R_0)}I_\alpha^k.$$
If $Y\in (\Omega\setminus B(X))\cap (\partial\Omega,4R_0)\setminus B(Q_X,2^{15}R_0)$ then $|Y-Q_X|\ge
 2^{15}R_0$ and there are $I_\alpha^k$ and $P_\alpha^k\in Q_\alpha^k$ so that $Y\in I_\alpha^k$, $8/\lambda<|P_\alpha^k-Y|8^k<8\lambda$. Given $Y'\in I_\alpha^k$ note that $|P_\alpha^k-Y'|\le
 {\rm{diam}}\,
Q_\alpha^k+\lambda8^{-k+1}\le
 C_08^{-k}+\lambda8^{-k+1}\le
 8C_0\delta(Y)/\lambda+64\delta(Y)\le
 65\delta(Y)\le
 2^{10}R_0$. Thus
$I_\alpha^k \subset B(P_\alpha^k,2^{10}R_0)\cap\Omega $,
and for $Y'\in B(P_\alpha^k,2^{10}R_0)\cap\Omega$ we have $|Y'-Q_X|\ge
 |Y-Q_X|-|Y'-Y|\ge
 |Y-Q_X|-|Y-P_\alpha^k|-|Y'-P_\alpha^k|\ge
 2^{15}R_0-\lambda8^{-k+1}-2^{10}R_0\ge
 2^{15}R_0-2^{10}R_0-2^6\delta(Y)\ge
 2^{14}R_0$. Moreover for $P\in Q_\alpha^k$, $|P-Q_X|\ge
 |P_\alpha^k-Q_X|-|P-P_\alpha^k|\ge
 |Y-Q_X|-|Y-P_\alpha^k|-|P-P_\alpha^k|\ge
 2^{15}R_0-\lambda8^{-k+1}-{\rm{diam}}\,
Q_\alpha^k\ge
 2^{14}R_0$. If $Z\in B(X,\delta(X)/8)$ and $Y\in (\Omega\setminus B(X))\cap (\partial\Omega,4R_0)\setminus B(Q_X,2^{15}R_0)$ then since $\delta(X)\le
 4R_0$ we have $|Z-Y|\ge
 |Y-Q_X|-|Q_X-X|-|X-Z|\ge
 2^{15}R_0-\delta(X)-\delta(X)/8\ge
 2^{14}R_0$. Similarly if $Y\in B(P_\alpha^k,2^{10}R_0)\cap\Omega$ we have $|Z-Y|\ge
 |Y-Q_X|-|Q_X-X|-|X-Z|\ge
 2^{13}R_0$. We mimic the stopping time argument used when the integration over the region $\Omega_0$ was considered. The key point is that for $Z\in B(X,\delta(X)/8)$ the pole of the Green's function is far away from the $I_\alpha^k$'s considered in the integration. Since $\partial\Omega\setminus \Delta(Q_X,2^{14}R_0)\subset \partial\Omega\setminus\Delta(Q_0,2^{13}R_0)$ and $\omega_0(\partial\Omega)\le
 C_{R_0}\omega_0(\Delta(P,2^9R_0))$ for any $P\in\partial\Omega$ by the doubling properties of $\omega_0$, estimate (\ref{tt7.206}) becomes $G_0(Z,Y)\le
 C_{R_0}G_0(Y)$ and (\ref{tt7.220}) becomes
\begin{equation*}
 \int_{(\Omega \backslash B(X)) \cap (\partial \Omega , 4R_0) \backslash B(Q_X, 2^{15}R_0)} |\nabla_Y G_0(Z,Y)||\varepsilon(Y)||\nabla u_1(Y)|
\le
 C \varepsilon_0
 \int_{\partial \Omega \backslash \Delta(Q_0,2^{13}R_0)} S_M (u_1)(Q) d\omega_0
\end{equation*}
\begin{equation}\label{tt7.235}
\le
 C \varepsilon_0
 M_{\omega_0} S_M (u_1)(Q_0).
\end{equation}
Combining (\ref{tt7.233}) and (\ref{tt7.235})
and noting that when $\delta(X) > 8 R_0$ the integration over $\Omega \cap (\partial \Omega , 4R_0)$ is treated as that over $\Omega_0$ or $\Omega \backslash B(X) \cap (\partial \Omega , 4R_0) \backslash B(Q_X, 2^{15}R_0)$ as the pole $Z \in B(X, \frac{\delta(X)}{8})$ is very far from the $I_{\alpha}^k$'s we obtain that
\begin{equation}\label{tt7.236}
|F_2(Z)| \le
 C \varepsilon_0
 M_{\omega_0} S_M (u_1)(Q_0), \ \forall Z \in B(X, \frac{\delta(X)}{8}).
\end{equation}
Hence (\ref{tt7.107}), (\ref{tt7.17}) and (\ref{tt7.236}) ) yield
for $M$ large and fixed
\begin{eqnarray*}
 \widetilde{N}F(Q_0)  &=& \sup_{X \in \Gamma(Q_0)} \fint_{B(X, \frac{\delta(X)}{8})} F^2(Z)dZ \\
   & \lesssim &  \sup_{X \in \Gamma(Q_0)} \fint_{B(X, \frac{\delta(X)}{8})} F_1^2(Z)dZ + \sup_{X \in \Gamma(Q_0)} \int_{B(X, \frac{\delta(X)}{8})} F_2^2(Z)dZ
\end{eqnarray*}
\begin{equation}\label{tt7.237}
\lesssim C \varepsilon_0
 M_{\omega_0} (S_M (u_1))(Q_0).
\end{equation}

We now estimate the second term in Lemma \ref{lem2.9}. Fix $Q_0 \in \partial \Omega, \ X \in \Gamma(Q_0),$ let $B(X)=B(X, \frac{\delta(X)}{16})$. Note that $B(X, \frac{\delta(X)}{8}) \subset \Gamma_2(Q_0),$ 
then
\begin{equation*}
\fint_{B(X)}(\delta |\nabla F|)^2(Z)dZ \le
 C \delta^2(X)\frac{1}{\delta(X)^n}\int_{B(X)}|\nabla F|^2(Z)dZ
\end{equation*}
\begin{equation}\label{tt7.238}
 \le
 C \frac{1}{\delta(X)^{n-1}}\int_{\frac{\delta(X)}{16}}^{\frac{\delta(X)}{8}}\int_{B(X,\rho)}|\nabla F|^2(Z)dZd\rho.
\end{equation}
The same argument as in \cite{fkp} which only uses interior estimates yields
\begin{equation}\label{tt7.239}
\widetilde{N}^2_{1/2}(\delta|\nabla F|)(Q_0) \le
 C \widetilde{N}F(Q_0)\widetilde{N} (\delta|\nabla F|)(Q_0) +
                                                  \varepsilon_0
 \widetilde{N}(\delta|\nabla F|)S_2(u_1)(Q_0)+ \varepsilon_0
 \widetilde{N}(F)(Q_0)S_2(u_1)(Q_0).
\end{equation}
Combining (\ref{tt7.239}) with (\ref{ttlem7.6A}) and using the fact $ab \le
 \frac{a^2}{2}+\frac{b^2}{2}$ we obtain (\ref{ttlem7.6B}).
 Integrating (\ref{ttlem7.6B}) and applying Remark \ref{N-hat-tilde} we obtain
\begin{eqnarray}\label{tt7.240}
\int \widetilde{N}(\delta|\nabla F|)^2 d\omega_0 &\le&
C \int \widetilde{N}_{1/2}(\delta|\nabla F|)^2 d\omega_0 \le
 C \varepsilon_0
 \int (M \omega_0 S_M (u_1) )^2d\omega_0  + C \varepsilon_0
 \int \widetilde{N}(\delta|\nabla F|)^2 d\omega_0 \\
 &\le&
 C \varepsilon_0
 \int (S_M (u_1) )^2 d\omega_0 + C \varepsilon_0
 \int \widetilde{N}(\delta|\nabla F|)^2 d\omega_0,\nonumber
 \end{eqnarray}
which yields 
\begin{equation}\label{tt7.241}
\int \widetilde{N}(\delta|\nabla F|)^2 d\omega_0 \le
 C \varepsilon_0
 \int S^2_M (u_1)(Q)  d\omega_0.
\end{equation}
Combining (\ref{tt7.241}), the integration of (\ref{ttlem7.6A}) and the maximal function theorem we obtain (\ref{ttlem7.6C}) which concludes the
proof of Lemma \ref{lem2.9}.
\end{proof}

\begin{lemma}\label{lem2.10}
Let $\Omega$ be a CAD and assume that (\ref{condThm2.11}) holds. Then there exists $C>1$ so that
\begin{equation}\label{ttlem7.8}
||SF||^2_{L^2(\omega_0)} \le
 C\left(||\widetilde{N}(\delta|\nabla F|)||^2_{L^2(\omega_0)}+||NF||^2_{L^2(\omega_0)}+||\widetilde{N}F||^2_{L^2(\omega_0)}+||f||^2_{L^2(\omega_0)}\right)
\end{equation}
\end{lemma}

\begin{proof}
For $s \in [1,2],$ let $\Omega_s=B(0,s\widetilde{R}_0)$ where $\widetilde{R}_0=\frac{\delta(0)}{2^{30}},$ then
\begin{eqnarray*}
\int_{\partial \Omega}S^2F(Q)d\omega_0 &=& \int_{\partial \Omega}\int_{\Gamma(Q)\cap \Omega_s}\delta(Z)^{2-n}|\nabla F(Z)|^2dZd\omega_0+ \int_{\partial \Omega}\int_{\Gamma(Q)\backslash \Omega_s}\delta(Z)^{2-n}|\nabla F(Z)|^2dZd\omega_0
\end{eqnarray*}
   $$ = \int_{\partial \Omega}\int_{\Gamma(Q)\cap \Omega_s}\delta(Z)^{-n}(\delta(Z)|\nabla F(Z)|)^2dZd\omega_0+ \int_{\Omega \backslash \Omega_s}\int_{\partial \Omega }\delta(Z)^{2-n}|\nabla F(Z)|^2\chi_{\{Z \in \Gamma(Q)\}}(Q) d\omega_0dZ $$
\begin{equation}\label{tt7.301}
   \le
 \int_{\partial \Omega}\int_{\Gamma(Q)\cap \Omega_s}\delta(Z)^{-n}(\delta(Z)|\nabla F(Z)|)^2dZd\omega_0(Q)+ \int_{\Omega \backslash \Omega_s}|\nabla F(Z)|^2 \delta(Z)^{2-n} \omega_0(\Delta(Q_Z,3\delta(Z)))dZ.
\end{equation}

Note that if $0 \in \Gamma_2(Q)$ then $B(0,s\widetilde{R}_0)\subset \Gamma_3(Q)$ and if $0 \notin \Gamma_2(Q)$ then $B(0,s\widetilde{R}_0) \cap \Gamma(Q)=\emptyset.$ Thus
\begin{eqnarray}\label{tt7.302}
\int_{\partial \Omega}\int_{\Gamma(Q)\cap \Omega_s}\delta(Z)^{-n}(\delta(Z)|\nabla F|)^2dZd\omega_0
&\le
& \int_{\partial \Omega}\int_{\Gamma(Q)\cap \Omega_s}\delta(Z)^{-n}(\delta(Z)|\nabla F|)^2\chi_{\{0 \in \Gamma_2(Q)\}}(Z)dZd\nonumber\omega_0\\
&\lesssim& \int_{\partial \Omega}\fint_{B(0,s\widetilde{R}_0)}(\delta(Z)|F(Z)|)^2\chi_{\{0 \in \Gamma_2(Q)\}}(Z)dZd\omega_0\nonumber\\
&\lesssim& \int_{\partial \Omega}\widetilde{N}
(\delta|\nabla F(Z)|)^2 (Q)d\omega_0(Q).
\end{eqnarray}
We now estimate the second term in (\ref{tt7.301}). Since $\omega_0(\Delta(Q_Z, 3\delta(Z)))\delta(Z)^{2-n} \thicksim G_0(Z)$ using the ellipticity
of $L_0$ we have
\begin{eqnarray}\label{tt.7.304}
\int_{\Omega\setminus\Omega_s}\delta(Z)^{2-n}|\nabla F(Z)|^2dZd\omega_0(Q)
&\lesssim & \int_{\Omega \backslash \Omega_s}|\nabla F(Z)|^2 G_0(Z)dZ\nonumber\\
&\lesssim & \int_{\Omega \backslash \Omega_s}\langle A_0\nabla F(Z), \nabla F(Z)\rangle G_0(Z) dZ\nonumber\\
&\lesssim& \int_{\Omega \backslash \Omega_s} {\rm{div}} (A_0\nabla F)F)G_0\, dZ-\int_{\Omega \backslash \Omega_s}{\rm{div}} ((A_0\nabla F)F G_0\, dZ\nonumber\\
&\lesssim&  \frac{1}{2}\int_{\Omega \backslash \Omega_s}L_0(F^2)G_0\, dZ - \int_{\Omega \backslash \Omega_s}(L_0 F)F G_0\, dZ.
\end{eqnarray}
Integration by parts on the second term in (\ref{tt.7.304}) yields
\begin{eqnarray}\label{tt7.305}
  \int_{\Omega \backslash \Omega_s}(L_0 F)(Z)F(Z) G_0(Z)dZ &=& - \int_{\Omega \backslash \Omega_s}{\rm{div}}(\varepsilon \nabla u_1)(F G_0)(Z)dZ\nonumber\\
                                                           &=& \int_{\Omega \backslash \Omega_s}\nabla (F G_0)\varepsilon\nabla u_1dZ\nonumber\\
                                                           &=& \int_{\Omega \backslash \Omega_s}G_0 \nabla F\varepsilon \nabla u_1 dZ + \int_{\Omega \backslash \Omega_s}\nabla G_0 F\varepsilon \nabla u_1 dZ.
\end{eqnarray}
since $G_0=0$ on $\partial \Omega$ and $\varepsilon=0$ on $\partial \Omega_s$ (recall $\varepsilon(Y)=0$ when $\delta(Y) \ge
 4R_0$).
We use the dyadic decomposition of $\partial \Omega$ to estimate each term. Recall that for $Y \in I^k_{\alpha}$,
$G_0(Y) \sim \frac{\omega_0(Q^k_{\alpha})}{({\rm{diam}}\,
 Q^k_{\alpha})^{n-2}}$ then given $\varepsilon>0$
\begin{eqnarray}\label{tt7.306A}
  \int_{\Omega \backslash \Omega_s \backslash (\partial \Omega,\varepsilon)}|\nabla F|G_0 |\varepsilon(Z)| |\nabla u_1|dZ
   &\le
 &  \sum_{\substack{Q^k_{\alpha}\subset \partial \Omega\\ k \le
 k_{\varepsilon}}} \sup_{I^k_{\alpha}} |\varepsilon(Z)| \int_{I^k_{\alpha}}G_0|\nabla F||\nabla u_1|dZ\\
   &\lesssim & \sum_{\substack{Q^k_{\alpha}\subset \partial \Omega\\ k \le
 k_{\varepsilon}}}  \sup_{I^k_{\alpha}} |\varepsilon(Z)| \frac{\omega_0(Q^k_{\alpha})}{({\rm{diam}}\,
 Q^k_{\alpha})^{n-2}}\bigg(\int_{I^k_{\alpha}}|\nabla F||\nabla u_1|\, dZ\bigg)\nonumber
 \end{eqnarray}
 To estimate $\int_{I^k_{\alpha}}|\nabla F||\nabla u_1|\, dZ$ cover $I^k_{\alpha}$ by balls $\{B(X_i,\lambda 8^{-k-3} )\}_{1 \le
 i\le
 N}$ with $X_i \in I^k_{\alpha}$ such that $|X_i-X_l| \ge
 \lambda 8^{-k-3}/2$. Here $N$ is independent of $k$ and the balls $B(X_i,\lambda 8^{-k-3} )$ have finite overlap (also independent of $k$).
 \begin{eqnarray}\label{tt7.306B}
 \int_{I^k_{\alpha}}|\nabla F||\nabla u_1|\, dZ&\le &\sum_{i=1}^N\int_{B(X_i,\lambda 8^{-k-3}) }|\nabla F| |\nabla u_1|\, dZ\\
 &\le &\sum_{i=1}^N\left(\int_{B(X_i,\lambda 8^{-k-3})}|\nabla F|^2\, dZ\right)^{1/2}\left(\int_{B(X_i,\lambda 8^{-k-3} )}|\nabla u_1|^2\, dZ\right)^{1/2}\nonumber\\
 &\le &\sum_{i=1}^N\left( \int_{B(X_i,\lambda 8^{-k-3}) }|\nabla u_1|^2\left(\int_{B(Z,\lambda 8^{-k-2}) }|\nabla F|^2\, dY\right)
 \, dZ\right)^{1/2}\nonumber\\
 &\lesssim &({\rm{diam}}\, Q^k_{\alpha})^{\frac{n-2}{2}}\sum_{i=1}^N\left( \int_{B(X_i,\lambda 8^{-k-3}) }|\nabla u_1|^2\left(\fint_{B(Z,\frac{\delta(Z)}{8} }(\delta(Y)|\nabla F|)^2\, dY\right)
 \, dZ\right)^{1/2}\nonumber\\
 &\lesssim & ({\rm{diam}}\, Q^k_{\alpha})^{\frac{n-2}{2}}\left(\sum_{i=1}^N  \int_{B(X_i,\lambda 8^{-k-3}) }|\nabla u_1|^2\left(\fint_{B(Z,\frac{\delta(Z)}{8} }(\delta(Y)|\nabla F|)^2\, dY\right)
 \, dZ \right)^{1/2}\nonumber\\
 &\lesssim&\left(({\rm{diam}}\, Q^k_{\alpha})^{n-2}\int_{I^k_{\alpha}}|\nabla u_1|^2\left(\fint_{B(Z,\frac{\delta(Z)}{8} }(\delta(Y)|\nabla F|)^2\, dY\right)\,dZ.\right)^{1/2}\nonumber
 \end{eqnarray}

 Combining (\ref{tt7.306A}) and (\ref{tt7.306B}) we have
 \begin{eqnarray}\label{tt7.306}
&& \int_{\Omega \backslash \Omega_s \backslash (\partial \Omega,\varepsilon)}|\nabla F|G_0 |\varepsilon(Z)| |\nabla u_1|dZ\\
   &\lesssim
 &  \sum_{\substack{Q^k_{\alpha}\subset \partial \Omega\\ k \le
 k_{\varepsilon}}} \sup_{I^k_{\alpha}} |\varepsilon(Z)|\omega_0(Q^k_{\alpha})
  \left(\int_{I^k_{\alpha}}|\nabla u_1|^2\delta(Z)^{2-n}\left(\fint_{B(Z,\frac{\delta(Z)}{8} }(\delta(Y)|\nabla F|)^2\, dY\right)\, dZ\right)^{1/2}\nonumber\\
   &\lesssim& \sum_{\substack{Q^k_{\alpha}\subset \partial \Omega\\ k \le
 k_{\varepsilon}}}   \left(\int_{I^k_{\alpha}}\frac{a^2(Y)G_0(Y)}{\delta(Y)^2}dY\right)^{1/2}\omega_0(Q^k_{\alpha})^{1/2}
  \left(\int_{I^k_{\alpha}}|\nabla u_1|^2\delta(Z)^{2-n}\left(\fint_{B(Z,\frac{\delta(Z)}{8} }(\delta(Y)|\nabla F|)^2\, dY\right)\, dZ\right)^{1/2}.\nonumber
\end{eqnarray}

Applying a stopping time argument similar to the one used in the proof of Lemma \ref{lem2.9} to estimate $F^\varepsilon_0
$, to the function
$$\widetilde{T}_{\varepsilon}(Q)=\left(\int_{\Gamma_M(Q)\backslash B_{2 \varepsilon}(Q)} |\nabla u_1|^2\delta(Z)^{2-n}\left(\fint_{B(Z, \delta(Z)/8)}\delta ^2|\nabla F|^2\right)dZ\right)^{1/2}$$
and letting $\varepsilon$ tend to 0 we obtain
\begin{eqnarray}\label{tt7.309}
 \int_{\Omega \backslash \Omega_s}|\nabla F|G_0|\varepsilon||\nabla u_1|  &\lesssim& \varepsilon_0
 \int_{\partial \Omega} \left(\int_{\Gamma_M(Q)}|\nabla u_1|^2 \delta(Z)^{2-n} \fint_{B(Z, \delta(Z)/8)}\delta ^2|\nabla F|^2 dZ \right)^{1/2} d\omega_0(Q) \nonumber\\
   &\lesssim&   \varepsilon_0
 \int_{\partial \Omega} \widetilde{N}
^M (\delta |\nabla F|)(Q) S_M (u_1)(Q)d\omega_0(Q).
\end{eqnarray}

Now we turn our attention to the second term in (\ref{tt7.305}). Applying Cacciopoli's (see (\ref{tt7.204})) we have
\begin{eqnarray}\label{tt7.310}
  \int_{\Omega \backslash \Omega_s}|\nabla G_0||F||\varepsilon||\nabla u_1|dZ &\le
& \sum_{Q^k_{\alpha}\subset \partial \Omega} \sup_{I^k_{\alpha}}|\varepsilon|
\left(\int_{I^k_{\alpha}}|\nabla G_0|^2 dZ\right)^{1/2}
\left(\int_{I^k_{\alpha}}|\nabla u_1(Z)|^2 F^2(Z)dZ\right)^{1/2} \nonumber\\
   &\lesssim&  \sum_{Q^k_{\alpha}\subset \partial \Omega} \sup_{I^k_{\alpha}}|\varepsilon|
\left(\int_{\widehat{I}^k_{\alpha}}\frac{G_0(Y)^2}{\delta(Y)^2} dY\right)^{1/2}
\left(\int_{I^k_{\alpha}}|\nabla u_1(Z)|^2 F^2(Z)dZ\right)^{1/2}.\nonumber\\
\end{eqnarray}
Once again a similar argument to the one that appears in the proof of Lemma \ref{lem2.9} with a stopping time argument applied to a truncation of
$\left(\int_{\Gamma_M(Q)}|\nabla u_1|^2 \delta(Y)^{2-n} F(Y)dY\right)^{1/2}$ yields the following estimate
\begin{eqnarray}\label{tt7.311}
\int_{\Omega \backslash \Omega_s}|\nabla G_0||F||\varepsilon||\nabla u_1|dZ
&\lesssim & \varepsilon_0
 \int_{\partial \Omega} \left(\int_{\Gamma_M(Q)}|\nabla u_1|^2 \delta(Z)^{2-n} F(Y)dY \right)^{1/2} d\omega_0(Q)\nonumber\\
&\lesssim&\varepsilon_0
 \int_{\partial \Omega} S_M (u_1)(Q)N_M F(Q)d\omega_0(Q).
\end{eqnarray}
Putting together (\ref{tt7.305}), (\ref{tt7.309}) and (\ref{tt7.311}) we obtain
\begin{equation}\label{tt7.311A}
\bigg|\int_{\Omega \backslash \Omega_s}L_0 F\cdot FG_0dZ\bigg|
\lesssim \varepsilon_0
 \int_{\partial \Omega} \widetilde{N}
_M (\delta |\nabla F|)(Q) S_M (u_1)(Q)d\omega_0+
         \varepsilon_0
 \int_{\partial \Omega} N_MF(Q) S_M (u_1)(Q)  d\omega_0.
\end{equation}
To estimate the first term in (\ref{tt.7.304}) observe that
\begin{eqnarray}\label{tt7.312}
  \frac{1}{2}\int_{\Omega \backslash \Omega_s} G_0 L_0(F^2)dZ
&=&\frac{1}{2}\int_{\Omega \backslash \Omega_s}{\rm{div}}(G_0 A_0 \nabla F^2)dZ -\frac{1}{2}\int_{\Omega \backslash \Omega_s} A_0 \nabla G_0 \nabla F^2(Z)dZ \nonumber\\
&=& \int_{\partial \Omega_s}G_0 A_0 F \nabla F \cdot \overrightarrow{\nu}d\sigma -\frac{1}{2}\int_{\Omega \backslash \Omega_s} {\rm{div}}(F^2 A_0 \nabla G_0)dZ \nonumber\\
&=&\int_{\partial \Omega_s}G_0 A_0 F \nabla F \cdot \overrightarrow{\nu}d\sigma -\frac{1}{2}\int_{\partial \Omega_s} F^2 A_0 \nabla G_0 \cdot \overrightarrow{\nu} d\sigma.
\end{eqnarray}
Integrating over $s \in [1,2]$ we obtain
$$\frac{1}{2}\int_{1}^{2}\bigg|\int_{\Omega \backslash B(0,s\widetilde{R}_0)} G_0 L_0F^2dZ \bigg|ds =
\int_{1}^{2}\int_{\partial B(0,s\widetilde{R}_0)} |G_0|| A_0|| F||\nabla F| \,d\sigma ds$$
$$+\frac{1}{2}\int_{1}^{2}\int_{\partial B(0,s\widetilde{R}_0)} F^2 |A_0| |\nabla G_0|\, d\sigma ds $$
\begin{equation}\label{tt7.313}
 = \int_{B(0,2\widetilde{R}_0)\backslash B(0,\widetilde{R}_0)} G_0| A_0|| F||\nabla F|dZ -
\frac{1}{2}\int_{B(0,2\widetilde{R}_0)\backslash B(0,\widetilde{R}_0)} F^2 |A_0| |\nabla G_0| dZ.
\end{equation}
Looking at each term in (\ref{tt7.313}) separately we have that
\begin{equation}\label{tt7.314}
 \int_{B(0,2\widetilde{R}_0)\backslash B(0,\widetilde{R}_0)} G_0| A_0|| F||\nabla F|dZ
 \lesssim \sum_{\substack{Q^k_{\alpha}\subset \partial \Omega\\ I^k_{\alpha}\cap \Omega_2 \backslash \Omega_1 \neq \emptyset}} \int_{I^k_{\alpha}\cap (\Omega_2 \backslash \Omega_1)} G_0| F||\nabla F|dZ.
\end{equation}
Note that if $I^k_{\alpha}\cap \Omega_2 \backslash \Omega_1 \neq \emptyset$ there is $Y \in \Omega$ so that $\lambda/8 <8^k\delta(Y)<8\lambda$
and $|\delta(Y)-\delta(0)|\lneqq 2\widetilde{R}_0=\delta(0)2^{-29}.$ Thus
$(1+2^{-29})^{-1}\lambda/8 <8^k\delta(0)<(1-2^{-29})^{-1}8\lambda.$ Since ${\rm{diam}}\,
 Q^k_{\alpha}\thicksim 8^{-k} \thicksim \delta(0)$ then
$\omega(Q^k_{\alpha}) \ge
 C_1$ an absolute constant only depending on the NTA constants of $\Omega.$ Thus for $Y \in I^k_{\alpha}\cap (\Omega_2 \backslash \Omega_1) $ $G_0(Y) \lesssim \frac{1}{\delta(0)^{n-2}}\lesssim \frac{\omega_0(Q^k_{\alpha})}{({\rm{diam}}\,
 Q^k_{\alpha})^{n-2}}.$ 
 These combined with a computation like the one that appears in (\ref{tt7.306A}), (\ref{tt7.306B}) and (\ref{tt7.306}) yields
\begin{eqnarray}\label{tt7.315}
 \int_{\Omega_2 \backslash \Omega_1 } G_0| A_0|| F||\nabla F|dZ
   &\lesssim & \sum_{\substack{Q^k_{\alpha}\subset \partial \Omega\\ I^k_{\alpha}\cap \Omega_2 \backslash \Omega_1 \neq \emptyset}}
               \left(\int_{I^k_{\alpha}}| F||\nabla F|dZ \right)\frac{\omega_0(Q^k_{\alpha})}{({\rm{diam}}\,
 Q^k_{\alpha})^{n-2}}\nonumber \\
   &\lesssim&  \sum_{\substack{Q^k_{\alpha}\subset \partial \Omega\\ I^k_{\alpha}\cap \Omega_2 \backslash \Omega_1 \neq \emptyset}} \omega_0(Q^k_{\alpha})^{1/2}
               \left(\omega_0(Q^k_{\alpha})\int_{I^k_{\alpha}}\delta^{2-n}| F|^2 dZ \fint_{B(Z,\frac{\delta(Z)}{8})} \delta^2|\nabla F|^2dZ \right)^{1/2}\nonumber\\
   &\lesssim& \int_{\partial \Omega}\left(\int_{\Gamma_M(Q)\cap \Omega_2 \backslash \Omega_1}\delta^{2-n}(Z)|F|^2 \left(\fint_{B(Z,\frac{\delta(Z)}{8})}\delta^2|\nabla F|^2dY\right)dZ\right)^{1/2}d\omega_0\nonumber \\
   &\lesssim&  \delta(0)^{(2-n)/2}\int_{\partial \Omega}N_M F(Q) \widetilde{N}^M(\delta |\nabla F|)(Q)\delta(0)^{n/2}d\omega_0\nonumber\\
&\lesssim& \delta(0) \int_{\partial \Omega}N_M F(Q) \widetilde{N}
^M(\delta |\nabla F|)(Q)d\omega_0(Q).
\end{eqnarray}
We control the second term in (\ref{tt7.313}) by recalling that if $Y \in I^k_{\alpha}\cap \Omega_2 \backslash \Omega_1$ then $\delta(0)\thicksim {\rm{diam}}\,
 Q^k_{\alpha}$ and $|\nabla G_0(Y)|\lesssim \frac{G_0(Y)}{\delta(Y)}\lesssim \frac{\omega_0(Q^k_{\alpha})}{\delta(0)^{n-1}}.$ Thus as in (\ref{tt7.315}) and using the doubling properties of $\o_0$ we have
\begin{eqnarray}\label{tt7.316}
\int_{\Omega_2 \backslash \Omega_1 } F^2| A_0||\nabla  G_0|dZ
&\lesssim& \sum_{\substack{Q^k_{\alpha}\subset \partial \Omega\\ I^k_{\alpha}\cap \Omega_2 \backslash \Omega_1 \neq \emptyset}} \int_{I^k_{\alpha}\cap \Omega_2 \backslash \Omega_1} F^2|\nabla  G_0|dZ\nonumber\\
&\lesssim&  \sum_{\substack{Q^k_{\alpha}\subset \partial \Omega\\ I^k_{\alpha}\cap \Omega_2 \backslash \Omega_1 \neq \emptyset}} \omega_0(Q^k_{\alpha})^{1/2}
            \bigg(\omega_0(Q^k_{\alpha})\int_{I^k_{\alpha}} \delta^{2-n}|F|^2 \fint_{B(Z,\frac{\delta(Z)}{8})}|F|^2\bigg)^{1/2}\nonumber\\
      &\lesssim&  \delta(0) \int_{\partial \Omega}N_M F(Q) \widetilde{N}
^M F(Q)d\omega_0(Q).      
\end{eqnarray}
Combining (\ref{tt7.313}), (\ref{tt7.315}) and (\ref{tt7.316}) we obtain
$$ \int_{1}^2 \bigg| \int_{\Omega\backslash B(0,s\widetilde{R}_0)}G_0(Z)L_0F^2(Z)dZ\bigg|ds $$
\begin{equation}\label{tt7.317}
\lesssim \delta(0) \int_{\partial \Omega}N_M F(Q) \widetilde{N}^M(\delta |\nabla F|)(Q)d\omega_0(Q)+ \\
         \delta(0) \int_{\partial \Omega}N_M F(Q) \widetilde{N}
^M F(Q)d\omega_0(Q).
\end{equation}
Combining (\ref{tt7.301}), (\ref{tt7.302}), (\ref{tt.7.304}),  (\ref{tt7.311}) and (\ref{tt7.317}) plus recalling the fact that $S_M (u_1) \le
 S_M (F)+S_M (u_0)$ and $||S_M F||_{L^2(\omega_0)} \thicksim ||SF||_{L^2(\omega_0)}$ and Remark \ref{N-hat-tilde}
we obtain

$$\int_{\partial \Omega}S^2F(Q)d\omega_0(Q)  = \int_{1}^2  \int_{\partial \Omega} S^2F(Q)d\omega_0(Q) ds $$
 \begin{eqnarray}\label{tt7.318}
 &=& \int_{1}^2  \int_{\partial \Omega} \int_{\Gamma(Q)\cap \Omega_s}\delta^{-n}(Z)(\delta(Z)|\nabla F|)^2 dZ d\omega_0 ds+ \int_{1}^2  \int_{\Omega\backslash \Omega_s}|\nabla F|^2\delta^{2-n}(Z)\omega_0(\Delta(Q_Z, 3\delta(Z))) dZ ds\nonumber\\
 &\lesssim& \int_{\partial \Omega} \int_{\Gamma(Q)\cap \Omega_s}\delta^{-n}(Z)(\delta(Z)|\nabla F|)^2 dZ d\omega_0
  + \int_{1}^2 \bigg| \int_{\Omega\backslash \Omega_s}L_0 F^2 G_0\bigg|ds + \int_{1}^2 \bigg| \int_{\Omega\backslash \Omega_s}(L_0 F)F G_0\bigg|ds\nonumber\\
 &\lesssim&  \int_{\partial \Omega} \widetilde{N}
 (\delta |\nabla F|)^2(Q)d\omega_0(Q)+ \varepsilon_0
 \int_{\partial \Omega}\widetilde{N}
^M(\delta |\nabla F|)(Q)S_M (u_1)(Q) d\omega_0(Q)\nonumber\\
&+&\varepsilon_0
 \int_{\partial \Omega}S_M (u_1)(Q) N_M F(Q) d\omega_0+\int_{\partial \Omega}N_M F(Q)\widetilde{N}
^M (\delta |\nabla F|) d\omega_0+ \int_{\partial \Omega}N_M F(Q) \widetilde{N}
^M (F)(Q) d\omega_0\nonumber\\
&\lesssim& ||\widetilde{N}(\delta |\nabla F|)||_{L^2(\omega_0)}^2 + || S u_0||_{L^2(\omega_0)}^2 + || NF||_{L^2(\omega_0)}^2 +||\widetilde{N}
(F)||_{L^2(\omega_0)}^2 + \varepsilon_0
 || S F||_{L^2(\omega_0)}^2.
\end{eqnarray}

Since by Lemma \ref{1510} $||S u_0||_{L^2(\omega_0)}^2 \lesssim || N u_0||_{L^2(\omega_0)}^2 \lesssim || f||_{L^2(\omega_0)}^2$ we obtain from (\ref{tt7.318})
\begin{equation}\label{tt7.319}
|| S F||_{L^2(\omega_0)}^2 \lesssim ||\widetilde{N}(\delta |\nabla F|)||_{L^2(\omega_0)}^2 + || NF||_{L^2(\omega_0)}^2 +||\widetilde{N}
(F)||_{L^2(\omega_0)}^2 + ||f||_{L^2(\omega_0)}^2
\end{equation}
which yields Lemma \ref{lem2.10}.

\emph{Proof of Theorem \ref{mainthm1}:} Since $S(u_1) \le
 S(F) + S(u_0),$ (\ref{ttlem7.6C}), (\ref{ttlem7.8}) and the argument above, (\ref{tt7.319})
yields
\begin{eqnarray*}
  \int_{\partial \Omega}\widetilde{N}F(Q)^2 d\omega_0 + \int_{\partial \Omega}\widetilde{N}(\delta |\nabla F|)^2(Q)d\omega_0 &\le
& C \varepsilon_0
^2 \int_{\partial \Omega} S^2 u_1 d\omega_0\\
  &\le
& C \varepsilon_0
^2 \int_{\partial \Omega} (S F)^2(Q)d\omega_0 + C \varepsilon_0
^2 \int_{\partial \Omega} (Su_o)^2 d\omega_0
\end{eqnarray*}
\begin{equation}\label{tt7.320}
\le
 C \varepsilon_0
^2 \left(\int_{\partial \Omega} \widetilde{N}F(Q)^2d\omega_0 + \int_{\partial \Omega}\widetilde{N}(\delta |\nabla F|)^2(Q)d\omega_0\right)
     + C \varepsilon_0
^2 \int_{\partial \Omega} N F(Q)^2d\omega_0 + C \varepsilon_0
^2 \int_{\partial \Omega} f^2 d\omega_0.
\end{equation}
Thus
\begin{equation}\label{tt7.321}
\int_{\partial \Omega} \widetilde{N}F(Q)^2d\omega_0 + \int_{\partial \Omega}\widetilde{N}(\delta |\nabla F|)^2(Q)d\omega_0 \le
C \varepsilon_0
^2 \int_{\partial \Omega} N F(Q)^2 d\omega_0 + C \varepsilon_0
^2 \int_{\partial \Omega} f^2 d\omega_0.
\end{equation}
Note that since $||\widetilde{N}u_i||_{L^2(\omega_0)} \thicksim ||N u_i||_{L^2(\omega_0)}^2,\ i=0,1$ then by (\ref{tt7.321})
\begin{eqnarray}\label{tt7.322}
  \int_{\partial \Omega} N F(Q)^2d\omega_0  &\le
 & \int_{\partial \Omega} N u_1^2(Q)d\omega_0 + \int_{\partial \Omega} N u_0^2(Q) d\omega_0\nonumber\\
  &\le
& C\int_{\partial \Omega} \widetilde{N}u_1^2(Q)d\omega_0 + C\int_{\partial \Omega} f^2 d\omega_0\nonumber\\
  &\le
& C\int_{\partial \Omega} \widetilde{N}F(Q)^2d\omega_0 + C\int_{\partial \Omega} f^2 d\omega_0\nonumber\\
  &\le
& C \varepsilon_0
^2 \int_{\partial \Omega} N F(Q)^2d\omega_0 + C \int_{\partial \Omega} f^2 d\omega_0
\end{eqnarray}
which ensures that
\begin{equation}\label{tt7.323}
\int_{\partial \Omega} N F(Q)^2d\omega_0 \le
 C\int_{\partial \Omega} f^2 d\omega_0
\end{equation}
which yields

\begin{eqnarray}\label{tt7.324}
  \int_{\partial \Omega} Nu_1^2(Q)d\omega_0 &\le
 & C\int_{\partial \Omega} NF(Q)^2d\omega_0 + C \int_{\partial \Omega} N u_0^2(Q)d\omega_0 \nonumber\\
   &\le
&  C\int_{\partial \Omega} f^2 d\omega_0. 
\end{eqnarray}
This concludes the proof of Theorem \ref{mainthm1}.

\end{proof}

\section{Regularity for the elliptic kernel on CADs}
\begin{thm}\label{thm8.1}
Let $\O$ be a CAD---assume there exists a constant $C>0$ such that
\begin{equation}\label{eqn8.1}
\sup_{\D\subset\po}\left(\frac{1}{\s(\D)} \int_{T(\D)}
\frac{a^2(X)}{\d(X)^n}dX\right)^{\frac{1}{2}}<C
\end{equation}
Then $\o_1\in A_\infty(\s)$ if $\o_0\in A_\infty(\s)$. 
\end{thm}

The argument used to
prove the result above is similar to the one used in \cite{fkp} to prove Theorem
2.3. In our case it relies on a generalization of Fefferman's result to the CAD
setting, as follows.

\begin{thm}\label{thm8.2}
Let $\O$ be a CAD and let
\begin{equation}\label{eqn8.2}
A(a)(Q) = \left(\int_{\G(Q)} \frac{a^2(X)}{\d(X)^n} dX\right)^{\frac{1}{2}}.
\end{equation}

If $\|A(a)\|_{L^\infty(\s)}\le C_0<\infty$ and $\o_0\in A_\infty(\s)$ then
$\o_1\in A_\infty(\s)$.
\end{thm}

\noindent{\bf Proof of Theorem \ref{thm8.2}}

This is a corollary of Theorem 2.9. In fact note that for $\D=B(Q_0,r)\cap\po$
with $Q_0\in\po$ the fact that $\frac{\o_0(\D(Q_X,\d(X))}{\d(X)^{n-1}}\sim
\frac{G_0(X)}{\d(X)}$ combined with Fubini's theorem and the doubling properties
of $\o_0$ yields
\begin{eqnarray}\label{eqn8.3}
\frac{1}{\o_0(\D)} \int_{T(\D)}\frac{a^2(X)G_G(X)}{\d(X)^2}
dX & \mathop{<}\limits_{\sim} & \frac{1}{\o_0(\D)}\int_{T(\D)}\frac{a^2(X)}{\d(X)}
\frac{\o_0(\D(Q_X\d(X))}{\d(X)^{n-1}}dX\nonumber \\
& = & \frac{1}{\o_0(\D)} \int_{T(\D)} \int_{\po} \frac{a^2(X)}{\d(X)^n} \chi_{\D(Q_X,\d(X))} (Q)d\o dX \\
& \le & \frac{1}{\o_0(\D)} \int_{3\D}\left(\int_{T(\D)} \frac{a^2(X)}{d(X)^n} \chi_{\G(Q)}(X)dX\right) d\o_0 \nonumber \\
&  \le & \frac{1}{\o_0(\D)} \int_{3\D} A^2(a)(Q)d\o_0(Q)\nonumber \\
& \le & \frac{1}{\o_0(\D)} \int_{3\D} A^2(a)(Q)d\o(Q) \nonumber
\end{eqnarray}
Hence there exists $\d>0$ depending on $n$ and the NTA constants of $\O$ such that if $\|A(a)\|_{L^\infty(\s)}\le \d$, and $\o_0\in A_\infty(\s)$ then $\o_1\in B_2(\o_0)$. Since $\o_0\in A_\infty(\s)$ the fact that $\|A(a)\|_{L^\infty(\s)}\le\d$ implies that $\|A(a)\|_{L^\infty(\o_0)}\le\d$. Estimate (\ref{eqn8.3}) guarantees that there exists $C>0$ depending on $n$ and the NTA constants of $\O$ such that
\begin{equation}\label{eqn8.4}
\sup_{\D C\po} \left(\frac{1}{\o_0(\D)} \int_{T(\D)}a^2(X)\frac{G_0(X)}{\d^2(X)}dX\right)^{\frac{1}{2}} \le C\d^{\frac{1}{2}} 
\end{equation}
choosing $\d>0$ small enough so that $C\d^{\frac{1}{2}}<\epsilon_0$ in Theorem 2.9 we conclude that $\o_1\in B_2(\o_0)$. To finish the proof of Theorem \ref{thm8.2} consider the family of operators $L_t=(1-t)A_0+tA_1$ for $0\le t\le 0$. Consider a partition of $[0,1]$ $\{t_i\}^m_{i=0}$ such that $0<t_{i+1}-t_i<\frac{\d}{C_0}$ where $C_0$ is as in the statement of Theorem \ref{thm8.2}. Let $a_i$ be the deviation function corresponding to $L_{t_{i+1}}=L_{i+1}$ and $L_{t_i}=L_i$, here $\epsilon_i(X) = A_{t_{i+1}}(X)=(t_{i+1}-t_i)\epsilon(X)$, and $a_i(X)=(t_{i+1}-t_i)a(X)$. Hence $\|A(a_k)\|_{L^\infty(\s)}=(t_{i-1}-t_i)\|A(a)\|_{L^\infty(\s)}<\d$. An iteration of the argument above ensures that for $i\in \{0,\ldots, m\}$ $\o_i\in A_\infty(\s)$ and $\o_{i+1}\in B_2(\o)$. Hence $\o_1\in A_\infty(\s)$.\hfill\qed
\medskip

\noindent{\bf Proof of Theorem \ref{thm8.1}}: We are assuming the Carleson condition (\ref{eqn8.1}) on $\frac{a^2(X)}{d(X)}dX$ and that $\o_0\in A\infty(\s)$. We will show that $\o_1\in A_\infty(\s)$ by showing that there exists $0<\a<1$ and $0<\b<1$ such that if $\D=B(Q_0,r)\cap\po$ and $E\subset\D$ then $\frac{\s(E)}{\s(\D)}>\a$ implies that $\frac{\o_1(E)}{\o_1(\D)}>\b$.

For $r>0$ and $\gamma>0$ we denote by $\G_{\gamma,r}(Q)$ the truncated cone of radius $r$ and aperture $\gamma$, i.e. $\G_{\gamma,r}(Q)=\{X\in\O:|X-Q|<(1+\gamma)\d(X), 0<\partial(X)<r <\d(X)<r\}$. We define the truncated square function with aperture determined by $\gamma$ for the deviation function $a(X)$ by
\[
A_{\gamma,r}(Q)=\left(\int_{\G_{\gamma,r}(Q)} \frac{a^2(X)}{\d(X)^n}dX\right).
\]
The appropriate constant $\gamma$ will be chosen later. 

Applying Lemma 3.13 to $A(X)=\frac{a^2(X)}{\d(X)^j} \chi B(Q_0; (2+\gamma)r)(X)$ we conclude that 
\[
\frac{1}{\s(\D)}\int_\D A^2_{\gamma,r}(Q)d\s(Q)  \le  \frac{1}{\s(\D)}\int_{T((2+\gamma)\D)} \frac{a^2(X)}{\d(X)}dX 
 \le  C_{\gamma}
\]
because $\s$ is doubling and hypothesis (\ref{eqn8.1}). Thus there is a closed set $S\subset\D$ so that $\frac{\s(S)}{\s(\D)}\ge \frac{1}{2}$ and $A_{\gamma,r}(Q)\le C'_\gamma$ for $Q\in S$

Recall that there exist constants $0<\b<\gamma$ and $C_1<C_2<0$ and a sawtooth domain $\O_S$ such that
\begin{itemize}
\item[(i)]$\bigcup_{Q\in S}\G_{\b,C_1 r}(Q)\subset\O_S\subset\bigcup_{Q\in S}\G_{\gamma,C_2 r}(Q)$
\item[(ii)]$\po_S\cap\po=S$
\item[(iii)]The NTA character of $\O_S$ is independent of $S$.
\end{itemize}
(See \cite{k1}).

Without loss of generality we may assume that $\frac{3}{2}\beta+\frac{1}{2}<\gamma$. Let $\O'=\bigcup_{Q\in S}\G_{\b,C_1r}(Q)$ and $\widetilde{\O}=\bigcup_{Q\in S} \G_{\gamma,C_2r}(Q)$. For $X\in\O$ with $\d(X)<C_1r$ if $B\left(X,\frac{\d(X)}{2}\right)\cap\O'\ne\emptyset$ then there exists $\widetilde Q\in S$ so that $B\left(X, \frac{\d(X)}{2}\right)\subset\G_{\gamma, C_2r}(\widetilde Q)$. In fact if $Y\in B\left(X, \frac{\d(X)}{2}\right)\cap\O'$ there is $\widetilde Q\in S$ so sthat $|\widetilde Q-Y|<(1+\b)\d(Y)$ and
$|\widetilde Q-X|\le |\widetilde Q-Y|+|Y-X|<(1+\b)\d(Y)\in \frac{\d(X)}{2}\le (1+\b)\frac{3\d(X)}{2} + \frac{\d(X)}{2}=\d(X)\left(1+\left(\frac{3}{2}\b+\frac{1}{2}\right)\right)$. Thus $X\in\G_{\gamma,r}(\widetilde Q)$. Define the operator $\widetilde L_{1}=\div\widetilde A_1\nabla$ by setting
\[
\widetilde{A}_1(X) = \left\{\begin{array}{l}
A_1(X)\mbox{ if }X\in\O'\\
A_0(X)\mbox{ if }X\in\O\backslash\O'\end{array}\right.
\]
Let $\widetilde a(X)=\sup_{Y\in B\left(X,\frac{\d(X)}{2}\right)}|\widetilde A_1(Y) - A_0(Y)|$ be the deviation function for $\widetilde L_1$ and $L_0$. Observe that $\widetilde a(X)\le a(X)$. For $\gamma$ and $\b$ as above consider 
\[
\widetilde A_\b(Q) = \left(\int_{\G_\b(Q)}\frac{\widetilde a(X)^2}{\d(X)^n}dX\right)^{\frac{1}{2}}.
\]
Note that by the definition of $\widetilde A_1$ and $\widetilde a$, 
\[
\widetilde A_\b(Q)=\left(\int_{\G_{\b,C_1r}(Q)}\frac{\widetilde a(X)^2}{\d(X)^n}dX\right)^{\frac{1}{2}}.
\]
If $B\left(X,\frac{\d(X)}{2}\right)\cap\O'=\emptyset$ then $\widetilde a(X)=0$. On the other hand if $X\in \G_{\b,C_1r}(Q)=0$ and $B\left(X,\frac{\d(X)}{2}\right)\cap\O'\ne\emptyset$ then there exists $\widetilde Q\in S$ so that $B\left(X,\frac{\d(X)}{2}\right)\subset\G_{\gamma,C_2r}(\widetilde Q)$.

Thus $\widetilde A_\b(Q)\le A_{\gamma,r}(\widetilde Q)\le C_{\gamma'}$. By Theorem \ref{thm8.2} $\widetilde \o_1=\o_{\widetilde L_1}\in A_\infty(\s)$. Choose $0<\a<1$ close to $1$ so that $\frac{\s(E\cap S)}{\s(\D)}\ge \frac{1}{4}$ whenever $\frac{\s(E)}{\s(\D)}>\a$. Let $F=S\cap E$ since $\widetilde \o_1\in A_\infty(\s)$ there exists constants $C>0$ and $\eta>0$ so that
\begin{equation}\label{eqn8.5}
\frac{\widetilde\o_1(F)}{\widetilde\o_1(\D)}>C\left(\frac{\s(F)}{\s(\D)}\right)^\eta\ge C'.
\end{equation}

By \cite{DJK} and \cite{JK} there exist constants $C>0$ and $Q>0$ depending on the ellipticity constants, the NTA constants of $\O$ and $n$ so that for $F\subset S$
\begin{equation}\label{eqn8.6}
\frac{1}{C}\left(\widetilde \o_1^{\O_S}(F)\right)^{\frac{1}{\theta}} \le \frac{\widetilde\o_1(F)}{\widetilde\o_1(\D)}\le C\left(\widetilde \o_1^{\O_S}(F)\right)^\theta
\end{equation}
and
\begin{equation}\label{eqn8.7}
\frac{1}{C} \left(\o_1^{\O_S}(F)\right)^{\frac{1}{\theta}} \le \frac{\o_1(F)}{\o_1(\D)} \le C\left(\o_1^{\O_S}(F)\right)^\theta
\end{equation}
(see Lemma 1.4.14 in \cite{k1}). Since $\O_S\subset\O'$, $L_1=L_1$ on $\O_S$ and $\widetilde\o_1^{\O_S}=\widetilde\o_1$. Combining (\ref{eqn8.5}), (\ref{eqn8.6}) and (\ref{eqn8.7}) we obtain since $F\subset E$
\begin{eqnarray}\label{eqn8.8}
\frac{\o_1(E)}{\o_1(\D)} \ge \frac{\o_1(F)}{\o_1(\D)} & \ge & \frac{1}{C} \left(\o_1^{\O_S}(F)\right)^{\frac{1}{\theta}} = \frac{1}{C}\left(\widetilde\o_1^{\O_S}(F)\right)^{\frac{1}{\theta}} \\
& \ge & C'\left(\frac{\widetilde\o_1(F)}{\widetilde\o_1(\D)}\right)^{\frac{1}{\theta^2}} \ge C' \nonumber
\end{eqnarray}
\vskip-.4in\hfill\qed
\medskip

\noindent{\bf Remark 1}:
Recall that by the work of David \& Jerison \cite{DJ} and Semmes \cite{S}, we have that if $\O$ is a CAD and $\o$ denotes the harmonic measure then $\o\in A_\infty(\s)$. Theorem \ref{thm8.1} shows that the elliptic measure of operators which are perturbations of the Laplacian in the sense of (1) is also in $A_\infty(\s)$.

\section{Acknowledgments} T. Toro was partially supported by NSF DMS grants 0600915 and 0856687. E. Milakis was supported by Marie Curie International Reintegration Grant No 256481 within the 7th European
Community Framework Programme and NSF DMS grant 0856687. J. Pipher was partially supported by NSF DMS grant 0901139.

\hspace{3em}
\\

\begin{tabular}{l}
Emmanouil Milakis\\ University of Cyprus \\ Department of Mathematics \& Statistics \\ P.O. Box 20537\\
Nicosia, CY- 1678 CYPRUS
\\ {\small \tt emilakis@ucy.ac.cy}
\end{tabular}
\begin{tabular}{l}
Jill Pipher\\ Brown University \\ Mathematics Department\\ Box 1917 \\
Providence, RI 02912  USA
\\ {\small \tt jpipher@math.brown.edu}
\hfill
\end{tabular}
\begin{tabular}{l}
Tatiana Toro \\ University of Washington \\ Department of Mathematics \\ Box 354350 \\
Seattle, WA 98195-4350 USA
\\ {\small \tt toro@math.washington.edu}\\
\end{tabular}

\end{document}